\documentclass[11pt,a4paper]{amsart}
\usepackage[utf8]{inputenc}
\usepackage{amsmath}
\usepackage{amsfonts}
\usepackage{amssymb}
\usepackage{graphicx}
\usepackage{booktabs}
\usepackage[ruled,vlined]{algorithm2e}
\usepackage{amsaddr}
\usepackage[left=2.2cm,right=2.2cm,top=2cm,bottom=2cm]{geometry}
\usepackage{mathtools}
\usepackage{accents}
\usepackage{mathrsfs}
\usepackage{tikz}
\usepackage{caption}
\usepackage{subcaption}
\usetikzlibrary{decorations.pathreplacing}
\usetikzlibrary{fadings}

\SetCommentSty{mycommfont}

\newcommand{\D}{{\mathop{}\!\mathrm{d}}} 

\newcommand{\R}{\mathbb{R}}
\newcommand{\Q}{\mathbb{Q}}

\newcommand{\N}{\mathbb{N}}
\newcommand{\Var}{\operatorname{Var}}
\newcommand{\PP}{\mathbb{P}}

\newcommand{\E}{\mathbb{E}}

\newcommand{\T }{\mathcal{T}}
\newcommand{\A}{\mathcal{A}}
\newcommand{\one}{ 1 \hspace{-3pt} \mathrm{l}} %

\newcommand{\Aloc}{{\mathcal{A}_{\operatorname{loc}}}}
\newcommand{\aloc}{{a_{\operatorname{loc}}}}
\newcommand{\Omloc}{{\Omega_{\operatorname{loc}}}}
\newcommand{\ab}{{\mathbf{a}}}

\numberwithin{equation}{section}  
\newtheorem{defn}{Definition}[section]

\newtheorem{rem}[defn]{Remark}
\newtheorem{thm}[defn]{Theorem}
\newtheorem{prop}[defn]{Proposition}

\newtheorem{lem}[defn]{Lemma}

\newtheorem{s_asu}[defn]{Standing Assumption}

\usepackage{hyperref}
\usepackage[textwidth=3.2cm]{todonotes}
\usepackage{xcolor}
\hypersetup{
    colorlinks,
    linkcolor={red!50!black},
    citecolor={blue!50!black},
    urlcolor={blue!80!black}
}

\title[Markov Decision Processes under Model Uncertainty]{Markov Decision Processes under Model Uncertainty}

\author[A. Neufeld, J. Sester, M. \v{S}iki\'c, ]{ Ariel Neufeld$^{1}$, Julian Sester$^{2}$, Mario \v{S}iki\'c$^{3}$}

\begin{document}

\maketitle

\begin{center}
\normalsize{\today} \\ \vspace{0.5cm}
\small\textit{$^{1}$NTU Singapore, Division of Mathematical Sciences,\\ 21 Nanyang Link, Singapore 637371.\\[2mm]
$^{2}$ National University of Singapore, Department of Mathematics,\\ 21 Lower Kent Ridge Road, 119077.                                                                                                                              \\[2mm]
$^{3}$University of Z\"{u}rich, Department of Banking and Finance,\\ Plattenstr.\,14, 8032 Zürich}                                                                                                                              
\end{center}

\begin{abstract}~
We introduce a general framework for Markov decision problems under model uncertainty in a discrete-time infinite horizon setting. 
By providing a dynamic programming principle we obtain a local-to-global paradigm, namely solving a local, i.e., a one time-step robust optimization problem leads to an optimizer of the global (i.e.\ infinite time-steps) robust stochastic optimal control problem, as well as to a corresponding worst-case measure.

Moreover, we apply this  framework to portfolio optimization involving data of the $S\&P~500$. We present two different types of ambiguity sets; one is fully data-driven given by a Wasserstein-ball around the empirical measure, the second one is described by a parametric set of multivariate normal distributions, where the corresponding uncertainty sets of the parameters are estimated from the data.
It turns out that in scenarios where the market is volatile or bearish, the optimal portfolio strategies from the corresponding robust optimization problem outperforms the ones without model uncertainty, showcasing the importance of taking model uncertainty into account.\\[2mm]

\noindent
{\bf Keywords:} {Markov decision problem, Ambiguity, Dynamic programming principle, Portfolio optimization}
\end{abstract}

\section{Introduction}
Suppose that today and at all future times an agent observes the state of the surrounding world, and based on the realization of this state she decides to execute an action that may also influence future states. All actions are rewarded according to a \emph{reward function} not immediately but once the subsequent state is realized. The \emph{Markov decision problem} consists of finding at initial time a policy, i.e., a sequence of state-dependent actions, that optimizes the expected cumulated discounted future rewards, referred to as the \emph{value} of the Markov decision problem. The underlying process of states in a Markov decision problem is a stochastic process $(X_t)_{t\in \N_0}$ and is called \emph{Markov decision process}. This process is usually modelled by a discrete-time  time-homogeneous Markov process that follows a pre-specified probability which is influenced by the current state of the process and the agent's current action.
The Markov decision problem leads to an infinite horizon stochastic optimal control problem in discrete-time which finds many applications in finance and economics, compare, e.g., \cite{bauerle2011markov}, \cite{hambly2021recent}, or \cite{white1993survey} for an overview. It can, among a multitude of other applications, be used to learn the optimal structure of portfolios and the optimal trading behaviour, see, e.g.\,\cite{bertoluzzo2012reinforcement}, \cite{chang2017incorporating},  \cite{gold2003fx}, \cite{hu2019deep}, \cite{xiong2018practical}, to learn optimal hedging strategies, see, e.g.\,\cite{angiuli2022reinforcement}, \cite{angiuli2021reinforcement}, \cite{cao2021deep}, \cite{dixon2020machine}, \cite{du2020deep},  \cite{halperin2020qlbs}, \cite{li2009learning}, \cite{schal2002markov}, to { optimize inventory-production systems (\cite{uugurlu2017controlled}), or to} study socio-economic systems under the influence of climate change as in  \cite{shuvo2020markov}. 

In most applications the choice of the distribution, or more specifically the probability kernel, of the Markov decision process however is a priori unclear and hence ambiguous. For this reason, in practice, the distributions of the process often need to be estimated, compare e.g.\,\cite{aguirregabiria2002swapping}, \cite{rust1994structural}, \cite{srisuma2012semiparametric}. To account for distributional ambiguity we are therefore interested to study an optimization problem respecting uncertainty with respect to the choice of the underlying distribution of $(X_t)_{t\in \N_0}$ by identifying a policy that maximizes the expected future cumulated rewards under the worst case probability measure from an ambiguity set of admissible probability measures. This formulation allows the agent to act optimally even if adverse scenarios are realized, such as for example during financial crises or extremely volatile market periods in financial markets.

The recent works \cite{bauerle2021q}, \cite{chen2019distributionally}, \cite{uugurlu2018robust}, and \cite{xu2010distributionally} also consider infinite horizon robust stochastic optimal control problems and follow a similar paradigm but use different underlying frameworks. More precisely, \cite{chen2019distributionally} and \cite{xu2010distributionally} assume a finite action and state space. The approach from \cite{uugurlu2018robust} assumes an atomless probability space and is restricted to so called \emph{conditional risk mappings}, whereas \cite{bauerle2021q} assumes the ambiguity set of probability measures to be dominated. To the best of our knowledge, the generality of the approach presented in this paper has not been established so far in the literature.

Our general formulation enables to specify a wide range of different ambiguity sets of probability measures and associated transition kernels, given some mild technical assumptions are fulfilled. More specifically, we require the correspondence that maps a state-action pair to the set of transition probabilities to be non-empty, continuous, compact-valued, and to fulfil a linear growth condition; see Assumption~\ref{asu_p}. As we will show, these requirements are \emph{naturally} satisfied. This is for example the case if the ambiguity set is modelled by a Wasserstein-ball around a transition kernel or if parameter uncertainty with respect to multivariate normal distributions is considered.

To solve the robust optimization problem we establish a dynamic programming principle that involves only a one time-step optimization problem. Via Berge's maximum theorem (see \cite{berge}) we obtain the existence of both an optimal action and a worst case transition kernel of this \emph{local} one time-step problem. It turns out that the optimal action that solves this one time-step optimization problem determines also the \emph{global} optimal policy of the infinite time horizon robust stochastic optimal control problem by repeatedly executing this \emph{local} solution. Similarly, the \emph{global} worst case measure can be determined as a product measure given by the infinite product of the worst case transition kernel of the \emph{local} one time-step optimization problem. We refer to Theorem~\ref{thm_main_result} for our main result. This local-to-global principle is in line with similar results for non-robust Markov decision problems, compare e.g.\,\cite[Theorem 7.1.7]{bauerle2011markov}, where the optimal global policy can also be determined locally. Note that the local-to-global paradigm obtained in Theorem~\ref{thm_main_result} is noteworthy, since $(X_t)_{t\in \N_0}$ does not need to be a time-homogeneous Markov process under each measure from the ambiguity set, as the corresponding transition kernel might vary with time. However, due to the particular setting that the \emph{set} of transition probabilities is constant in time and only depends on the current state and action, and not on the whole past trajectory, we are able to derive the analogue local-to-global paradigm for Markov decision processes under model uncertainty as for the ones without model uncertainty.  

Eventually we show how the discussed robust stochastic optimal control framework can be applied to portfolio optimization with real data, which was already studied extensively in the non-robust case, for example in \cite{bauerle2009mdp}, \cite{moody1998performance}, \cite{yu2019model}, and \cite{zhang2020deep}{, and in the robust case using a mean-variance approach in \cite{blanchet2021distributionally} and \cite{pham2022portfolio}}. To that end, we show how, based on a time series of realized returns of multiple assets of the $S\&P~500$, a data-driven ambiguity set of probability measures can be derived in two cases. The first case is an entirely data-driven approach where ambiguity is described by a Wasserstein-ball around the empirical measure. In the second case a multivariate normal distribution of the considered returns is assumed while the set of parameters for the multivariate normal distribution is estimated from observed data. Hence, this approach can be considered as semi data-driven approach.
We then train neural networks to solve the (semi) data-driven robust optimization problem based on the local-to-global paradigm obtained in Theorem~\ref{thm_main_result} and compare the trading performance of the two approaches with non-robust approaches. It turns out that under adverse market scenarios both robust approaches outperform comparable non-robust approaches. 
These results emphasize the importance of taking into account model uncertainty when making decisions that rely on financial assets.

The remainder of the paper is as follows. In Section~\ref{sec_setting} we present the setting and formulate the underlying distributionally robust stochastic optimal control problem. We also present our main results that include a dynamic programming principle. In Section~\ref{sec_distributional_uncertainty} we discuss different possibilities to define ambiguity sets of probability measures and we show that these specifications meet the requirements of our setting. In Section~\ref{sec_portfolio_optimization} we { provide a numerical routine using neural networks to approximately solve the distributionally robust stochastic optimal control problem of Section~\ref{sec_main_results}. We apply this numerical method} to portfolio optimization using real financial data and compare the different ambiguity sets introduced in Section~\ref{sec_distributional_uncertainty} also with non-robust approaches. The proof of the main results is reported in Section~\ref{sec_proof_main_results}, while the proofs of the results from Section~\ref{sec_distributional_uncertainty} and \ref{sec_portfolio_optimization} can be found in Section~\ref{sec_proofs_sec3} and \ref{sec_proofs_sec4}, respectively. Finally, the appendix contains  several useful auxiliary known mathematical results.
\section{Setting, Problem Formulation, and Main Result}\label{sec_setting}
We first present the underlying setting for the considered stochastic process and then formulate an associated distributionally robust optimization problem.
\subsection{Setting}\label{subsec_setting}
We consider a closed subset $\Omloc \subseteq\R^d$, equipped with its Borel $\sigma$-field $\mathcal{F}_{\operatorname{loc}}$, which we use to define the infinite Cartesian product
\[
\Omega:=\Omloc^{{ \N_0}}=\Omloc\times \Omloc \times \cdots
\]
and the $\sigma$-field $\mathcal{F}:=\mathcal{F}_{\operatorname{loc}}\otimes \mathcal{F}_{\operatorname{loc}} \otimes \cdots$.
We denote by $\mathcal{M}_1(\Omega)$ the set of probability measures on $(\Omega, \mathcal{F})$, by $d \in \N$ the dimension of the state space, and by $m \in \N$ the dimension of the control space.

On this space, we consider an infinite horizon time-discrete stochastic process. To this end, we define on $\Omega$ the stochastic process $\left(X_{t}\right)_{t\in \N_0}$  by the canonical process $X_t(\left(\omega_0,\omega_1,\dots,\omega_t,\dots\right)):=\omega_t$ for $(\omega_0,\omega_1,\dots,\omega_t,\dots) \in \Omega$,  $t \in \N_0$.
We fix a compact set $A \subseteq \R^m$ and define the set of controls (also called actions) through
\begin{align*}
\mathcal{A}:&=\left\{\ab=(a_t)_{t \in \N_0}~\middle|~(a_t)_{t \in \N_0}: \Omega \rightarrow A;~a_t \text{ is } \sigma(X_{t})\text{-measurable} \text{ for all } t \in \N_0 \right\}\\
&=\left\{\left(a_t(X_t)\right)_{t\in \N_0}~\middle|~ a_t:\Omloc \rightarrow A \text{ Borel measurable for all } t \in \N_0 \right\}.
\end{align*}


For every $k \in \N$, $X \subseteq \R^k$, and $p \in \N_0$, we define the set of continuous functions $g: X\rightarrow \R$ with polynomial growth at most of degree $p$  via
\[
C_p(X, \R):=\left\{g \in C(X,\R) ~\middle|~ \sup_{x \in X }\frac{|g(x)|}{1+\|x\|^p}< \infty\right\},
\]
where $C(X,\R)$ denotes the set of continuous functions mapping from $X$ to  $\R$ and $\| \cdot \|$ denotes the Euclidean norm on $\R^k$. We define on $C_p(\Omloc,\R)$ the norm
\begin{equation}\label{eq_defn_CP_norm}
\|g\|_{C_p}:=\sup_{x\in \Omloc} \frac{|g(x)|}{1+\|x\|^p}
\end{equation}
Moreover, recall the Wasserstein $p$ - topology  $\tau_p$ on $\mathcal{M}_1(\Omloc)$ induced by the { following} convergence: { for any $\mu \in \mathcal{M}_1(\Omloc)$ and $(\mu_n)_{n \in \N} \subseteq \mathcal{M}_1(\Omloc)$ we have}
\begin{equation}\label{eq_convergence_topology_1}
\mu_n \xrightarrow{\tau_p} \mu \text{ for } n \rightarrow \infty ~\Leftrightarrow~ \lim_{n \rightarrow \infty} \int g \D \mu_n = \int g \D \mu \text{ for all } g \in C_p(\Omloc, \R).
\end{equation}
Note that for $p=0$, the topology $\tau_0$ coincides with the topology of weak convergence.
To be able to formulate a robust optimization problem, we make use of the theory of set-valued maps, also called \emph{correspondences}, see also \cite[Chapter 17]{Aliprantis} for an extensive introduction to the topic. In the following we clarify how continuity is defined for correspondences, compare also Lemma~\ref{lem_upper_hemi} and Lemma~\ref{lem_lower_hemi}, where characterizations of upper hemicontinuity and lower hemicontinuity are provided.
\begin{defn}\label{def_hemi}
Let $\varphi:X \twoheadrightarrow Y$ be a correspondence between two topological spaces.
\begin{itemize}
\item[(i)]
$\varphi$ is called \emph{upper hemicontinuous}, if $\{x \in X~|~\varphi(x) \subseteq A\}$ is open for all open sets $A \subseteq Y$.
\item[(ii)] $\varphi$ is called \emph{lower hemicontinuous}, if 
$\{x \in X~|~ \varphi(x) \cap A \neq \emptyset\}$ is open for all open sets $A \subseteq Y$.
\item[(iii)] We say $\varphi$ is continuous, if $\varphi$ is upper and lower hemicontinuous.
\end{itemize}
\end{defn}
Moreover, for a correspondence $\varphi:X \twoheadrightarrow Y$ its graph is defined as 
\[
\operatorname{Gr}\varphi: = \left\{(x,y) \in X \times Y ~\middle|~y \in \varphi(x)\right\}.
\]

We impose the following standing assumptions\footnote{{ In this paper, for any $n \in \N$ and topological spaces $X_1,\dots,X_n$, we always endow $X=X_1\times \cdots \times X_n$ with the corresponding product topology.}} on the process $\left(X_t\right)_{t\in \N_0}$ and on the set of admissible measures, which are from now on assumed to be valid for the rest of the paper.
\begin{s_asu}[Assumptions on the set of measures]\label{asu_p}~
Fix $p \in \{0,1\}$.
\begin{itemize}
\item[(i)]
The set-valued map 
\begin{align*}
\Omloc \times A &\rightarrow(\mathcal{M}_1(\Omloc), \tau_p)\\
(x,a) &\twoheadrightarrow  \mathcal{P}(x,a)
\end{align*}
is assumed to be nonempty, compact-valued, and continuous.
\item[(ii)]
There exists $ C_P \geq 1 $ such that for all $(x,a) \in \Omloc \times A$ and $\PP \in \mathcal{P}\left(x,a\right)$ it holds
\begin{equation}\label{eq_growth_constraint_on_p}
\int_{\Omloc} \left(1+\|y\|^p\right) \PP(\D y)\leq C_P (1+\|x\|^p).
\end{equation}
\end{itemize}
\end{s_asu}
Under these assumptions we define for every $x \in \Omloc, \ab \in \mathcal{A}$ the set of admissible measures
\begin{align*}
\mathfrak{P}_{x,\ab}:=\bigg\{\delta_x \otimes \PP_0\otimes \PP_1 \otimes \cdots~\bigg|~&\text{ for all } t \in \N_0:~\PP_t:\Omloc \rightarrow \mathcal{M}_1(\Omloc) \text{ Borel-measurable, } \\ 
&\text{ and }\PP_t(\omega_t) \in \mathcal{P}\left(\omega_t,a_t(\omega_t)\right)\text{ for all } \omega_t\in \Omloc \bigg\},
\end{align*}
where the notation $\PP=\delta_x \otimes\PP_0\otimes \PP_1 \otimes\cdots \in \mathfrak{P}_{x,\ab}$ abbreviates\footnote{  We denote by $\delta_x$ the  Dirac measure centered on $x \in \R^d$, i.e., for any Borel set $A \subseteq \R^d$ we have $\delta_x(A) = 1$ if $x\in A$ and $0$ else.}
\[
\PP(B):=\int_{\Omloc}\cdots \int_{\Omloc} \cdots \one_{B}\left((\omega_t)_{t\in \N_0}\right) \cdots \PP_{t-1}(\omega_{t-1};\D\omega_t)\cdots \PP_0(\omega_0;\D\omega_1) \delta_x(\D \omega_0),\qquad B \in \mathcal{F}.
\]
\begin{rem}
To ensure that the set $\mathfrak{P}_{x,\ab}$ is nonempty, one needs to show that $\mathcal{P}$ admits a measurable selector. By Assumption~\ref{asu_p} the correspondence $\Omloc \times A \ni (x,a) \twoheadrightarrow  \mathcal{P}(x,a)$ is closed-valued and measurable. Hence, by  Kuratovski's Theorem (compare, e.g., \cite{kuratowski1930probleme} and \cite[Theorem 18.13]{Aliprantis}), there exists a measurable selector $\Omloc \times A \ni (x,a) \mapsto \PP(x,a) \in \mathcal{M}_1(\Omloc)$ such that $\PP(x,a) \in \mathcal{P}(x,a)$ for all $(x,a) \in \Omloc \times A$. Since actions are by definition measurable, we also obtain that for all $(a_t)_{t \in \N_0} \in \mathcal{A}$ and for all $t\in \N_0$ the map $\Omloc \ni \omega_t \mapsto \PP(\omega_t,a_t(\omega_t))=:\PP_t(\omega_t; \D \omega_{t+1})$ is measurable, as required. Then, the non-emptiness of $\mathfrak{P}_{x,\ab}$ follows by the Ionescu--Tulcea theorem { (compare, e.g.,  \cite[Theorem 14.32]{klenke2013probability} and \cite{ionescu_tulceau})}.
\end{rem}

\subsection{Problem Formulation}
Let $r:\Omloc \times A \times \Omloc \rightarrow \R$ be some \emph{reward function}. We assume from now on that it fulfils the following assumptions.
\begin{s_asu}[Assumptions on the reward function and the discount factor]\label{asu_2}~
Let $p \in \{0,1\}$ be the number fixed in Assumption~\ref{asu_p}.
\begin{itemize}
\item[(i)] The map 
\[
\Omloc \times A \times \Omloc \ni (x_0,a,x_1) \mapsto r(x_0,a,x_1)
\]
is continuous
\item[(ii)]There exists some $L > 0$ { and moduli of continuity\footnote{{ A modulus of continuity is a function $\rho:[0,\infty] \rightarrow [0,\infty]$ satistfying $\lim_{x\rightarrow 0} \rho(x)=0=\rho(0).$}} $\rho_0:[0,\infty] \rightarrow [0,\infty]$ and $\rho_A:[0,\infty] \rightarrow [0,\infty]$ } such that for all $x_0,x_0',x_1\in \Omloc$ and $a,a'\in A$ we have
\begin{equation}\label{eq_c_Lipschitz}
\left|r(x_0,a,x_1)-r(x_0',a',x_1)\right|\leq L \cdot { (1+\|x_1\|^p) \cdot} \left({ \rho_0 \left(\|x_0-x_0'\|\right)}+{ \rho_A\left(\|a-a'\|\right)}\right).
\end{equation}
\item[(iii)] There exists some $C_r\geq 1$ such that for all $x_0,x_1\in \Omloc$ we have
\begin{equation}\label{eq_c_bounded}
|r(x_0,a,x_1)| \leq C_r(1+\|x_0\|^p+\|x_1\|^p) \text{ for all } a \in A.
\end{equation}
\item[(iv)]
We fix an associated \emph{discount factor} $\alpha<1$ which satisfies
\[
0< \alpha < { \frac{1}{C_P},}
\]
{ where $C_P\geq 1$ is the constant defined in Assumption~\ref{asu_p}~(ii).}
\end{itemize}
\end{s_asu}

\begin{rem}[{ Discussion of the assumptions}]
\begin{itemize}
\item[(i)]

Note that if $r$ is Lipschitz-continuous, i.e., if there exists some $L > 0$ such that for all $x_0,x_0',x_1,x_1'\in \Omloc$ and $a,a'\in A$ we have
\begin{equation*}
\left|r(x_0,a,x_1)-r(x_0',a',x_1')\right|\leq L \left(\|x_0-x_0'\|+\|a-a'\|+\|x_1-x_1'\|\right),
\end{equation*}
then \eqref{eq_c_Lipschitz} follows directly. Therefore the requirement of Assumption~\ref{asu_2}~~(i) and (ii) is weaker than assuming Lipschitz continuity of the reward function. In particular, if $m=d$ holds for the dimensions, then the function of the form 
\begin{equation}\label{eq_linear_reward_function}
\Omloc \times A \times \Omloc \ni (x_0,a,x_1) \mapsto r(x_0,a,x_1):= a \cdot x_1 { - \lambda \cdot a^T \cdot M \cdot a}
\end{equation}
{ for some $\lambda \geq 0$ and some $M \in \R^{m \times m }$}
fulfils the requirement imposed in \eqref{eq_c_Lipschitz} but { is} not Lipschitz continuous, unless $\Omloc$ is bounded. Compare also Section~\ref{sec_portfolio_optimization}, where we apply portfolio optimization while taking into account a reward function of the form \eqref{eq_linear_reward_function}.
\item[(ii)]{ Note that Assumption~\ref{asu_p}~(ii) and Assumptions \ref{asu_2}~(iii), (iv) are standard assumptions for contracting Markov decision processes (compare, e.g., \cite[Definition 7.1.2~(ii)]{bauerle2011markov} and \cite[Corollary 7.2.2]{bauerle2011markov}). The continuity properties required in Assumption~\ref{asu_p}~(i), Assumption~\ref{asu_2}~(i), and Assumption~\ref{asu_2}~(ii) are assumptions tailored for robust Markov decision processes to ensure that the operator $\mathcal{T}$ defined in \eqref{eq_Tv_defn_main} is a contraction even if the image of  $\mathcal{P}$ is not a singleton and non-dominated, compare also the proof of Theorem~\ref{thm_main_result}~(ii).
}
\end{itemize}
\end{rem}

Our main problem consists, for every initial value $x\in \Omloc$, in maximizing the expected value of $\sum_{t=0}^\infty \alpha^tr(X_{t},a_t,X_{t+1})$ under the worst case measure from $\mathfrak{P}_{x,\ab}$ over all possible actions $\ab \in \A$. More precisely, we introduce the value function
\begin{equation}\label{eq_robust_problem_1}
\begin{aligned}
   \Omloc \ni x \mapsto V(x):=\sup_{\ab \in \mathcal{A}}\inf_{\PP \in \mathfrak{P}_{x,\ab}} \left(\E_{\PP}\bigg[\sum_{t=0}^\infty \alpha^tr(X_{t},a_t,X_{t+1})\bigg]\right).
\end{aligned}
\end{equation}

\begin{defn}
We call $\left(X_{t}\right)_{t\in \N_0}$ a {\emph{Markov decision process under model uncertainty}} on state space $\Omloc \subseteq \R^d$ with corresponding set of transition probabilities $\mathcal{P}$, and we call the problem defined in \eqref{eq_robust_problem_1} a {\emph{Markov decision problem under model uncertainty}}.
\end{defn}
\subsection{Main Result: The Dynamic Programming Principle}\label{sec_main_results}
In this section we provide the main results of the paper which comprise a \emph{dynamic programming principle} which in particular allows to solve the optimization problem \eqref{eq_robust_problem_1} by solving a related one-step { fixed} point equation.
%

To this end, we define the space of \emph{one-step actions}
\[
\mathcal{A}_{\operatorname{loc}} := \left\{ a_{\operatorname{loc}}:\Omloc \rightarrow A \text{ measurable} \right\},
\]
and, we define for every $\aloc \in \Aloc$ the set of kernels
\begin{align*}
\mathbf{P}_{\aloc}:= \big\{\PP_0:\Omloc\rightarrow \mathcal{M}_1(\Omloc) \text{ measurable}~\big|~&\PP_0(x) \in \mathcal{P}\left(x,\aloc (x)\right)\text{ for all } x\in \Omloc \big\}.
\end{align*}
Moreover, we define on $C_p(\Omloc,\R)$ the operator $\mathcal{T}$ which for every $v \in C_p(\Omloc,\R)$is defined by
\begin{equation}\label{eq_Tv_defn_main}
\Omloc \ni x \mapsto \T v(x):= \sup_{a \in A} \inf_{\PP \in \mathcal{P}(x,a)} \E_{\PP} \left[r(x,a,X_1)+\alpha v(X_1)\right].
\end{equation}
Our main findings are collected in the subsequent theorem.

\begin{thm}\label{thm_main_result}
Assume that Assumption~\ref{asu_p} and Assumption~\ref{asu_2} hold true. Then the following holds.
\begin{itemize}
\item[(i)] For every $v \in C_p(\Omloc,\R)$ there exists $\PP_0^*:\Omloc \times A \rightarrow \mathcal{M}_1(\Omloc)$ such that for all $(x,a) \in \Omloc \times A$ we have $\PP_0^*(x,a) \in \mathcal{P}(x,a)$
and 
\begin{equation}
\begin{aligned}\label{eq_claim_1_thm}
\E_{\PP_0^*(x,a)}[r(x,a,X_1)+\alpha v(X_1)]:&= \int_{\Omloc} r(x,a,\omega_1)+\alpha v (\omega_1) \PP_0^*(x,a;\D\omega_1)\\
&= \inf_{\PP_0 \in \mathcal{P}(x,a)} \E_{\PP_0}[r(x,a,X_1)+\alpha v(X_1)].
\end{aligned}
\end{equation}
Moreover, there exists $\aloc^* \in \Aloc$ such that for every $x \in \Omloc$ we have
\begin{equation}
\begin{aligned}\label{eq_claim_2_thm}
&\inf_{\PP_0 \in \mathcal{P}(x,\aloc^*(x))} \E_{\PP_0}[r(x,\aloc^*(x),X_1)+\alpha v(X_1)]\\
&=\sup_{\aloc \in \Aloc} \inf_{\PP_0 \in \mathcal{P}(x,\aloc(x))} \E_{\PP_0}[r(x,\aloc(x),X_1)+\alpha v(X_1)].
\end{aligned}
\end{equation}
Furthermore, let $\PP_{\operatorname{loc}}^*: \Omloc \rightarrow \mathcal{M}_1(\Omloc)$ be defined by
\[
\PP_{\operatorname{loc}}^*(x):=\PP_0^*(x,\aloc^*(x)),\qquad x \in \Omloc.
\]
Then $\PP_{\operatorname{loc}}^* \in \mathbf{P}_{\aloc^*}$ and for every $x\in \Omloc$ it holds that
\begin{equation}\label{eq_thm_assertion_1}
\begin{aligned}
\T v(x)&=\sup_{\aloc \in \Aloc}\inf_{\PP_0 \in \mathbf{P}_{{\aloc}}}\E_{\PP_0(x)}\big[r(x,\aloc(x),X_1)+\alpha  v(X_1)\big]\\&=\inf_{\PP_0 \in \mathbf{P}_{{\aloc^*}}}\E_{\PP_0(x)}\big[r(x,\aloc^*(x),X_1)+\alpha  v(X_1)\big]\\
&=\E_{\PP_{\operatorname{loc}}^*(x)}\big[r(x,\aloc^*(x),X_1)+\alpha  v(X_1)\big].
\end{aligned}
\end{equation}
\item[(ii)] We have that $\T  \bigg(C_p(\Omloc,\R)\bigg) \subseteq C_p(\Omloc,\R)$, i.e., $\T v \in C_p(\Omloc,\R)$ for all $v \in C_p(\Omloc,\R)$ and for all $v,w \in C_p(\Omloc,\R)$ the following inequality holds true
\begin{equation}\label{eq_contraction}
\left\|\T v-\T w\right\|_{C_p} \leq \alpha C_P \|v-w\|_{C_p}.
\end{equation}
In particular, there exists a unique $v \in C_p(\Omloc,\R)$ such that $\T v = v$.
Moreover, for every $v_0 \in C_p(\Omloc,\R)$ we have $v=\lim_{n \rightarrow \infty} \T ^nv_0$.
\item[(iii)] Let $v \in C_p(\Omloc,\R)$ satisfy $\T v=v$ and let $\aloc^* \in \Aloc$, $\PP_{\operatorname{loc}}^* \in \mathbf{P}_{{\aloc^*}}$ be defined as in (i).  Define 
$\ab^*:=(\aloc^*(X_0),\aloc^*(X_1),\dots)\in \mathcal{A}$
 and for all $ x \in \Omloc$, $\PP_x^*:=\delta_x \otimes \PP_{\operatorname{loc}}^* \otimes \PP_{\operatorname{loc}}^* \otimes \cdots \in \mathfrak{P}_{x,\ab^*}$. Then, for all $x \in \Omloc$ we have that
 \begin{equation}\label{eq_thm_assertion_3} 
\begin{aligned}
\E_{\PP^*_x}\bigg[\sum_{t=0}^\infty \alpha^tr(X_{t},\aloc^*(X_t),X_{t+1})\bigg]&=\inf_{\PP \in \mathfrak{P}_{x,\ab^*}}\E_{\PP}\bigg[\sum_{t=0}^\infty \alpha^tr(X_{t},\aloc^*(X_t),X_{t+1})\bigg]\\
&= V(x) \\
&=v(x).
\end{aligned}
 \end{equation}
\end{itemize}
\end{thm}
 \begin{rem}
Note that the local-to-global paradigm obtained in Theorem~\ref{thm_main_result} is noteworthy, since $(X_t)_{t\in \N_0}$ does not need to be a (time homogeneous) Markov process under each $\PP \in\mathfrak{P}_{x,\ab^*}$, as the corresponding transition kernel might vary with time. However, due to the particular setting that the set of transition probabilities $(x,a) \mapsto \mathcal{P}(x,a)$ is constant in time and only depends on the current state and action, and not on the whole past trajectory, we are able to derive the analogue local-to-global paradigm for Markov decision processes under model uncertainty as for the ones without model uncertainty.  Moreover, if $(x,a) \mapsto \mathcal{P}(x,a)$ is single-valued, then $(X_t)_{t\in \N_0}$ is a Markov decision process in the classical sense, compare, e.g., \cite{bauerle2011markov}. This justifies to call $(X_t)_{t\in \N_0}$ a Markov decision process under model uncertainty on state space $\Omloc \subseteq \R^d$ with respect to $\mathcal{P}$.
Moreover, note that a posteriori, we see that $(X_t)_{t\in \N_0}$ is a time-homogeneous Markov process under the worst-case measure, and the optimal strategy only depends on the current state of the process, and not on time, as observed for classical Markov decision problems. 

 \end{rem}

\section{Capturing distributional uncertainty}\label{sec_distributional_uncertainty}
In this section we present different approaches that enable to capture uncertainty with respect to the choice of the underlying probability measure.{ We focus on ambiguity sets of probability measures which can be constructed from observed data. More precisely, in the first case we consider an entirely data-driven approach where the ambiguity set is described by a Wasserstein-ball around, for example, the empirical measure. In the second case, we follow a semi data-driven approach by considering a parametric family of distributions as ambiguity set, where the set of feasible parameters can be estimated from observed data. Finally, we also demonstrate how the two approaches can be generalized to the case where the state process $(X_t)_{t \geq 0}$ describes an autocorrelated time series}. We show that all of the presented approaches fulfil the requirements of the setting presented in Section~\ref{sec_setting}.
\subsection{Uncertainty expressed through the Wasserstein distance}\label{sec_wasserstein}
The first example involves the case when distributional uncertainty is captured through the $q$-Wasserstein-distance $W_q(\cdot,\cdot)$ for some  $q\in \N$. For any $\PP_1,\PP_2 \in \mathcal{M}_1(\Omloc)$ let $W_q(\PP_1,\PP_2)$ be defined as 
\[
W_q(\PP_1,\PP_2):=\left(\inf_{\pi \in \Pi(\PP_1,\PP_2)}\int_{\Omloc \times \Omloc} \|x-y\|^q \D \pi(x,y)\right)^{1/q},
\]
where $\|\cdot\|$ denotes the Euclidean norm on $\R^d$, and where $\Pi(\PP_1,\PP_2)$ denotes the set of joint distributions of $\PP_1$ and $\PP_2$, compare also for example \cite[Definition 6.1.]{villani2009optimal}. 

We fix some $q\in \N$ and specify $p:=0$ in Assumption~\ref{asu_p} and \ref{asu_2}. Further, we assume that there exists a continuous map
\begin{equation}\label{eq_definition_p_hat}
\begin{aligned}
\Omloc \times A &\rightarrow (\mathcal{M}_1(\Omloc),\tau_q)\\
(x,a) &\mapsto \widehat{\PP}(x,a)
\end{aligned}
\end{equation}
such that $\widehat{\PP}(x,a)$ has finite $q$-th moments for all $(x,a) \in \Omloc \times A$.  Then, we define for any $\varepsilon>0$ the set-valued map
\begin{equation}\label{eq_definition_wasserstein_ball}
\Omloc \times A  \ni (x,a) \twoheadrightarrow  \mathcal{P}(x,a):=\mathcal{B}^{(q)}_\varepsilon\left(\widehat{\PP}(x,a)\right):=\left\{\PP\in \mathcal{M}_1(\Omloc)~\middle|~W_q(\PP,\widehat{\PP}(x,a)) \leq  \varepsilon \right\},
\end{equation}
where $\mathcal{B}^{(q)}_\varepsilon\left(\widehat{\PP}(x,a)\right)$ denotes the $q$-Wasserstein-ball (or Wasserstein-ball of order $q$) with $\varepsilon$-radius and center $\widehat{\PP}(x,a)$.

\begin{prop}\label{prop_wasserstein}
Let $\Omloc \times A \ni (x,a) \mapsto \widehat{\PP}(x,a) \in  (\mathcal{M}_1(\Omloc),\tau_q)$ be continuous with finite $q$-th moments. Then, the set-valued map  $\Omloc \times A \ni (x,a) \twoheadrightarrow \mathcal{P}(x,a)$ defined as in \eqref{eq_definition_wasserstein_ball} fulfils the requirements of Assumption~\ref{asu_p} with $p=0$.
\end{prop}

\subsection{Knightian uncertainty in parametric models}\label{sec_parametric}
Next, we consider a parametric approach, taking into account the so called \emph{Knightian uncertainty}, 
{ which is named after the American economist Frank Knight (see \cite{knight1921risk}) and which describes the unquantifiable risk of having chosen the wrong model to determine probabilities for future events. Following this paradigm we aim at describing a class of parametric models that model future events, but in order to take into account Knightian uncertainty and to avoid a misspecification by choosing a \emph{wrong} model, we allow for a range of possible parameters.}
 To this end, we consider a set-valued map of the form 
\begin{equation}\label{eq_defn_set_valued_theta}
\Omloc \times A \ni (x,a) \twoheadrightarrow \Theta(x,a) \subseteq \R^\mathfrak{D},\qquad  \text{ for some } \mathfrak{D} \in \N.
\end{equation}
The set $\Theta(x,a)$ refers to the set of parameters that are admissible in dependence of $(x,a) \in \Omloc \times A$. The underlying parametric probability distribution is described by 
\begin{equation}\label{eq_defn_set_valued_p_hat}
\begin{aligned}
\{(x,a,\theta) ~|~(x,a) \in \Omloc \times A,~ \theta \in \Theta(x,a) \} &\rightarrow (\mathcal{M}_1(\Omloc), \tau_p)\\
(x,a,\theta) &\mapsto \widehat{\PP}(x,a,\theta),
\end{aligned}
\end{equation}
which enables us to define the ambiguity set of probability measures by
\begin{equation}\label{eq_defn_set_valued_P}
\Omloc \times A \ni (x,a) \twoheadrightarrow \mathcal{P}(x,a):=\left\{\widehat{\PP}(x,a,\theta)~\middle|~\theta \in \Theta(x,a) \right\}\subseteq (\mathcal{M}_1(\Omloc), \tau_p).
\end{equation}
\begin{prop}\label{prop_knightian_1}
Let $(x,a) \twoheadrightarrow \Theta(x,a)$, as defined in \eqref{eq_defn_set_valued_theta}, be nonempty, compact-valued, and continuous, let $(x,a,\theta) \mapsto \widehat{\PP}(x,a,\theta)$, as defined in \eqref{eq_defn_set_valued_p_hat}, be continuous. Then $(x,a) \twoheadrightarrow \mathcal{P}(x,a)$, as defined in \eqref{eq_defn_set_valued_P}, is nonempty, compact-valued, and continuous.
\end{prop}

\subsection{Uncertainty in autocorrelated time series}\label{sec_uncertainty_autocorr}

Next, we consider the case where the state process $(X_t)_{t\in \N_0}$ is given by an autocorrelated time series. More precisely, we assume that at time $t\in \N_0$ the past $m \in \N$ observations $(Y_{t-m+1},\dots,Y_t)$ of a time series $(Y_{t})_{t \in \{-m,\dots-1,0,1\dots\}}$ may have an influence on the next value of the state process. In this case we have for all $t\in \N$ a representation of the form
\[
X_t:=(Y_{t-m+1},\dots,Y_t) \in \Omloc:={ Z}^m \subseteq \R^{{D}\cdot m},\text{ with } { Z} \subseteq \R^{{D}} \text{ closed, for some } {D} \in \N.
\]
To define the ambiguity set of measures, we first consider a set-valued map of the form
\begin{equation}\label{eq_defn_P_tilde_autocorr}
\Omloc \times A \ni (x,a) \twoheadrightarrow \widetilde{\mathcal{P}}(x,a) \subseteq (\mathcal{M}_1({ Z}),\tau_p).
\end{equation}
We consider the projection $\Omloc \ni (x_1,\dots,x_m) \mapsto \pi((x_1,\dots,x_m)):= (x_2,\dots,x_m)\in { Z}^{m-1}$ that projects onto the last $m-1$ components,
and define a set-valued map $\mathcal{P}$, in dependence of $\widetilde{\mathcal{P}}$, by
\begin{equation}\label{eq_defn_P_autocorr}
\Omloc \times A \ni (x,a) \twoheadrightarrow \mathcal{{P}}(x,a):=\left\{\delta_{\pi(x)} \otimes \PP~\middle|~\PP \in \widetilde{\mathcal{P}}(x,a) \right\} \subset (\mathcal{M}_1(\Omloc),\tau_p).
\end{equation}
This means, by considering ${\mathcal{P}}$, we take into account uncertainty with respect to the evolution of the next value of the time series. However, we do not want to consider uncertainty with respect to the $m-1$ preceding values of the time series, as they constitute of the already observed realizations.

\begin{prop}\label{prop_auto_correlation}
Let $(x,a) \twoheadrightarrow \widetilde{\mathcal{P}}(x,a) $, as defined in \eqref{eq_defn_P_tilde_autocorr}, be nonempty, compact-valued, and continuous. Then $(x,a) \twoheadrightarrow \mathcal{{P}}(x,a)$, as defined in \eqref{eq_defn_P_autocorr} is nonempty, compact-valued, and continuous.
\end{prop}

\section{Application to Portfolio Optimization}\label{sec_portfolio_optimization}
In this section we discuss a finance-related application of the presented robust stochastic optimal control problem of Section~\ref{sec_setting}. In particular, we compare different specifications to measure uncertainty with respect to the choice of the underlying probability measure.

\subsection*{Setting}\label{sec_portfolio_optimization_setting}

We present a setting that can be applied to the robust optimization of financial portfolios. Compare among many others also \cite{boyd2017multi}, \cite[Chapter 10]{dixon2020machine}, and \cite{filos2019reinforcement}, where alternative approaches to portfolio optimization relying on the optimal control of Markov decision processes are discussed. 
Let $D \in \N$ denote the number of assets that are taken into account for portfolio optimization.
Then, the underlying asset returns in the time period between $t-1$ and $t$ are given by
\[
\mathcal{R}_{t}:=\left(\mathcal{R}_{t}^i\right)_{i=1,\dots,D}:=\left(\frac{S_{t}^i-S_{t-1}^i}{S_{t-1}^i}\right)_{i=1,\dots,D} \in { Z} \subseteq \R^{D},\qquad t \in \{-m+1,\dots,0,1,\dots,\},
\]
where $S_{t}^i\in (0,\infty)$ denotes the time $t$-value of asset $i \in \{1,\dots,D\}$, $m \in \N$, and ${ Z} \subseteq \R^D$ closed.

To take into account the autocorrelation of the time series, we want to base our portfolio allocation decisions not only on the current portfolio allocation and the present state of the financial market, but also on the past $m\in \N$ observed returns. Thus, we consider at every time $t\in \N_0$ realized returns $\left(\mathcal{R}_{t-m+1},\cdots,\mathcal{R}_{t}\right) \in { Z}^m$. 
Then, the underlying stochastic process $(X_t)_{t\in \N_0}$ is modelled 
as
\begin{equation}\label{eq_defn_X_t_portfolio}
X_t:=\left(\mathcal{R}_{t-m+1},\cdots,\mathcal{R}_{t}\right) \in \Omloc,\qquad t\in \N_0,
\end{equation}
with  
\[
\Omloc := { Z} ^m\subseteq \R^{D \cdot m}.
\]
Next, we introduce the compact set
\[
A:= \left\{a=(a^i)_{i=1,\dots,D} \in [-C,C]^{D} \right\},
\]
for the possible values of the controls, which corresponds to the monetary investment in the $D$ stocks, where $C>0$ relates to a budget constraint when investing.
Then, we define the reward function by
\begin{equation}\label{eq_def_c_portfolio}
\Omloc \times A \times \Omloc \ni \left(X_{t},a_t,X_{t+1}\right) \mapsto r\left(X_{t},a_t,X_{t+1}\right):=\sum_{i=1}^D a_t^i\cdot \mathcal{R}_{t+1}^i{ -\lambda \cdot  \left(a_t^T \cdot \Sigma_{\mathcal{R}} \cdot a_t \right),}
\end{equation}
{ for some risk-aversion parameter $\lambda\geq 0$ and a covariance matrix $\Sigma_{\mathcal{R}} \in \R^{D \times D}$ associated to the asset returns.}
The reward function in \eqref{eq_def_c_portfolio} expresses the cumulated gain from trading in the period between $t$ and $t+1$, { where \emph{risky} positions are additionally penalized by a risk measure expressed in terms of the variance of the cumulated gain from trading  between $t$ and $t+1$  multiplied with a risk-aversion parameter $\lambda$. This approach is similar to the approaches presented in \cite[Section 4.2]{boyd2017multi} or \cite[Chapter 10, Section 5.6]{dixon2020machine}. We will specify a data-driven estimate for the covariance matrix $\Sigma_{\mathcal{R}}$ in the next subsection.}
\subsection{Data-driven ambiguity set and Wasserstein-uncertainty}\label{sec_portfolio_wasserstein}
We rely on the setting elaborated above.

As exposed in Section~\ref{sec_wasserstein}, we may capture distributional uncertainty by considering a Wasserstein-ball around some kernel 
\[
 \Omloc \times A \ni (x,a)\mapsto  \widehat{\PP}(x,a)\in \mathcal{M}_1({ Z} ),
 \]
 for ${ Z}  \subseteq \R^D$ closed.
We consider a time series of past realized returns 
\begin{equation}\label{eq_time_series_returns}
\left(\mathscr{R}_1,\dots,\mathscr{R}_{N} \right) \in { Z} ^N,\qquad \text{ for some } N \in \N { \cap [2,\infty)} .
\end{equation}
Compare also Figure~\ref{fig_dates}, where we illustrate the relation between this time series and the time series of future returns.

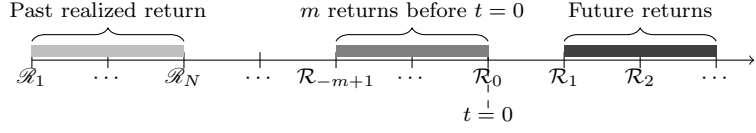
\begin{figure}[h!]
\begin{center}
\begin{tikzpicture}[%
    every node/.style={
        font=\scriptsize,
        text height=1ex,
        text depth=.25ex,
    },
]
\draw[->] (0,0) -- (9.5,0);

\foreach \x in {0,1,...,9}{
    \draw (\x cm,3pt) -- (\x cm,-3pt);
}
\draw[dashed,black](6,0.1)--(6,-0.7);

\node[anchor=north] at (0,0) {$\mathscr{R}_1$};
\node[anchor=north] at (1,0) {$\cdots$};
\node[anchor=north] at (2,0) {$\mathscr{R}_N$};
\node[anchor=north] at (3,0) {$\cdots$};
\node[anchor=north] at (4,0) {$\mathcal{R}_{-m+1}$};
\node[anchor=north] at (5,0) {$\cdots$};
\node[anchor=north] at (6,-0.5) {$t=0$};
\node[anchor=north] at (6,0) {$\mathcal{R}_0$};
\node[anchor=north] at (7,0) {$\mathcal{R}_1$};
\node[anchor=north] at (8,0) {$\mathcal{R}_2$};
\node[anchor=north] at (9,0) {$\cdots$};

\fill[lightgray] (0,0.05) rectangle (2,0.2);
\fill[gray] (4,0.05) rectangle (6,0.2);
\fill[darkgray] (7,0.05) rectangle (9,0.2);

\draw[decorate,decoration={brace,amplitude=5pt}] (0,0.25) -- (2,0.25)
    node[anchor=south,midway,above=4pt] {Past realized return};
 \draw[decorate,decoration={brace,amplitude=5pt}] (4,0.25) -- (6,0.25)
    node[anchor=south,midway,above=4pt] {$m$ returns before $t=0$};
\draw[decorate,decoration={brace,amplitude=5pt}] (7,0.25) -- (9,0.25)
    node[anchor=south,midway,above=4pt] {Future returns};
\end{tikzpicture}
\end{center}
\caption{Illustration of the observed and already realized return $(\mathscr{R}_t)_{t=1,\dots,N}$ and the future random returns $(\mathcal{R}_t)_{t=-m+1, \cdots,0,1,\cdots}$.}\label{fig_dates}
\end{figure}

Relying on the time series from \eqref{eq_time_series_returns}, we aim at constructing an ambiguity set $\mathcal{P}$. To this end, we define $\widehat{\PP}$ through a sum of Dirac-measures given by\footnote{Note that $x \mapsto \widehat{\PP}(x)$ does not depend on $a \in A$.}
\begin{equation}\label{eq_defn_phat_portfolio_1}
\Omloc \ni X_t= \left(\mathcal{R}_{t-m+1},\cdots,\mathcal{R}_{t}\right) \mapsto \widehat{\PP}\left(X_t\right) (\D x):= \sum_{s=m}^{N-1}\pi_s(X_t) \cdot \delta_{\mathscr{R}_{s+1}}(\D x)\in \mathcal{M}_1({ Z}),
\end{equation}
where $\pi_s(X_t)\in [0,1]$, $s =m,\cdots,N-1$ with $\sum_{s=m}^{N-1}\pi_s(X_t) = 1$. We want to weight the distance between the past $m$ returns  before $\mathcal{R}_{t+1}$ and  the $m$ returns before $\mathscr{R}_{s+1}$, while assigning higher probabilities to more similar sequences of $m$ returns. This means, the measure $\widehat{\PP}$ relies its prediction for the next return on the best fitting sequence of $m$ consecutive returns that precede the prediction. To this end, we set for some (small) constant $\widetilde{\varepsilon}>0$\footnote{The constant $\widetilde{\varepsilon}>0$ is merely a technical requirement which is considered to avoid division by zero in the case $\operatorname{dist}_s(X_t)=0$ for some indices $s\in \{1,\dots,N\},t\in \N_0$, i.e., in the case that a sequence of $m$ random returns equals a sequence of past realized returns. Hence, in practice, $\widetilde{\varepsilon}$ can be set to be a negligible small positive real number.}
\begin{equation*}\label{eq_defn_ps_empirical}
\Omloc \ni X_t= \left(\mathcal{R}_{t-m+1},\cdots,\mathcal{R}_{t}\right)  \mapsto\pi_s(X_t):= \left(\frac{(\operatorname{dist}_s(X_t)+\widetilde{\varepsilon})^{-1}}{\sum_{\ell=m}^{N-1}(\operatorname{dist}_\ell(X_t)+\widetilde{\varepsilon})^{-1}}\right),
\end{equation*}
with
\begin{align*}
\operatorname{dist}_s(X_t):=\left\|\left(\mathscr{R}_{s-m+1},\cdots,\mathscr{R}_{s}\right)-X_t\right\|=\left\|\left(\mathscr{R}_{s-m+1},\cdots,\mathscr{R}_{s}\right)-\left(\mathcal{R}_{t-m+1},\cdots,\mathcal{R}_{t}\right)\right\|,~~ 
\end{align*}
for all $s=m,\cdots,N-1$. Then, we define for any fixed $\varepsilon>0$ and $q\in \N$ the ambiguity set of probability measures on $\mathcal{M}_1(\Omloc)$ via the set-valued map\footnote{Note that $\mathcal{P}$ does not depend on $a \in A$, and recall that  $\Omloc \ni (x_1,\dots,x_m) \mapsto \pi((x_1,\dots,x_m)):= (x_2,\dots,x_m)\in T^{m-1}$ denotes the projection onto the last $m-1$ components. }
\begin{equation}\label{eq_defn_P_portfolio_wasserstein}
\Omloc \ni x \twoheadrightarrow \mathcal{P}(x):= \left\{ \delta_{\pi(x)}\otimes \PP~\middle|~\PP\in  \mathcal{B}_{\varepsilon}^{(q)}\left(\widehat{\PP}(x)\right)\right\} \subseteq \left(\mathcal{M}_1(\Omloc), \tau_p\right)
\end{equation}
that takes into account a $q$-Wasserstein-ball around $\widehat{\PP}$ for the next future return.
{ Relying on the time series of realized returns in \eqref{eq_time_series_returns}, we estimate a covariance matrix $\Sigma_{\mathcal{R}} \in \R^{D\times D}$ by
\[
\Sigma_{\mathcal{R}}: = \frac{1}{N-1}\sum_{s=1}^N (\mathscr{R}_{s}-\operatorname{ER})\cdot (\mathscr{R}_{s}-\operatorname{ER})^T
\]
for $\operatorname{ER}:= \tfrac{1}{N} \sum_{i=1}^N \mathscr{R}_{s} \in \R^D$.
Hence, this choice for $\Sigma_{\mathcal{R}}$ specifies the reward function from \eqref{eq_def_c_portfolio}.
}
\begin{prop}\label{prop_portfolio_wasserstein} Let ${ Z}  \subseteq \R^D$ be compact, and let $p=0$. Then, the set-valued map $\mathcal{P}$, defined in \eqref{eq_defn_P_portfolio_wasserstein}, satisfies Assumption \ref{asu_p}. Moreover, the reward function $r$, defined in { \eqref{eq_def_c_portfolio}}, satisfies Assumption~\ref{asu_2}.
\end{prop}

\subsection{Parametric Uncertainty}\label{sec_portfolio_knightian}

Next, we introduce a parametric approach in which we assume that the asset returns follow a multivariate normal distribution with unknown parameters.\footnote{
We say that $X \in \R^D$ has a $D$-dimensional multivariate normal distribution with mean $\mu \in \R^D$ and covariance matrix $\Sigma \in \R^{D\times D}$ which is symmetric and positive semidefinite if the characteristic function of $X$ is of the form $\R^D \ni u \mapsto \varphi_X(u):= \exp\left(i u^T \mu- \tfrac{1}{2} u^T \Sigma u\right)$, compare e.g. \cite[p. 124]{gut2009multivariate}. We write $X \sim \mathcal{N}_D(\mu, \Sigma)$.}

To this end, we build on the setting exposed in Section~\ref{sec_portfolio_optimization_setting}, where $m>1$, and where we choose ${ Z} = \R^D$, and $p=1$.
Moreover, we consider the following unbiased estimators of mean and covariance
\begin{equation}\label{eq_mean_estimator}
\begin{aligned}
\mathfrak{m}:(\R^D)^m &\rightarrow \R^D\\
x=(x_1,\dots,x_m) &\mapsto \frac{1}{m}\sum_{i=1}^m x_i,
\end{aligned}
\end{equation}
and 
\begin{equation}\label{eq_cov_estimator}
\begin{aligned}
\mathfrak{c}:(\R^D)^m &\rightarrow \R^{D\times D}\\
x=(x_1,\dots,x_m) &\mapsto \frac{1}{m-1}\sum_{i=1}^m (x_i-\mathfrak{m}(x))\cdot (x_i-\mathfrak{m}(x))^T.
\end{aligned}
\end{equation}
Let $\varepsilon>0$. To define the set of admissible parameters we consider the following set-valued maps
\begin{align*}
\Omloc \ni x &\twoheadrightarrow  \widehat{\mu}(x):=\left\{ {\mu} \in \R^D~\middle|~\|{\mu}-\mathfrak{m}(x)\|\leq \varepsilon \right\},\\
\Omloc \ni x &\twoheadrightarrow  \widehat{\Sigma}(x):=\left\{{\Sigma} \in \R^{D\times D} ~\middle|~{\Sigma} = \mathfrak{c}(y
) \text{ for some } y \in \Omloc \text{ with } \|y-x \|\leq \varepsilon \right\},\\
\Omloc \ni x &\twoheadrightarrow \Theta(x):=\left\{ ({\mu}, {\Sigma}) \in \R^D \times \R^{D\times D}~\middle|~{\mu} \in \widehat{\mu}(x), \Sigma \in \widehat{\Sigma}(x)\right\}.
\end{align*}
We define an ambiguity set related to $D$-dimensional multivariate normal distributions by
\begin{align*}
\Omloc  \ni x &\twoheadrightarrow\widetilde{\mathcal{P}}(x):=\left\{\mathcal{N}_D( {\mu},{\Sigma})~\middle|~({\mu},{\Sigma}) \in \Theta(x) \right\} \subseteq \left(\mathcal{M}_1(\R^D),\tau_1 \right).
\end{align*}
As in Section~\ref{sec_uncertainty_autocorr} we denote by $\Omloc \ni (x_1,\dots,x_m) \mapsto \pi((x_1,\dots,x_m)):= (x_2,\dots,x_m)\in \R^{D \cdot (m-1)}$ the projection onto the last $m-1$ components. 
These definitions allow us to define the ambiguity set on $\mathcal{M}_1(\Omloc)$ by\footnote{Note that $\mathcal{P}$ does not depend on $a \in A$, and recall that  $\Omloc \ni (x_1,\dots,x_m) \mapsto \pi((x_1,\dots,x_m)):= (x_2,\dots,x_m)\in R^{D(m-1)}$ denotes the projection onto the last $m-1$ components.}
\begin{equation}\label{eq_defn_P_multivariate_norm_portfolio}
\Omloc \times A \ni (x,a) \twoheadrightarrow \mathcal{P}(x,a) := \left\{ \delta_{\pi(x)}\otimes \PP~\middle|~ \PP \in \widetilde{\mathcal{P}}(x) \right\} \subseteq \left(\mathcal{M}_1(\Omloc),\tau_1 \right).
\end{equation}
This means, we consider as an ambiguity set for the next return a set of multivariate normal distributions with unknown mean and covariance, where the set of admissible means and covariances is specified by the estimators $\mathfrak{m}$ and $\mathfrak{c}$ as well as by the degree of ambiguity specified through $\varepsilon$.

\begin{prop}\label{prop_parametric}
Let ${ Z}  = \R^D$ and $p=1$. Then, the ambiguity set $\mathcal{P}$, as defined in \eqref{eq_defn_P_multivariate_norm_portfolio}, fulfils the requirements from Assumption~\ref{asu_p}. Moreover, the reward function $r$, defined in \eqref{eq_def_c_portfolio}, satisfies Assumption~\ref{asu_2}.
\end{prop}

\subsection{{ Numerics}}\label{sec_numerics}
We present an explicit numerical algorithm that can be applied to { approximately} compute the optimal value function $V$ and to determine an optimal policy $\ab^* \in \mathcal{A}$. { We emphasize the need of an algorithm which produces approximate optimal solutions, since even the one time-step optimization problem in \eqref{eq_Tv_defn_main} leading to an optimal policy according to Theorem~\ref{thm_main_result} involves the fixed point $\mathcal{T}v=v$ which cannot be derived explicitly. Hence one cannot expect to obtain an explicit optimal policy.}
\subsubsection{Value Iteration}
Theorem~\ref{thm_main_result} directly provides an algorithm for the computation of the optimal value $V(x)$ to which we refer as the \emph{value iteration algorithm}.\\
In this algorithm we start with an arbitrary $V^{(0)} \in C_p(\Omloc,\R)$ and then compute recursively $V^{(n+1)}:=\T V^{(n)}$ for all $n \in \N_0$. According to Theorem~\ref{thm_main_result} we then have 
\begin{equation}\label{eq_value_iteration}
\lim_{n \rightarrow \infty} \T V^{(n)} = V.
\end{equation}
Compare also, e.g., \cite[Section 7]{bauerle2011markov}, where this algorithm (in a non-robust setting) is discussed in detail.
\subsubsection{Numerical Algorithm}
With Algorithm~\ref{algo_reinforcement_value_iteration} we present a pseudocode of the  methodology that can be applied to compute both the optimal value function and the optimal policy numerically. The algorithm relies on the  value iteration principle. This means we solve \eqref{eq_value_iteration} by approximating the value function $V$ through neural networks\footnote{For a general introduction to neural networks we refer the reader to \cite{lecun2015deep}, for applications of neural networks in finance see \cite{dixon2020machine}, for a proof of the universal approximation property of neural networks compare \cite{hornik1991approximation}.} and by repeatedly applying the recursion $V^{(n+1)}:=\T V^{(n)}$. Note that Algorithm~\ref{algo_reinforcement_value_iteration} approximates an optimal one-step action $\aloc^* \in \Aloc$. According to Theorem~\ref{thm_main_result} an approximation of the optimal policy can then be obtained as $\ab^*:=(\aloc^*(X_0),\aloc^*(X_1),\dots) \in \mathcal{A}$.
{ We also remark that, to approximate the minimum among all measures from the ambiguity set, we sample a finite amount $N_{\mathcal{P}}\in \N$ of measures from the ambiguity set in each iteration step. In the case that the ambiguity set is given by a Wasserstein-ball, we apply the algorithm presented in \cite[Algorithm 2]{neufeld2022detecting} ensuring that the sampled measures lie within the Wasserstein-ball around the corresponding reference measure.}

\begin{algorithm}[h!]
\SetAlgoLined
\SetKwInOut{Input}{Input}
\SetKwInOut{Output}{Output}

\Input{Batch Size $B\in \N$; Hyperparameters for the neural networks;  Number of epochs ${E}$; Number of  iterations $\operatorname{Iter}_v$  for the improvement of the value function; Number of iterations  $\operatorname{Iter}_a$ for the improvement of the action function; Number of measures  $N_{\mathcal{P}}$; Number  of Monte-Carlo simulations $N_{\operatorname{MC}}$; State space $\Omloc$; Action space $A$; Reward function $r$;}
Initialize a neural network $V^0$;\\
Initialize a neural network $\aloc$;\\
\For{$\operatorname{epoch}= 1,\dots,{E}$}{
Set ${V}^{\text{epoch}}=V^{\text{epoch}-1}$ and freeze the weights of $V^{\text{epoch}-1}$;\\

\For{$\operatorname{iteration} =1,\dots, \operatorname{Iter}_a $}{
  \tcc{We maximize $\inf_{\PP_0 \in \mathcal{P}(x,\aloc(x))} \E_{\PP_0}[r(x,\aloc(x),X_1)+\alpha V^{\text{epoch}-1}(X_1)]$ with respect to $\aloc \in \Aloc$.}
Sample a batch of states $(x_i)_{i=1,\dots,B} \subseteq \Omloc$;\\
\For{$i=1,\dots,B$}{
Pick measures $\PP_1^{(i)},\dots,\PP^{(i)}_{N_{\mathcal{P}}} \in \mathcal{P}(x_i,\aloc(x_i))$;
}
Denote by $X_{1,\PP}^{(j)}$ a random variable that is sampled according to a measure $\PP\in \mathcal{M}_1(\Omloc)$ for $j=1,\dots,N_{\operatorname{MC}}$;\\
Sample $X_{1,\PP}^{(j)}$ for all $\PP\in \left\{\PP_1^{(i)},\dots,\PP^{(i)}_{N_{\mathcal{P}}}\right\}$, $j \in \{1, \dots,N_{\operatorname{MC}}\}$, $i\in \{1,\dots,B\}$;\\
Define for $i=1,\dots,B$:
\[
\widehat{\mathcal{T}V}(x_i):=\min_{\PP\in \left\{\PP_1^{(i)},\dots,\PP^{(i)}_{N_{\mathcal{P}}}\right\}} \frac{1}{N_{\operatorname{MC}}}\sum_{j=1}^{N_{\operatorname{MC}}}r\left(x_i,\aloc(x_i),X_{1,\PP}^{(j)}\right)+\alpha V^{\text{epoch}-1}\left(X_{1,\PP}^{(j)}\right);
\]

Maximize 
\[
\sum_{i=1}^B\widehat{\mathcal{T}V}(x_i)
\]
with respect to parameters from the neural network $\aloc$ (e.g.\,back-propagation with a stochastic gradient descent algorithm).
}

\For{$\operatorname{iteration} =1,\dots,\operatorname{Iter}_v$}{
  \tcc{We minimize the quadratic error between $V^{\text{epoch}}(x)$ and the approximation of  $\sup_{\aloc \in \Aloc}\inf_{\PP_0 \in \mathcal{P}(x,\aloc(x))} \E_{\PP_0}[r(x,\aloc(x),X_1)+\alpha V^{\text{epoch}-1}(X_1)]=\mathcal{T}V^{\text{epoch}-1}(x)$, that was computed in the previous step, for all states $x$ from a batch of sampled values.}
Sample Batch of states $(x_i)_{i=1,\dots,B}\subseteq \Omloc$;\\
Minimize
\[
\sum_{i=1}^B\left(V^{\text{epoch}}(x_i)-\widehat{\mathcal{T}V}(x_i)\right)^2
\]
with respect to parameters from $V^{\text{epoch}}$.
}

}
\Output{Neural network $V^{E}$ approximating the optimal value function;\\ Neural network $\aloc$ approximating the optimal one-step policy;}
 \caption{Value Iteration}\label{algo_reinforcement_value_iteration}
\end{algorithm}

\subsection{Numerical Experiments}
In the sequel we solve the portfolio optimization  problem that was discussed in Section~\ref{sec_portfolio_optimization_setting} by applying the numerical method based on Theorem~\ref{thm_main_result} that is elaborated in Appendix~\ref{sec_numerics} to real financial data. In particular, we compare the different approaches to capture distributional uncertainty outlined in Section~\ref{sec_portfolio_wasserstein} and \ref{sec_portfolio_knightian}, respectively, and evaluate how these approaches perform under different market scenarios.

\subsubsection{Implementation}\label{sec_implementation}
To apply  the numerical method from Appendix~\ref{sec_numerics}, we use the following hyperparameters: 
Number of measures $N_{\mathcal{P}}= 10$; Batch size $B= 2^{8}$; Monte-Carlo sample                                    size $N_{\operatorname{MC}}=2^3$; Discount factor $\alpha = 0.45$; Number of epochs $E=50$; number of iterations for $a$: $\operatorname{Iter}_a=10$; number of iterations for $v$: $\operatorname{Iter}_v=10$. The neural networks that approximate $a$ and $v$ constitute of $2$ layers with $128$ neurons each possessing \emph{ReLu} activation functions in each layer, except for the output layer of $a$ which possesses a \emph{tanh} activation function in order to constraint the output. The learning rate used to optimize the networks $a$ and $v$ when applying the \emph{Adam} optimizer (\cite{kingma2014adam}) is $0.001$. Further details of the implementation can be found under \href{https://github.com/juliansester/Robust-Portfolio-Optimization}{https://github.com/juliansester/Robust-Portfolio-Optimization}.

\subsubsection{Data}
To train and test the performance of the portfolio optimization approach, we consider the price evolution of $d=5$ constituents\footnote{The constituents are Apple, Microsoft, Google, Ebay, and Amazon.} of the $S\&P~500$ between $2010$ and $2021$. We consider a \emph{lookback period} of $m=10$ days, i.e., the prediction of the optimal trading execution relies on the previously realized $m=10$ returns. { This choice of $m$ is in line with empirical findings showing that most of the autocorrelation of daily stock returns is contained in the past $10$ stock returns (compare, e.g., \cite[Table 3.1]{ding1993long}). Moreover, $m=10$ also provides a numerically tractable number of returns that can be taken into account without unnecessarily bloating the state space.} We split the data into a training period ranging from January $2010$ until September $2018$, and three different testing periods thereafter. The normalized evolution of the asset values is depicted in Figure~\ref{fig_train_test_split}.
\begin{center}
\begin{figure}[h!]
\includegraphics[scale=0.5]{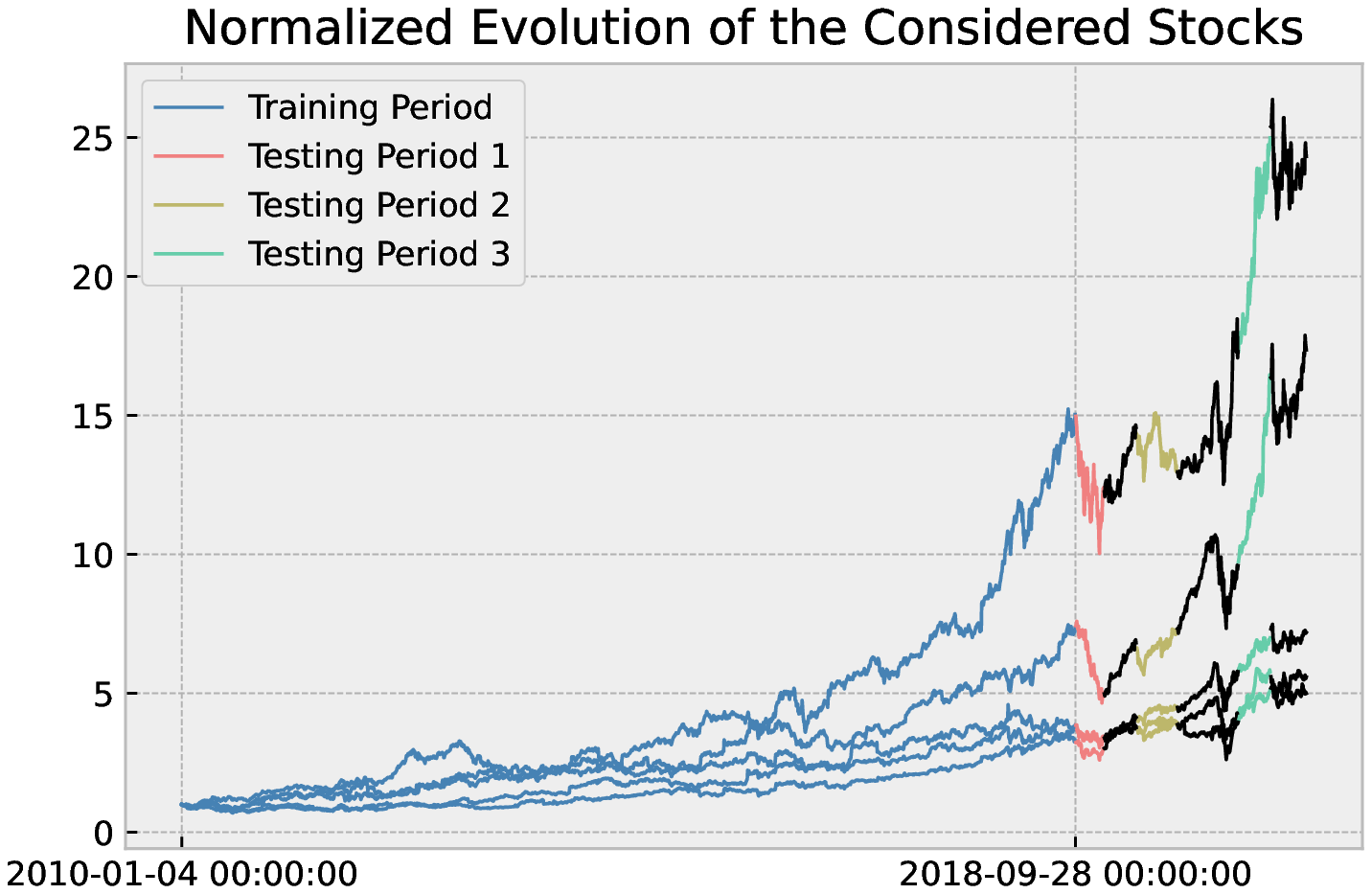}
\includegraphics[scale=0.5]{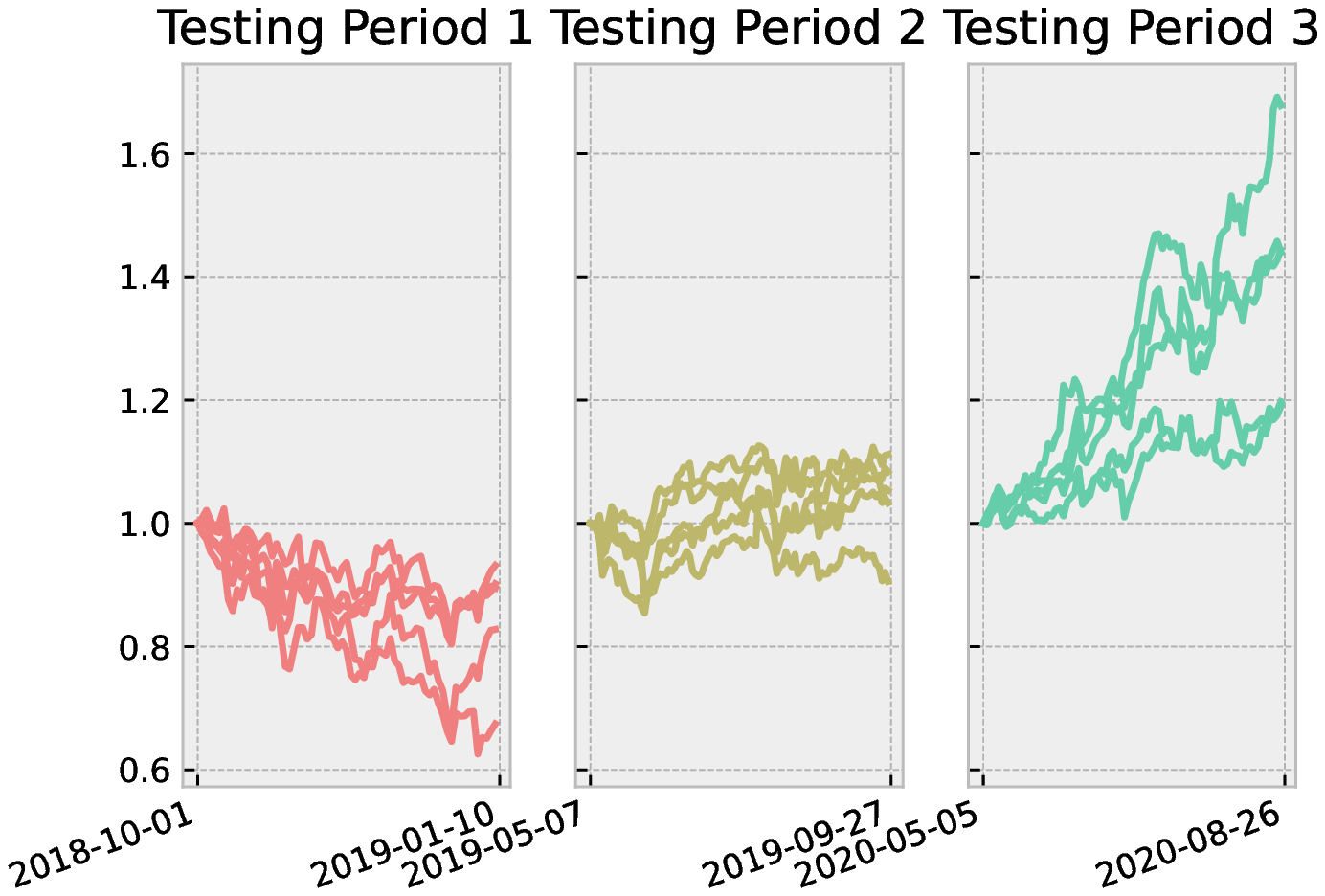}
\caption{The left panel of the figure shows the normalized evolution (with initial value $= 1$) of the stocks of $d=5$ constituents of the $S\&P~500$. Further, we divide the data into a training phase (blue) and three testing periods thereafter that are highlighted with different colors.\\
The right panel of the figure shows the normalized evolution (with initial value $= 1$ for each of the assets) of the considered stocks in the three testing periods.}\label{fig_train_test_split}
\end{figure}
\end{center}
The testing periods are illustrated in detail in the right panel of Figure~\ref{fig_train_test_split}, and they comprise three different market scenarios. While the first testing period covers an overall declining market phase, the second testing period is a volatile period without a clear trend. Eventually, the third period is a bullish market with a strong upward trend.

\subsubsection{Results}

{ We consider the reward function from \eqref{eq_def_c_portfolio} with the particular choices $\lambda =0$ and $\lambda = \tfrac{1}{2}$, resprectively. We recall that choosing $\lambda>0$ penalizes a large variance of the cumulated one-period gains. We follow the discussion from \cite[Section 4.2]{boyd2017multi} and choose in the following $\lambda =\tfrac{1}{2}$ and compare the results with the case $\lambda = 0$. Note that the parameter $\lambda$ could be adjusted for more risk-seeking or more risk-averse agents.} 
In Table~\ref{tbl_period_1}, \ref{tbl_period_1_lam12},~\ref{tbl_period_2}, \ref{tbl_period_2_lam_half}, \ref{tbl_period_3} and ~\ref{tbl_period_3_lam_half}, respectively, we depict the results of the numerical method from { Section}~\ref{sec_numerics} applied to the presented data in the three testing periods with different radii $\varepsilon$ for both of the considered approaches to define ambiguity sets, see Section~\ref{sec_portfolio_wasserstein} and \ref{sec_portfolio_knightian}. Note also that we consider different ranges of values for $\epsilon$ for the two considered approaches. In Figure~\ref{fig_trades_period_1},~\ref{fig_trades_period_2}, and ~\ref{fig_trades_period_3}, 
 we depict the cumulated trading profits of the trained strategies in the respective 
testing periods.
\newpage
\begin{center}
\textbf{Testing Period 1}
\end{center}

\begin{center}
\begin{figure}[h!]
\begin{subfigure}[b]{0.9\textwidth}
\includegraphics[scale=0.5]{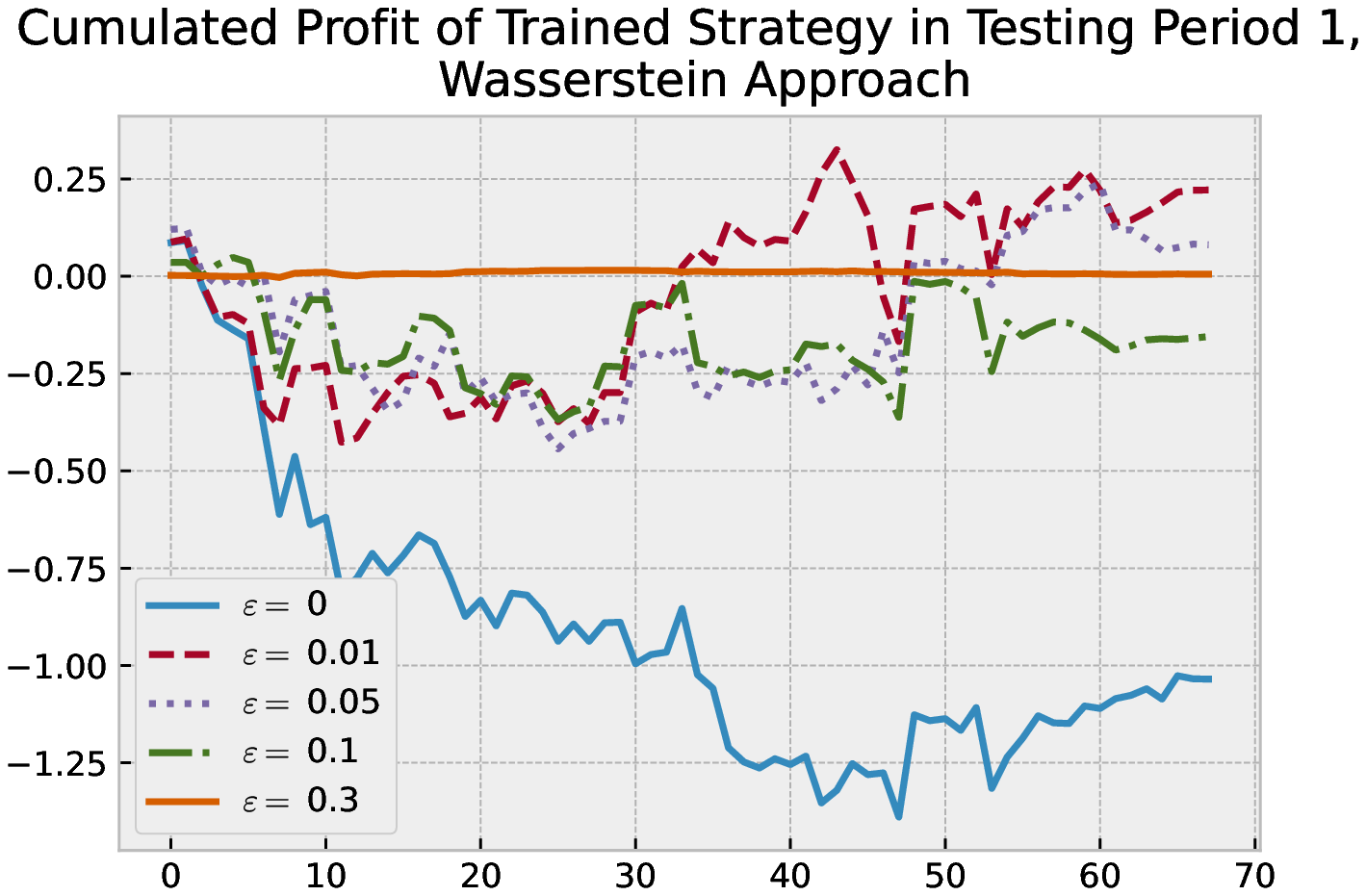}
\includegraphics[scale=0.5]{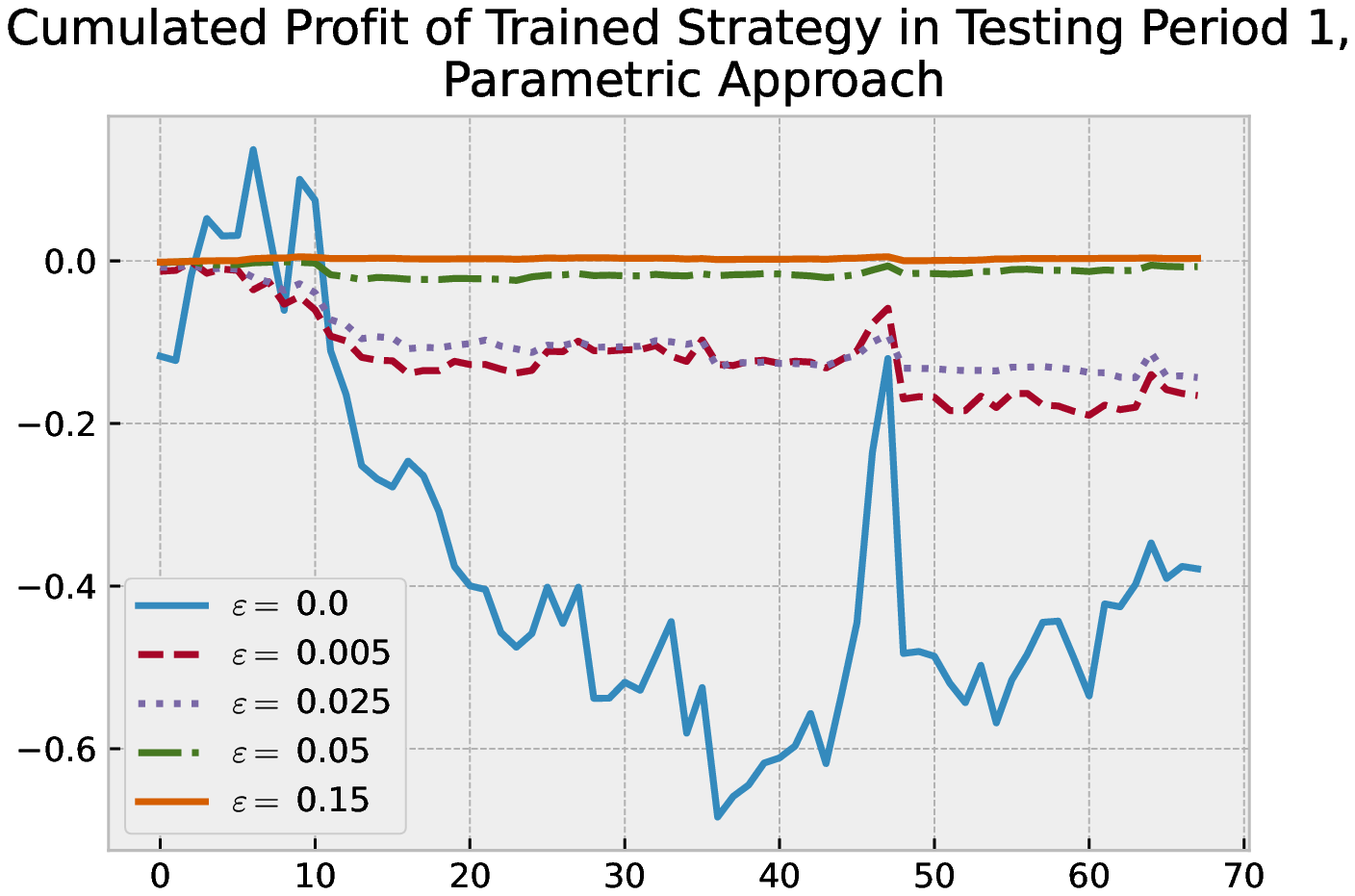}
         \caption{$\lambda = 0$}
         \label{fig_trades_period_1_lam_0}
 \end{subfigure}
 \end{figure}

 \begin{figure}[h!]
  \ContinuedFloat
\begin{subfigure}[b]{0.9\textwidth}
\includegraphics[scale=0.5]{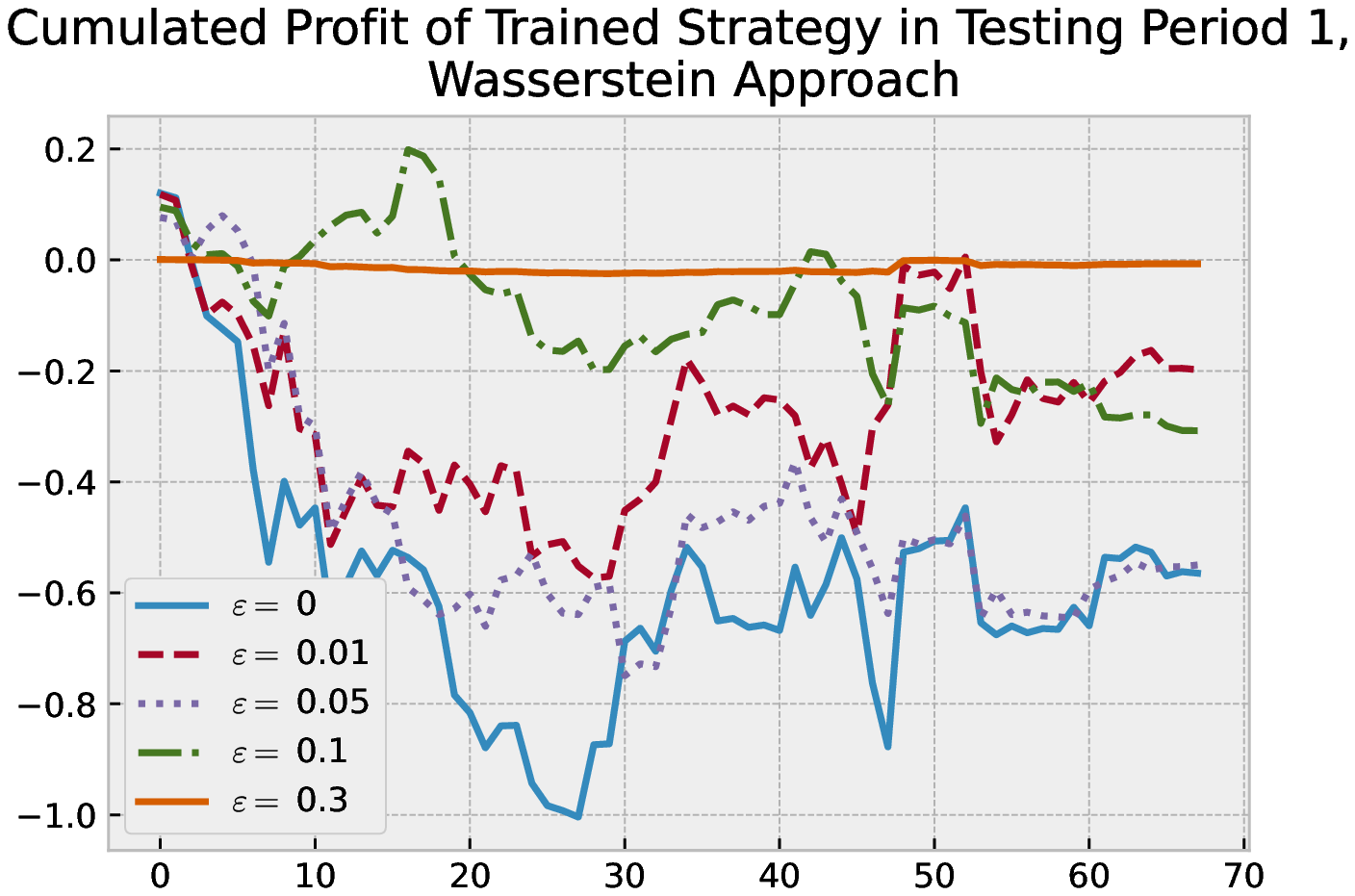}
\includegraphics[scale=0.5]{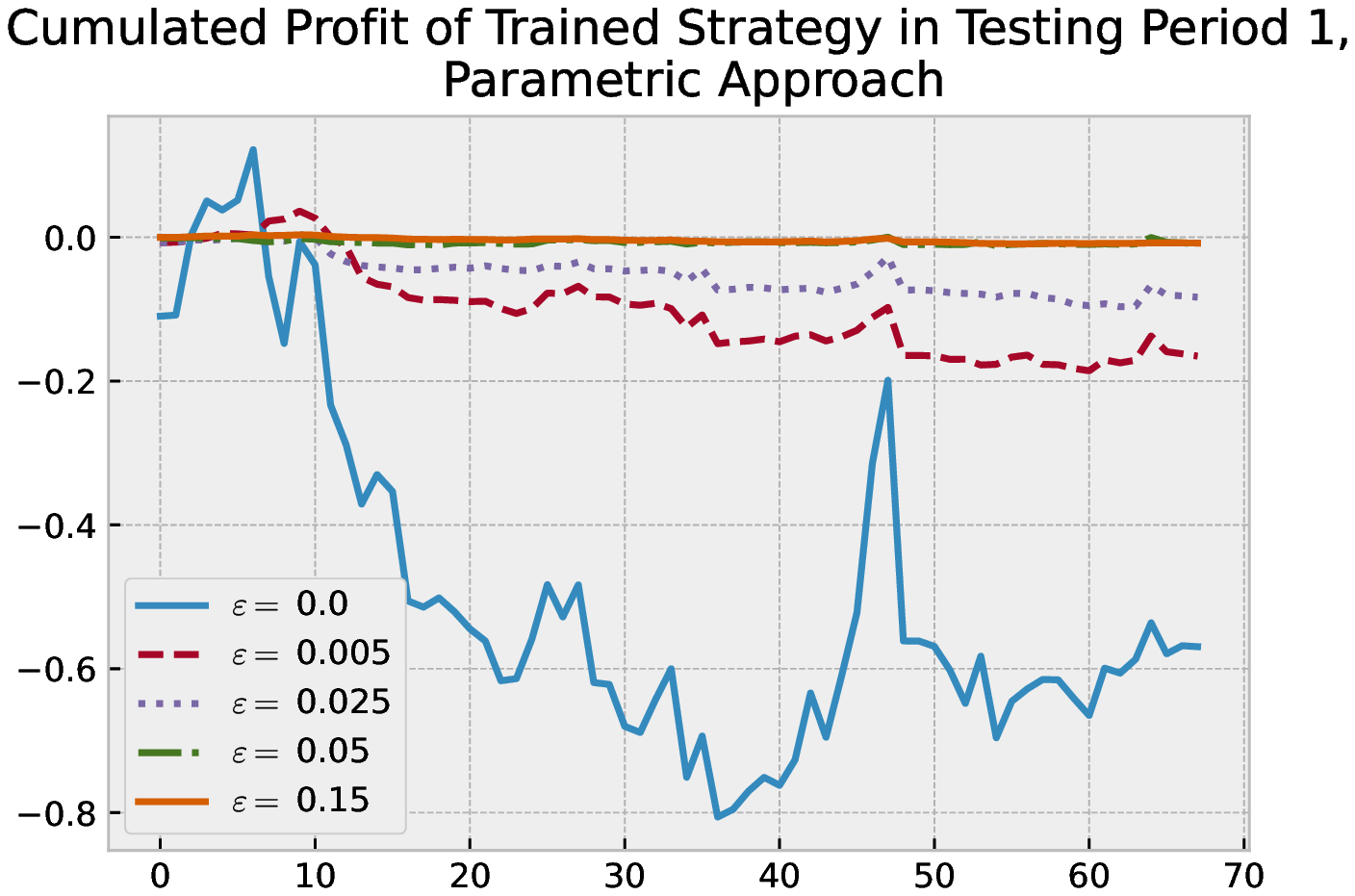}
                  \caption{$\lambda = \frac{1}{2}$}\label{fig_trades_period_1_lam_12}
 \end{subfigure}
\caption{The figure shows the cumulated training profit of the trained strategies in testing period 1 { where in Figure~(A) we report the results when setting $\lambda =0$ in \eqref{eq_def_c_portfolio} and  where in Figure~(B) we set $\lambda =\frac{1}{2}$} . The left panel { of both figures} illustrates the profit when applying a Wasserstein-ball approach, whereas the right panel illustrates the profit under a parametric approach. { Note that the x-axis describes the trading days, whereas the y-axis describes the cumulated profit.}}\label{fig_trades_period_1}
\end{figure}
\end{center}

\begin{table}[h!]
\begin{center}\resizebox{\linewidth}{!}{
\begin{tabular}{lccccc} \toprule
 & Overall Profit &Average Profit & $\%$ of profitable trades & Sharpe Ratio & Sortino Ratio\\
\midrule
\textbf{Wasserstein Approach, $\lambda = 0$} \\

$\varepsilon = 0$ &-1.035019 	&-0.015221 	&47.06 	&-0.172211 	&-0.207249 \\
$\varepsilon = 0.01$ &0.221663 	&0.003260 	&58.82 	&0.035932 	&0.052785\\
$\varepsilon = 0.05$ &	0.080857 	&0.001189 	&55.88 	&0.015079 	&0.022100\\ 
$\varepsilon = 0.1$ &-0.154961 	&-0.002279 	&44.12 	&-0.028555 	&-0.041425\\ 
$\varepsilon = 0.3$ &0.005442 	&0.000080 	&48.53 	&0.035702 	&0.055826
\\ 
\midrule  

\textbf{Parametric Approach , $\lambda = 0$} &  \\

$\varepsilon = 0$ &-0.378669 	&-0.005569 	&50.00 	&-0.067022 	&-0.083380 \\
$\varepsilon = 0.005$ &-0.165537 	&-0.002434 	&42.65 	&-0.127499 	&-0.146742\\
$\varepsilon = 0.025$ &-0.143357 	&-0.002108 	&38.24 	&-0.198093 	&-0.225493\\ 
$\varepsilon = 0.05$ &	-0.006977 	&-0.000103 	&50.00 	&-0.035993 	&-0.043762\\ 
$\varepsilon = 0.15$ &0.003136 	&0.000046 	&58.82 	&0.053274 	&0.066614
\\ 
 \bottomrule 
 \end{tabular}}
\caption{The table shows the results of the Wasserstein-ball approach (Section~\ref{sec_portfolio_wasserstein}) and the parametric approach (Section~\ref{sec_portfolio_knightian}) in the first testing period { where we set $\lambda =0$ in \eqref{eq_def_c_portfolio}}.}
\label{tbl_period_1}
\end{center}
\end{table}

\begin{table}[h!]
\begin{center}\resizebox{\linewidth}{!}{
{ 
\begin{tabular}{lccccc} \toprule
 & Overall Profit &Average Profit & $\%$ of profitable trades & Sharpe Ratio & Sortino Ratio\\
 \textbf{Wasserstein Approach, $\lambda = \tfrac{1}{2}$} \\

$\varepsilon = 0$ &-0.564527 	&-0.008302 	&44.12 	&-0.087936 	&-0.120798 \\
$\varepsilon = 0.01$ &-0.197820 	&-0.002909 	&47.06 	&-0.035237 	&-0.048919\\
$\varepsilon = 0.05$ &-0.549781 	&-0.008085 	&50.00 	&-0.113071 	&-0.137965\\ 
$\varepsilon = 0.1$ &-0.307764 	&-0.004526 	&42.65 	&-0.084528 	&-0.112735\\ 
$\varepsilon = 0.3$ &-0.007034 	&-0.000103 	&42.65 	&-0.034035 	&-0.065102
\\ 
\midrule  

\textbf{Parametric Approach, $\lambda = \tfrac{1}{2}$} &  \\

$\varepsilon = 0$ &-0.569345 	&-0.008373 	&47.06 	&-0.097639 	&-0.117886 \\
$\varepsilon = 0.005$ &-0.165356 	&-0.002432 	&44.12 	&-0.163848 	&-0.188656\\
$\varepsilon = 0.025$ &-0.083480 	&-0.001228 	&41.18 	&-0.124146 	&-0.151676\\ 
$\varepsilon = 0.05$ &-0.008617 	&-0.000127 	&51.47 	&-0.050311 	&-0.066198\\ 
$\varepsilon = 0.15$ &-0.008218 	&-0.000121 	&38.24 	&-0.129144 	&-0.152631
\\ 
 \bottomrule
 
\end{tabular}}}
\caption{ { The table shows the results of the Wasserstein-ball approach (Section~\ref{sec_portfolio_wasserstein}) and the parametric approach (Section~\ref{sec_portfolio_knightian}) in the first testing period { where we set $\lambda = \tfrac{1}{2}$ in \eqref{eq_def_c_portfolio}}.}}
\label{tbl_period_1_lam12}
\end{center}
\end{table}

\begin{center}
\newpage
\textbf{Testing Period 2}
\end{center}
\begin{center}

\begin{figure}[h!]
\begin{subfigure}[b]{0.9\textwidth}
\includegraphics[scale=0.5]{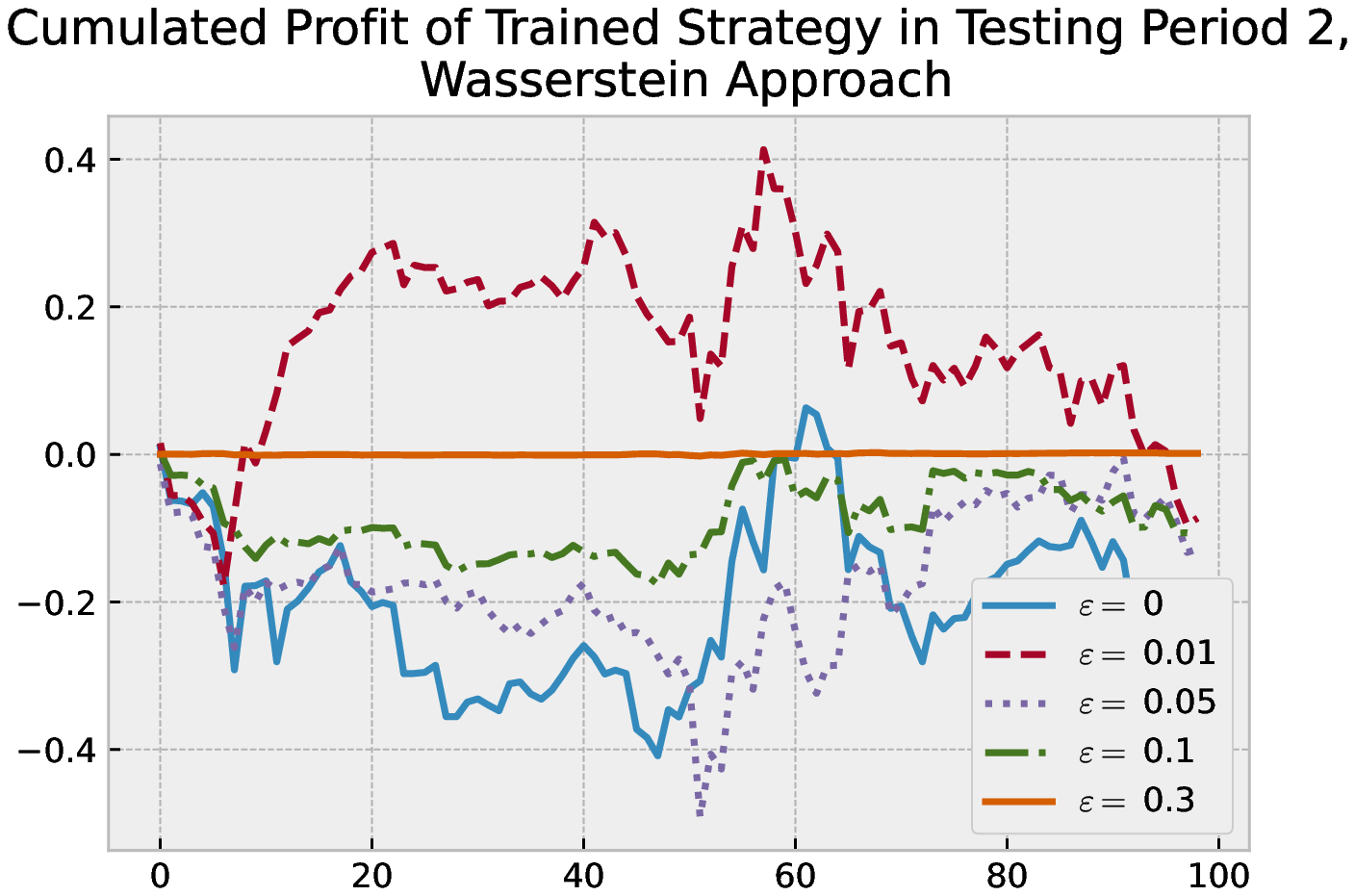}
\includegraphics[scale=0.5]{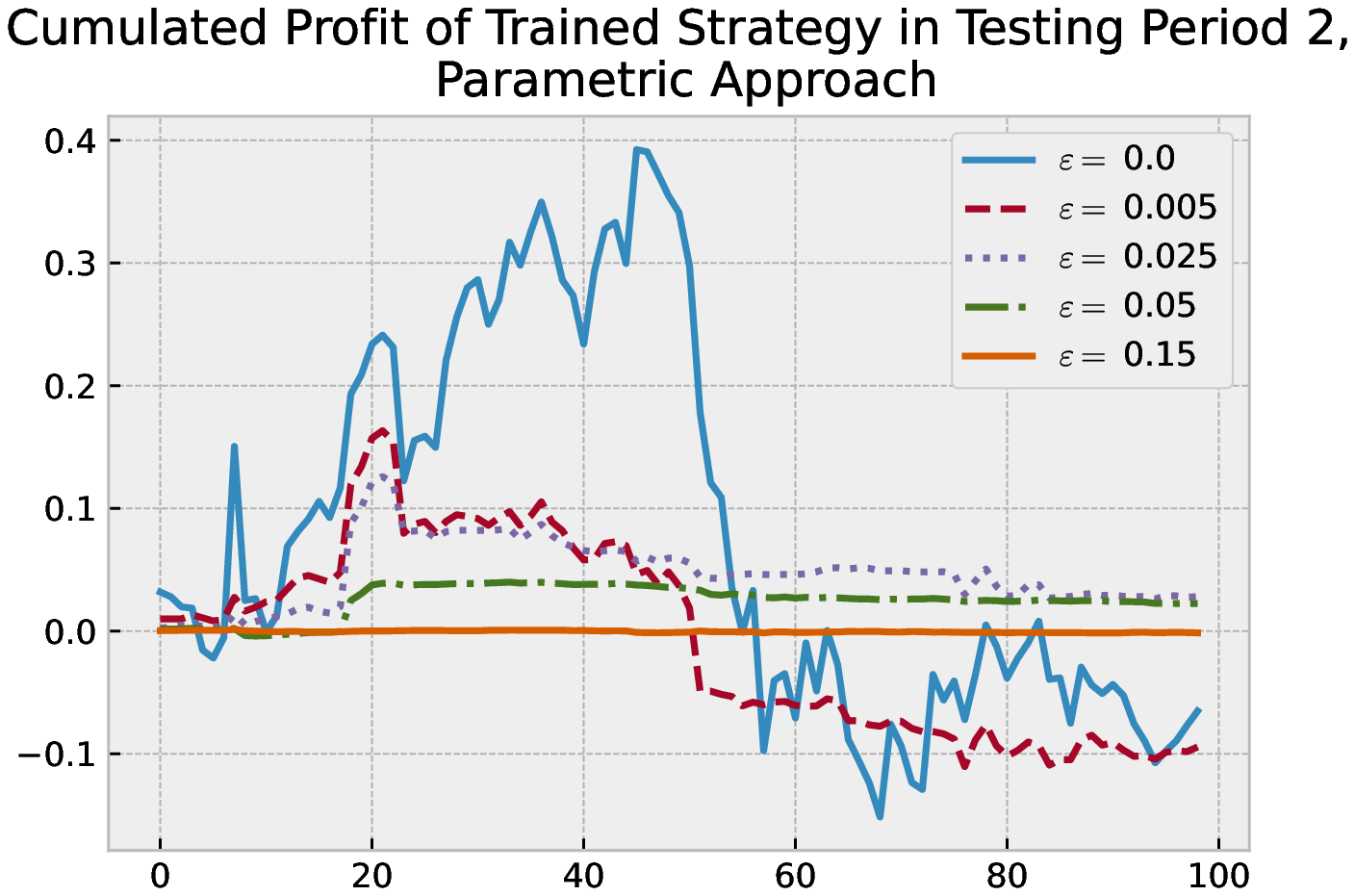}
         \caption{$\lambda = 0$}
         \label{fig_trades_period_2_lam_0}
 \end{subfigure}
 \end{figure}

 \begin{figure}[h!]
  \ContinuedFloat
\begin{subfigure}[b]{0.9\textwidth}
\includegraphics[scale=0.5]{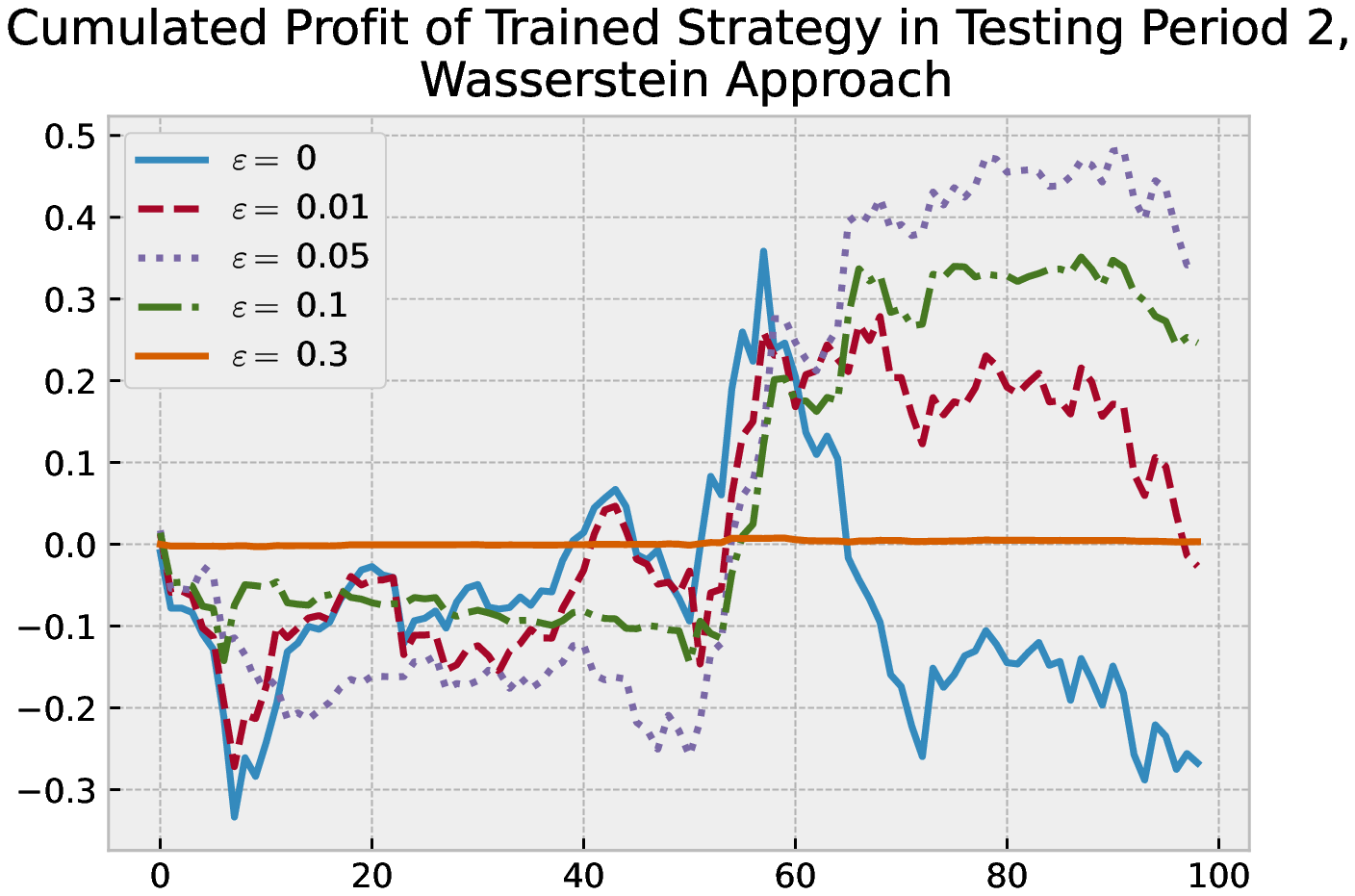}
\includegraphics[scale=0.5]{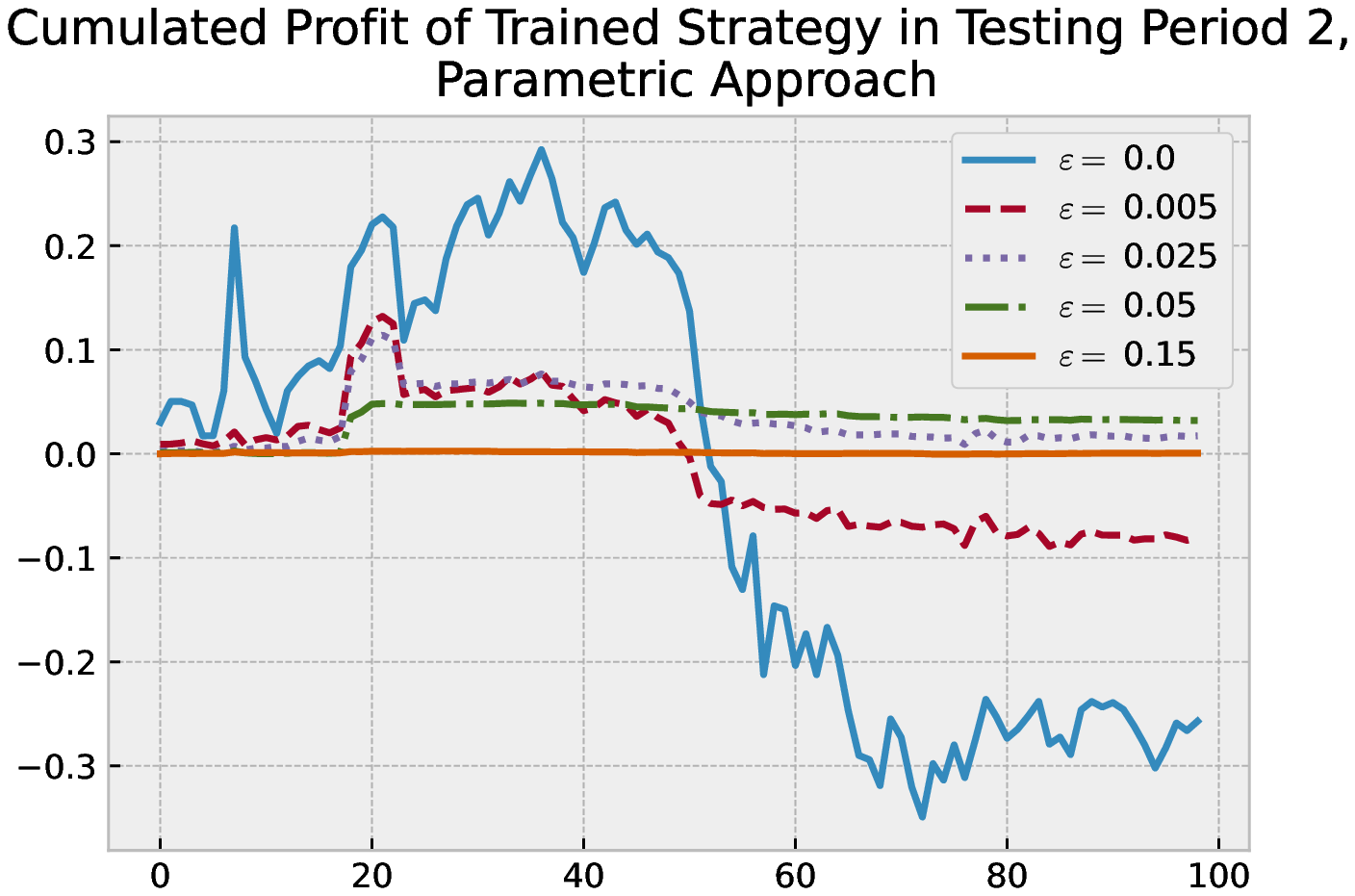}
                  \caption{$\lambda = \frac{1}{2}$}\label{fig_trades_period_2_lam_12}
 \end{subfigure}
\caption{The figure shows the cumulated training profit of the trained strategies in testing period 2 { where in Figure~(A) we report the results when setting $\lambda =0$ in \eqref{eq_def_c_portfolio} and  where in Figure~(B) we set $\lambda =\frac{1}{2}$} . The left panel { of both figures} illustrates the profit when applying a Wasserstein-ball approach, whereas the right panel illustrates the profit under a parametric approach. { Note that the x-axis describes the trading days, whereas the y-axis describes the cumulated profit.}}\label{fig_trades_period_2}
\end{figure}
\end{center}
\begin{table}[h!]
\begin{center}\resizebox{\linewidth}{!}{
\begin{tabular}{lccccc} \toprule
 & Overall Profit &Average Profit & $\%$ of profitable trades & Sharpe Ratio & Sortino Ratio\\
\midrule
\textbf{Wasserstein Approach, $\lambda = 0$} \\

$\varepsilon = 0$ &-0.319652 	&-0.003229 	&51.52 	&-0.068955 	&-0.091312 \\
$\varepsilon = 0.01$ &	-0.086229 	&-0.000871 	&56.57 	&-0.018503 	&-0.025662\\
$\varepsilon = 0.05$ &-0.128121 	&-0.001294 	&52.53 	&-0.031648 	&-0.044057\\ 
$\varepsilon = 0.1$ &-0.108500 	&-0.001096 	&50.51 	&-0.054067 	&-0.074651\\ 
$\varepsilon = 0.3$ &0.001415 	&0.000014 	&52.53 	&0.027997 	&0.040511\\ 
\midrule  

\textbf{Parametric Approach, $\lambda = 0$} &  \\

$\varepsilon = 0$ &-0.065187 	&-0.000658 	&48.48 	&-0.014956 	&-0.020914 \\
$\varepsilon = 0.005$ &-0.094396 	&-0.000953 	&50.51 	&-0.063190 	&-0.079788 \\
$\varepsilon = 0.025$ &0.027990 	&0.000283 	&55.56 	&0.030187 	&0.053250\\ 
$\varepsilon = 0.05$ &0.022454 	&0.000227 	&49.49 	&0.079555 	&0.277870
\\ 
$\varepsilon = 0.15$ &	-0.001429 	&-0.000014 	&49.49 	&-0.048469 	&-0.061703\\ 
 \bottomrule
 \end{tabular}}
\caption{The table shows the results of the Wasserstein-ball approach (Section~\ref{sec_portfolio_wasserstein}) and the parametric approach (Section~\ref{sec_portfolio_knightian}) in the second testing period  { where we set $\lambda =0$ in \eqref{eq_def_c_portfolio}}.}
\label{tbl_period_2}
\end{center}
\end{table}
\begin{table}[h!]
\begin{center}\resizebox{\linewidth}{!}{
{ 
\begin{tabular}{lccccc} \toprule
 & Overall Profit &Average Profit & $\%$ of profitable trades & Sharpe Ratio & Sortino Ratio\\
\midrule
\textbf{Wasserstein Approac, $\lambda = \tfrac{1}{2}$} \\

$\varepsilon = 0$ &-0.267468 	&-0.002702 	&47.47 	&-0.058955 	&-0.082091 \\
$\varepsilon = 0.01$ &-0.026789 	&-0.000271 	&53.54 	&-0.006770 	&-0.009480\\
$\varepsilon = 0.05$ &0.344366 	&0.003478 	&52.53 	&0.099127 	&0.180309\\ 
$\varepsilon = 0.1$ &0.246400 	&0.002489 	&43.43 	&0.092752 	&0.181330\\ 
$\varepsilon = 0.3$ &0.003219 	&0.000033 	&49.49 	&0.041700 	&0.076582\\ 
\midrule  

\textbf{Parametric Approach, $\lambda = \tfrac{1}{2}$} &  \\

$\varepsilon = 0$ &-0.257053 	&-0.002596 	&47.47 	&-0.063529 	&-0.084109 \\
$\varepsilon = 0.005$ &-0.080422 	&-0.000812 	&50.51 	&-0.064610 	&-0.086200 \\
$\varepsilon = 0.025$ &0.017278 	&0.000175 	&45.45 	&0.020485 	&0.035732\\ 
$\varepsilon = 0.05$ &0.032120 	&0.000324 	&44.44 	&0.090080 	&0.578979\\ 
$\varepsilon = 0.15$ &0.000581 	&0.000006 	&51.52 	&0.022708 	&0.037737\\ 
 \bottomrule
 \end{tabular}}}
\caption{{ The table shows the results of the Wasserstein-ball approach (Section~\ref{sec_portfolio_wasserstein}) and the parametric approach (Section~\ref{sec_portfolio_knightian}) in the second testing period { where we set $\lambda =\frac{1}{2}$ in \eqref{eq_def_c_portfolio}}.}}
\label{tbl_period_2_lam_half}
\end{center}
\end{table}
\begin{center}
\newpage
\textbf{Testing Period 3}
\end{center}
\begin{center}

\begin{figure}[h!]
\begin{subfigure}[b]{0.9\textwidth}
\includegraphics[scale=0.5]{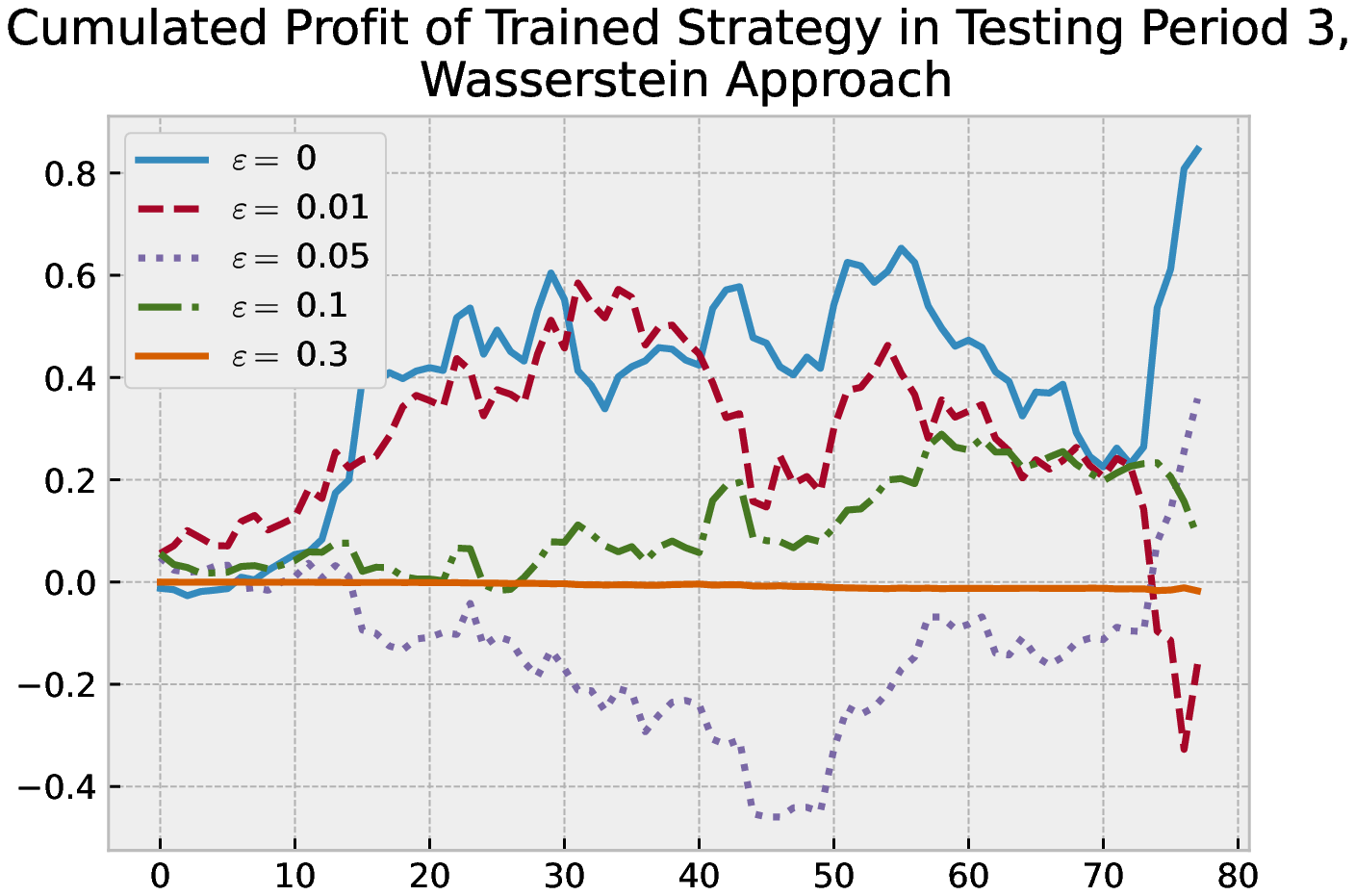}
\includegraphics[scale=0.5]{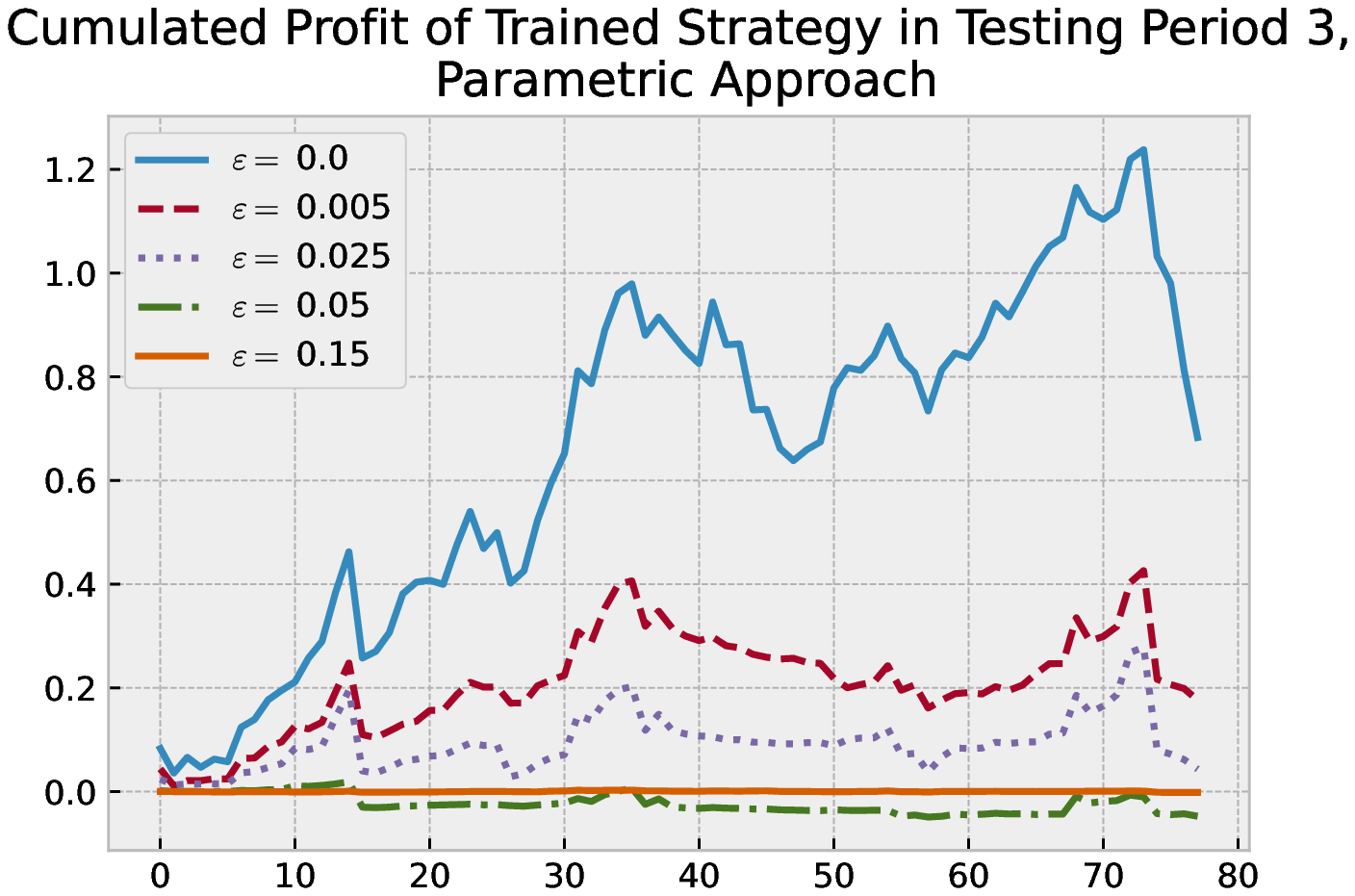}
         \caption{$\lambda = 0$}
         \label{fig_trades_period_3_lam_0}
 \end{subfigure}
  \end{figure}

 \begin{figure}[h!]
  \ContinuedFloat
\begin{subfigure}[b]{0.9\textwidth}
\includegraphics[scale=0.5]{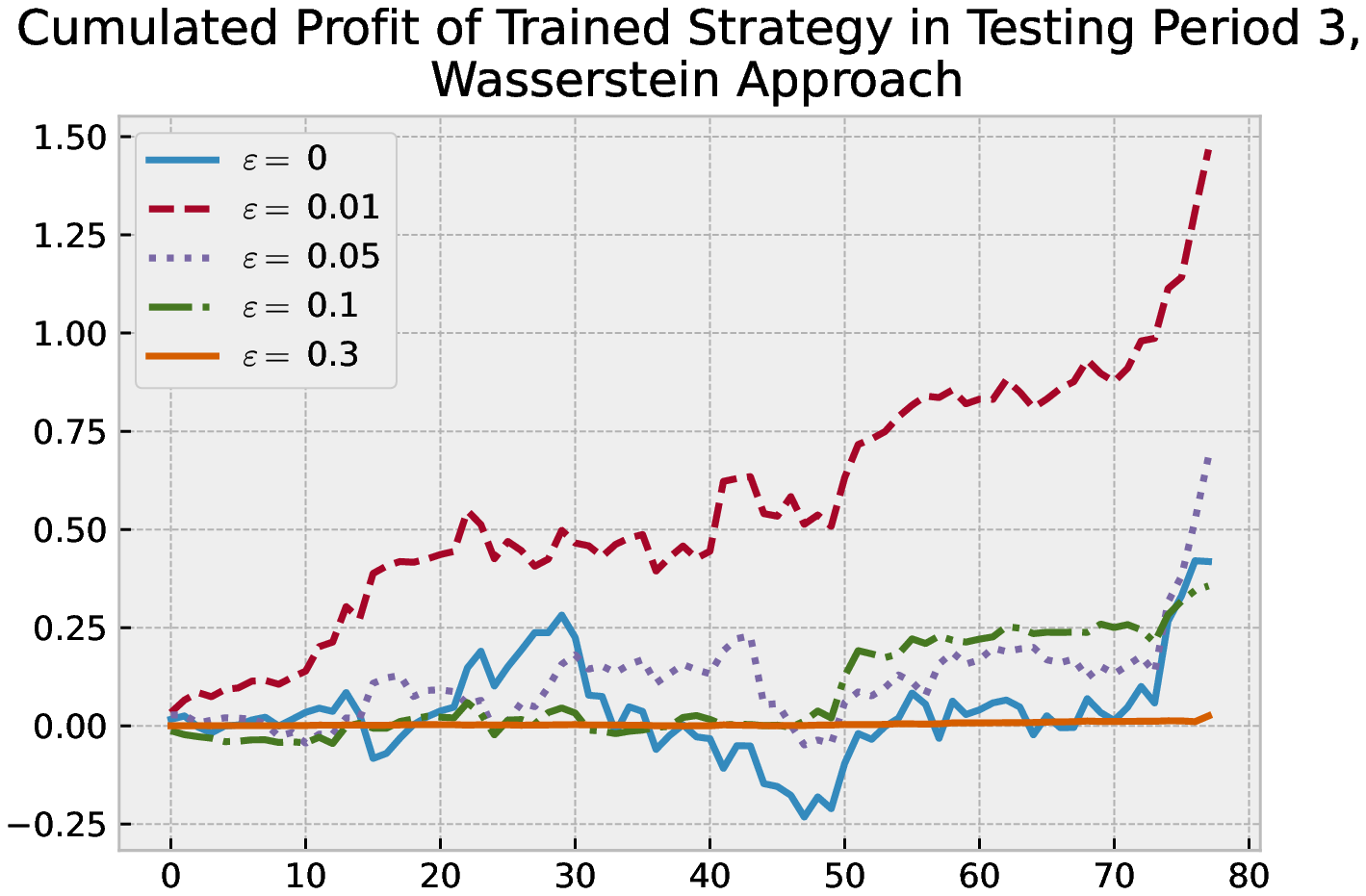}
\includegraphics[scale=0.5]{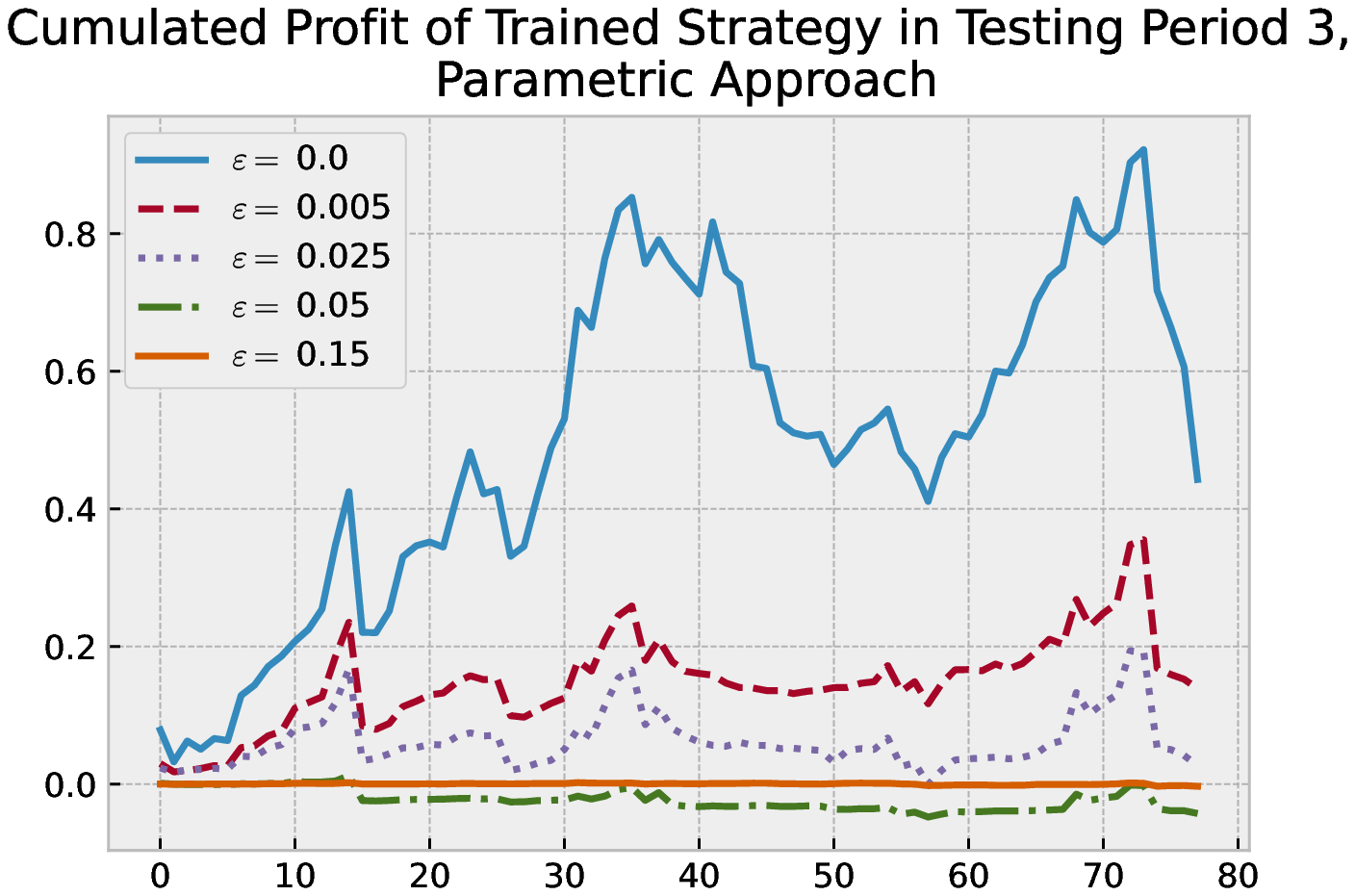}
                  \caption{$\lambda = \frac{1}{2}$}\label{fig_trades_period_3_lam_12}
 \end{subfigure}
\caption{The figure shows the cumulated training profit of the trained strategies in testing period 3 { where in Figure~(A) we report the results when setting $\lambda =0$ in \eqref{eq_def_c_portfolio} and  where in Figure~(B) we set $\lambda =\frac{1}{2}$} . The left panel { of both figures} illustrates the profit when applying a Wasserstein-ball approach, whereas the right panel illustrates the profit under a parametric approach. { Note that the x-axis describes the trading days, whereas the y-axis describes the cumulated profit.}}\label{fig_trades_period_3}
\end{figure}
\end{center}
\begin{table}[h!]
\begin{center}\resizebox{\linewidth}{!}{
\begin{tabular}{lccccc} \toprule
 & Overall Profit &Average Profit & $\%$ of profitable trades & Sharpe Ratio & Sortino Ratio\\
\midrule
\textbf{Wasserstein Approach, $\lambda = 0$} \\

$\varepsilon = 0$ &0.845952 	&0.010846 	&52.56 	&0.166630 	&0.335301
 \\
$\varepsilon = 0.01$ &-0.160733 	&-0.002061 	&48.72 	&-0.031032 	&-0.040478
\\
$\varepsilon = 0.05$ &	0.358808 	&0.004600 	&52.56 	&0.096431 	&0.154537
\\ 
$\varepsilon = 0.1$ &0.087670 	&0.001124 	&52.56 	&0.037463 	&0.053469
\\ 
$\varepsilon = 0.3$ &-0.017214 	&-0.000221 	&43.59 	&-0.199209 	&-0.231336
\\ 
\midrule  

\textbf{Parametric Approach, $\lambda = 0$} &  \\

$\varepsilon = 0$ &0.682794 	&0.008754 	&64.10 	&0.124028 	&0.170558 \\
$\varepsilon = 0.005$ &0.176339 	&0.002261 	&61.54 	&0.054989 	&0.069144\\
$\varepsilon = 0.025$ &0.043324 	&0.000555 	&64.10 	&0.014557 	&0.017213\\ 
$\varepsilon = 0.05$ &	-0.047627 	&-0.000611 	&52.56 	&-0.064649 	&-0.075748\\ 
$\varepsilon = 0.15$ &-0.001740 	&-0.000022 	&51.28 	&-0.039336 	&-0.049270 \\
 \bottomrule
\end{tabular}}
\caption{The table shows the results of the Wasserstein-ball approach (Section~\ref{sec_portfolio_wasserstein}) and the parametric approach (Section~\ref{sec_portfolio_knightian}) in the third testing period  { where we set $\lambda =0$ in \eqref{eq_def_c_portfolio}}.}
\label{tbl_period_3}
\end{center}
\end{table}
\begin{table}[h!]
\begin{center}\resizebox{\linewidth}{!}{{ 
\begin{tabular}{lccccc} \toprule
 & Overall Profit &Average Profit & $\%$ of profitable trades & Sharpe Ratio & Sortino Ratio\\
\midrule
\textbf{Wasserstein Approach, $\lambda = \tfrac{1}{2}$} \\

$\varepsilon = 0$ &0.418353 	&0.005364 	&56.41 	&0.094879 	&0.144220\\
$\varepsilon = 0.01$ &1.467143 	&0.018810 	&69.23 	&0.339216 	&0.793886
\\
$\varepsilon = 0.05$ &0.685678 	&0.008791 	&58.97 	&0.174976 	&0.310182
\\ 
$\varepsilon = 0.1$ &0.356257 	&0.004567 	&52.56 	&0.194070 	&0.417113
\\ 
$\varepsilon = 0.3$ &0.025321 	&0.000325 	&57.69 	&0.179535 	&0.757265
\\ 
\midrule  

\textbf{Parametric Approach, $\lambda = \tfrac{1}{2}$} &  \\

$\varepsilon = 0$ &0.442670 	&0.005675 	&58.97 	&0.086985 	&0.117434\\
$\varepsilon = 0.005$ &0.138669 	&0.001778 	&62.82 	&0.048293 	&0.058150\\
$\varepsilon = 0.025$ &0.026678 	&0.000342 	&61.54 	&0.011119 	&0.013274\\ 
$\varepsilon = 0.05$ &-0.042799 	&-0.000549 	&60.26 	&-0.071591 	&-0.083011\\ 
$\varepsilon = 0.15$ &-0.003295 	&-0.000042 	&56.41 	&-0.056704 	&-0.065249 \\
 \bottomrule
\end{tabular}}}
\caption{{ The table shows the results of the Wasserstein-ball approach (Section~\ref{sec_portfolio_wasserstein}) and the parametric approach (Section~\ref{sec_portfolio_knightian}) in the third testing period { where we set $\lambda =\frac{1}{2}$ in \eqref{eq_def_c_portfolio}}.}}
\label{tbl_period_3_lam_half}
\end{center}
\end{table}
{ 
\begin{rem}[Computational Time]
It turns out that training  a robust trading strategy in the Wasserstein approach is slightly faster than training a strategy in the parametric approach: for training $50$ epochs with the parameters specified in Section~\ref{sec_implementation} the Wasserstein approach needs $15.07$ minutes whereas the parameteric approach needs $23.12$ minutes on a standard computer.\footnote{  We used for the computations a \emph{Gen Intel(R) Core(TM) i7-1165G7, 2.80 GHz} processor with $40$ GB RAM.}
\end{rem}}

\subsubsection{Discussion of the Results}

Note that in contrast to the third bullish testing period, the first and second testing periods comprise scenarios that did not occur in similar form during the training period. Hence, it cannot be guaranteed that a non-robust trading strategy that was trained in periods where such scenarios never occurred can trade profitably during testing period 1 and testing period 2.
Indeed, as the numerical results reveal, applying a trained non-robust trading strategy not respecting any distributional ambiguity (i.e. $\varepsilon = 0$) leads to significant losses in the adverse scenarios considered in testing period 1 and 2, while robust trading strategies, which did encounter also adverse scenarios during training, clearly outperform the non-robust strategy for both considered approaches, namely the Wasserstein-ball approach and the parametric approach. 

As the third testing period comprises a bullish market period, occurring in similar form in the training data, the non-robust approach turns out to be the most profitable approach in this period, since it is the approach that is best adjusted to this scenario. However, when choosing the right level of $\varepsilon$, the robust approach still can trade profitably in this period.

In general, it turns out to be important to correctly identify the appropriate \emph{size} of the ambiguity set (here encoded by the radius $\varepsilon$) { which needs to be adjusted for each optimzation under consideration}. Due to the formulation of the robust optimization problem as a worst-case approach, respecting for more distributional ambiguity means to consider more \emph{bad} scenarios, and therefore may eventually result in a more careful, less volatile trading behavior with smaller returns, as it clearly can be seen in the case $\varepsilon=0.3$ for the Wasserstein-ball approach and in the case $\varepsilon=0.15$ for the parametric approach, respectively.
In contrast, insufficiently accounting for  uncertainty comes at the cost of not being well equipped when adverse scenarios occur, an observation that was already made in similar empirical studies that use different approaches to solve robust optimization problems, compare, e.g., \cite{lutkebohmert2021robust} and \cite{neufeld2022detecting}. 
Hence, choosing an \emph{intermediate} level of ambiguity seems to be an appropriate choice.
Supporting this rationale, our numerical results show that { for both cases $\lambda =0$ and $\lambda = \frac{1}{2}$ }choosing a level of $\varepsilon=0.05$ in the Wasserstein-ball approach and of $\varepsilon=0.025$ in the parametric approach, respectively, lead to trading strategies that outperform in testing periods $1$ and $2$ the respective non-robust trading strategies and still lead to profits in testing period 3.

 This provides evidence that taking into account distributional uncertainty may be of particular importance in volatile and crisis-like market scenarios which did not occur previously in a similar form.
{ Our results also show that taking into account a penalization of large variances of cumulated gains by setting $\lambda =\frac{1}{2}$ in \eqref{eq_def_c_portfolio} leads to slightly changed outcomes while the above discussed observations remain true.}
 
{ A practitioner may be in the situation to decide whether to use a robust approach or a non-robust approach for portfolio optimization: one can derive from our results that the robust portfolio optimization approach takes into account that an assumed or estimated underlying model may be misspecified, i.e., realized scenarios are not appropriately covered by an assumed model. Hence, if the practitioner sees no risk of misspecification and believes strongly in his model, then it is advisable for him to use a non-robust approach as the empirical results provide evidence that if there is indeed no misspecification, non-robust approaches slightly outperform robust approaches. If however the practitioner wants to take into account the (difficult to determine) risk of having chosen a wrong model and wants to be robust against adverse scenarios, then our results imply that she should rather rely on a distributionally robust approach. }

{
\subsubsection{{ Comparison with other Robust Portfolio Optimization Approaches}}
The literature on robust portfolio optimization is rapidly growing. { We refer to, e.g.,  \cite{blanchet2021distributionally} and \cite{pham2022portfolio} for alternative approaches to robust portfolio optimization, to \cite{bartl2019exponential}, \cite{bartl2021duality}, \cite{bartl2019robust}, \cite{blanchard2018multiple}, \cite{neufeld2018robust}, \cite{neufeld2019nonconcave} for the fast growing literature on the related problem of robust utility optimization in non-dominated settings in discrete time, and to \cite{biagini2017robust}, \cite{chau2019robust}, \cite{denis2013optimal}, \cite{fouque2016portfolio}, \cite{guo2022robust}, \cite{ismail2019robust}, \cite{liang2020robust}, \cite{lin2021optimal}, \cite{lin2020horizon}, \cite{matoussi2012robust}, \cite{neufeld2018robust}, \cite{pham2022portfolio}, \cite{pun2021g}, \cite{tevzadze2013robust}, \cite{yang2019constrained} for robust utility maximization in continuous time.
Let us compare our contribution with articles that recently analyzed how different \emph{choices of (the size of) ambiguity sets} affect the performance of the portfolio optimization problem.

In the distributionally robust optimization approach proposed in \cite{blanchet2021distributionally} the ambiguity set is specifically given by a Wasserstein ambiguity set and their specific objective function relates to Markowitz's mean-variance portfolio optimization problem. Their setting is reduced to this ambiguity specification but therefore allows to apply a duality result which transfers the distributionally robust optimization problem  into a non-robust regularized optimization problem which can be solved explicitly. Moreover, the authors provide a data-driven methodology to determine the \emph{size} of the ambiguity set relying on results from \cite{blanchet2019robust}.

In a recent paper, \cite{pham2022portfolio} study in a continuous-time framework portfolio optimization under drift and correlation ambiguity. Their framework allows to determine different optimal allocation decisions in dependence of the degree of ambiguity specified, where more ambiguity leads to less investment in risky assets. This result can be seen in line with our finding that a higher degree of ambiguity leads to more careful trading decisions.

In the same spirit, the results from \cite{obloj2021distributionally}, where the authors consider a portfolio optimization approach under Wasserstein-ambiguity in an one-period market, imply that the size of the ambiguity set corresponds to the risk attitude of the market participant.

In \cite{du2020new}, the authors also rely their approach on Wasserstein-ambiguity sets and analyze the mean-CVar portfolio optimization problem where CVar refers to the conditional value at risk. Their empirical study that is carried out on Chinese and American stock markets again confirms the previously discussed findings that a higher radius of the Wasserstein-ball corresponds to a higher degree of risk aversion which leads to  different, more careful allocation results.
}
\section{Proof of Theorem~\ref{thm_main_result}} \label{sec_proof_main_results}
\begin{proof}[Proof of Theorem~\ref{thm_main_result}~(i)]
Let $v\in C_p(\Omloc,\R)$. We define the map 
\begin{equation}\label{eq_f_continuous_contraction_proof}
\begin{aligned}
F:\operatorname{Gr} \mathcal{P} = \{(x,a_0,\PP_0)~|~x\in \Omloc,a_0\in A, \PP_0\in  \mathcal{P}(x,a_0)\} &\rightarrow \R\\
(x,a_0,\PP_0) &\mapsto \int_{\Omloc}r(x,a_0,\omega_1)+\alpha  v(\omega_1)\PP_0(\D\omega_1).
\end{aligned}
\end{equation}
We claim that the map $F$ is continuous {(or equivalently sequentially continuous)}. { To this end, let $(x,a_0,\PP_0) \in \operatorname{Gr} \mathcal{P}$ and $(x^{(n)},a_0^{(n)},\PP_0^{(n)}) \subseteq \operatorname{Gr} \mathcal{P}$ be a sequence  with\footnote{{We highlight that $\operatorname{Gr} \mathcal{P} \subseteq \Omloc \times A \times (\mathcal{M}_1(\Omloc),\tau_p)$ is endowed with the product topology, see also Section~\ref{sec_setting}. }}}  $(x^{(n)},a_0^{(n)},\PP_0^{(n)}) \rightarrow (x,a_0,\PP_0)$ for $n \rightarrow \infty$. Then,
\begin{align}
&|F(x^{(n)},a_0^{(n)},\PP_0^{(n)})-F(x,a_0,\PP_0)| \notag \\
\leq &|F(x^{(n)},a_0^{(n)},\PP_0^{(n)})-F(x,a_0,\PP_0^{(n)})| +|F(x,a_0,\PP_0^{(n)})-F(x,a_0,\PP_0)|. \label{eq_F_ineq}
\end{align}
The second summand $|F(x,a_0,\PP_0^{(n)})-F(x,a_0,\PP_0)|$ vanishes for $n \rightarrow \infty$ since the integrand $\omega_1 \mapsto r(x,a_0,\omega_1)+\alpha  v(\omega_1)$ is an element of $C_p(\Omloc,\R)$. For the first summand we obtain, by using Assumption~\ref{asu_2}~(ii), that
 \begin{align*}
  \lim_{n \rightarrow \infty}&\left|F(x^{(n)},a_0^{(n)},\PP_0^{(n)})-F(x,a_0,\PP_0^{(n)})\right|\\
  &\leq    \lim_{n \rightarrow \infty}  \int_{\Omloc} \left|r(x^{(n)},a_0^{(n)},\omega_1)-r(x,a_0,\omega_1)\right|\PP_0^{(n)}(\D\omega_1)\\
 &\leq    \lim_{n \rightarrow \infty}  \int_{\Omloc} L  \cdot { (1+\|\omega_1\|^p) \cdot} \bigg( { \rho_0}\left( \left\|x^{(n)}-x\right\|\right)+{ \rho_A}\left(\left\|a_0^{(n)}-a_0\right\| \right)\bigg)\PP_0^{(n)}(\D\omega_1)\\
 &= \lim_{n \rightarrow \infty}   L \cdot \bigg( { \rho_0}\left( \left\|x^{(n)}-x\right\|\right)+{ \rho_A}\left(\left\|a_0^{(n)}-a_0\right\| \right)\bigg)  \cdot  \int_{\Omloc}  { (1+\|\omega_1\|^p) \cdot} \PP_0(\D\omega_1) =0, 
 \end{align*}
where we use in the last equality that due to the convergence $\PP_0^{(n)} \xrightarrow{\tau_p}  \PP_0$ also the $p$-th moments converge, see e.g. \cite[Definition 6.8.~(i), and Theorem 6.9]{villani2009optimal}.
Thus, $F$ is continuous. Therefore, we may apply Berge's maximum theorem (Theorem~\ref{thm_berge}) and obtain that the map
\begin{equation}\label{eq_continuous_berge}
\begin{aligned}
G:\Omloc \times A  &\rightarrow \R \\
 (x,a_0) &\mapsto \inf_{\PP_0 \in \mathcal{P}(x,a_0)}\int_{\Omloc}r(x,a_0,\omega_1)+\alpha  v(\omega_1)\PP_0(\D\omega_1)
\end{aligned}
\end{equation}
is continuous, and  the set of minimizers is nonempty for each $(x,a_0) \in \Omloc \times A$. According to the measurable maximum theorem {(Theorem~\ref{thm_measurable_maximum})} there exists a measurable selector $\Omloc \times A \ni (x,a_0) \mapsto {\PP_0}^*(x,a_0)\in \mathcal{P}(x,a_0)$, where ${\PP_0}^*(x,a_0)$ minimizes the integral in \eqref{eq_continuous_berge} for each $(x,a_0) \in \Omloc \times A$. This shows the assertion in \eqref{eq_claim_1_thm}.
Next, we use that $\Omloc \times A \ni (x,a_0) \mapsto G(x,a_0)$ is continuous, and  apply again Berge's maximum theorem to the constant compact-valued correspondence $\Omloc \ni x\twoheadrightarrow A\subset \R^m$ and to $G$. This yields that the map 
\begin{equation}\label{eq_continuous_berge_2}
\begin{aligned}
H:\Omloc   &\rightarrow \R \\
 x &\mapsto \sup_{a_0 \in A} G(x,a_0).
\end{aligned}
\end{equation}
is continuous and that the set of maximizers of $G(x,\cdot)$ is nonempty for each $x \in \Omloc$. Moreover, by the measurable maximum theorem {(Theorem~\ref{thm_measurable_maximum})} there exists a measurable selector $\Omloc \ni x\mapsto \aloc^*(x) \in A $, where $\aloc^*(x)$ maximizes  $G(x,\cdot)$ for each $x \in \Omloc$. This shows \eqref{eq_claim_2_thm}. Next, we define the map  $\PP_{\operatorname{loc}}^* \in \mathbf{P}_{\aloc^*}$ by
\begin{align*}
\PP_{\operatorname{loc}}^*:\Omloc &\rightarrow \mathcal{M}_1(\Omloc)\\
x &\mapsto {\PP_0}^*(x,\aloc^*(x)).
\end{align*}
{ Then, by \eqref{eq_continuous_berge}, \eqref{eq_continuous_berge_2}, the definition of $\Omloc \ni x \mapsto \aloc^*(x) \in A$, and the definition of $\Omloc \times A \ni (x,a_0) \mapsto \PP_0^*(x,a_0) \in \mathcal{M}_1(\Omloc)$,} we  obtain for each $x \in \Omloc$ that  
\begin{equation}\label{eq_proof_thm_lastequation_i}
\begin{aligned}
{ \T} v(x) = H(x)&= { \sup_{a_0 \in A} G(x,a_0)} \\
&=\sup_{a_0 \in A} \inf_{\PP_0 \in \mathcal{P}(x,a_0)}\int_{\Omloc}r(x,a_0,\omega_1)+\alpha  v(\omega_1)\PP_0(\D\omega_1) \\
&=\inf_{\PP_0 \in \mathcal{P}(x,a_0)}\int_{\Omloc}r(x,\aloc^*(x),\omega_1)+\alpha  v(\omega_1)\PP_0(\D\omega_1) \\
&= \int_{\Omloc}r(x,\aloc^*(x),\omega_1)+\alpha  v(\omega_1)\PP_{\operatorname{loc}}^*(x;\D\omega_1)\\
&{ = \inf_{\PP \in \mathbf{P}_{\aloc^*}} \int_{\Omloc}r(x,\aloc^*(x),\omega_1)+\alpha  v(\omega_1)\PP(x;\D\omega_1) }\\
&= \sup_{\aloc \in \Aloc} \inf_{\PP \in \mathbf{P}_{\aloc}} \int_{\Omloc}r(x,\aloc(x),\omega_1)+\alpha  v(\omega_1)\PP (x;\D\omega_1).
\end{aligned}
\end{equation}
This shows \eqref{eq_thm_assertion_1} and therefore completes the proof of (i).
\end{proof}

\begin{proof}[Proof of Theorem~\ref{thm_main_result}~ (ii)]
The continuity of $\Omloc \ni x \mapsto \T v(x)$ follows from the continuity of $H$ defined in \eqref{eq_continuous_berge_2} and \eqref{eq_proof_thm_lastequation_i}.  By the growth conditions on $r$ and $X_1$  (Assumption~\ref{asu_2}~(iii)), we obtain for all $x\in \Omloc$ that 
\begin{align*}
\T v(x)&\leq \sup_{\aloc \in \Aloc}\inf_{\PP_0 \in \mathbf{P}_{\aloc}}\E_{\PP_0}\left[C_r(1+\|x\|^p+\left\|X_1\right\|^p)+\alpha \|v\|_{C_p}\left(1+\left\|X_1\right\|^p\right)\right]\\
&\leq C_r\left (\|x\|^p+C_P(1+\|x\|^p)\right)+\alpha \|v\|_{C_p}C_P(1+\|x\|^p)\\
&\leq \left(C_r +C_rC_P+\alpha \|v\|_{C_p} C_P\right)(1+\|x\|^p).
\end{align*}
Hence, $\T v \in C_p(\Omloc,\R)$.  Note that for every nonempty set $\mathcal{Q}$ and for all $G,H:\mathcal{Q}\rightarrow \R$ we have that 
\begin{equation}\label{eq_inf_smaller_than_sup}
{ \bigg|}\inf_{\Q \in \mathcal{Q}}G(\Q)-\inf_{\Q \in \mathcal{Q}}H(\Q){ \bigg|} \leq \sup_{\Q \in \mathcal{Q}} \left|G(\Q)-H(\Q)\right|.
\end{equation}
Therefore, we obtain for every $v,w \in C_p(\Omloc,\R), x\in \Omloc$, by using \eqref{eq_growth_constraint_on_p}, that
\begin{align*}
{ \bigg|}\T v(x)-\T w(x){ \bigg|} &\leq \alpha  \sup_{a \in A}\sup_{\PP_0 \in \mathcal{P}(x,a)} \left|\E_{\PP_0}\bigg[v(X_1)-w(X_1)\bigg]\right|\\
&\leq \alpha  \sup_{a \in A}\sup_{\PP_0 \in \mathcal{P}(x,a)} \E_{\PP_0}\bigg[\left|v(X_1)-w(X_1)\right|\bigg]\\
&\leq \alpha  \sup_{a \in A}\sup_{\PP_0  \in \mathcal{P}(x,a)} \E_{\PP_0}\bigg[\|v-w\|_{C_p}\left(1+\left\|X_1\right\|^p\right)\bigg]\\
&\leq \alpha   C_P \|v-w\|_{C_p}\cdot \left(1+\|x\|^p\right).
\end{align*}
{ Hence, we obtain} \eqref{eq_contraction}.
Now, let $v_0 \in C_p(\Omloc,\R)$. Then, since by Assumption~\ref{asu_2}~(iv) we have $0 < \alpha C_P<1$, $\T$ is a contraction on $C_p(\Omloc,\R)$. Hence Banach's { fixed} point theorem (Theorem~\ref{thm_banach}) implies existence and uniqueness of a fix point $v\in C_p(\Omloc,\R)$ such that $v = \T v = \lim_{n \rightarrow \infty} \T^n v_0$.
\end{proof}

\begin{proof}[Proof of Theorem~\ref{thm_main_result}~(iii)]
We first show the inequality $v(x) \leq V(x)$.\\
Let $v\in C_p(\Omloc,\R)$ be the { fixed} point of $\T$ whose existence and uniqueness were proved in part (ii), i.e., we have that $v = \T v$. Let ${\aloc}^* \in \Aloc$ be the minimizer from part (i) and write ${\ab}^*:=(\aloc^*(X_0),\aloc^*(X_1),\cdots)\in \mathcal{A}$, and let $x \in \Omloc$.  First, we claim it holds for all $n\in \N$ with $ n \geq 2$ that
\begin{equation}\label{eq_induction_claim_1}
\T ^n v(x) \leq  \inf_{\PP \in \mathfrak{P}_{x,\ab^*}} \E_{\PP}\bigg[\sum_{t=0}^{n-1} \alpha^tr(X_{t},{\aloc^*}(X_t),X_{t+1})+\alpha ^nv(X_n^{x})\bigg].
\end{equation}
We prove the claim \eqref{eq_induction_claim_1} inductively. To that end, note that by part~(i) we can write
\begin{align*}
\T v(x)&= \inf_{\PP_0 \in \mathbf{P}_{{\aloc}^*}} \E_{\PP_0(x)}\bigg[r(x,{\aloc}^*(x),X_1)+\alpha v(X_1)\bigg]\\
&=\inf_{\PP_0 \in \mathbf{P}_{{\aloc}^*} } \int_{\Omloc}r\left(x,{\aloc}^*(x),\omega_1\right)+\alpha v (\omega_1)~\PP_0(x;\D\omega_1).
\end{align*}
This implies for all $t\in \N_0$ and $\omega=(\omega_t)_{t\in \N_0} \in \Omega$ that
\begin{equation}
\begin{aligned}
\label{eq_T_a_equality}
\T v (X_{t}(\omega ))
&=\inf_{\PP_0\in \mathbf{P}_{{\aloc}^*}  } \int_{\Omloc}r\left(X_{t}(\omega ) ,{\aloc}^*(X_{t}(\omega )),\omega_1\right)+\alpha v\left(\omega_1\right)\PP_0 (X_{t}(\omega );\D \omega_1)\\
&=\inf_{\PP_t\in \mathbf{P}_{{\aloc}^*}  } \int_{\Omloc}r\left(\omega_t ,{\aloc}^*(\omega_t ),\omega_{t+1}\right)+\alpha v\left(\omega_{t+1}\right)\PP_t(\omega_t;\D\omega_{t+1}),
\end{aligned}
\end{equation}
where we just used the definition of the canonical process and relabelled the variables $\PP_0$ and $\omega_1$.
For $n=2$, as $\T v=v$, we thus have
\begin{equation}\label{eq_TTv_1}
\begin{aligned}
\T \left(\T v \right)(x)&= \T v(x)=\inf_{\PP_0 \in \mathbf{P}_{{\aloc}^*} } \int_{\Omloc}r\left(x,{\aloc}^*(x),\omega_1\right)+\alpha v(x)\PP_0(x;\D\omega_1)\\
&=\inf_{\PP_0 \in \mathbf{P}_{{\aloc}^*} } \int_{\Omloc}r\left(x,{\aloc}^*(x),\omega_1\right)+\alpha \T v(x)\PP_0(x;\D\omega_1)\\
&=\inf_{{\PP_0 \in\mathbf{P}_{{\aloc}^*} }} \int_{\Omloc}\bigg[ r\left(x,{\aloc}^*(x),\omega_1\right)\\
&\hspace{1cm}+\alpha \sup_{\aloc \in \Aloc}\inf_{{\PP_1  \in \mathbf{P}_{{\aloc}}}}\int_{\Omloc}\bigg\{r\left(\omega_1,{\aloc}(\omega_1),\omega_2\right)+\alpha v \left(\omega_2\right)\bigg\}~\PP_1(\omega_1; \D\omega_2)\bigg]\PP_0(x;\D\omega_1).
\end{aligned}
\end{equation}
This and part~(i) ensures that
\begin{align*}
\T \left(\T v \right)(x)&=\inf_{{\PP_0 \in\mathbf{P}_{{\aloc}^*} }} \int_{\Omloc}\bigg[ r\left(x,{\aloc}^*(x),\omega_1\right)\\
&\hspace{3cm}+\alpha \inf_{{\PP_1  \in \mathbf{P}_{{\aloc}^*}}}\int_{\Omloc}\bigg\{r\left(\omega_1,{\aloc}^*(\omega_1),\omega_2\right)+\alpha v \left(\omega_2\right)\bigg\}~\PP_1(\omega_1; \D\omega_2)\bigg]\PP_0(x;\D\omega_1)\\
&\leq \inf_{\PP_t  \in \mathbf{P}_{{\aloc}^*} ,\atop t=0,1} \int_{\Omloc}\int_{\Omloc} \sum_{t=0}^{1} \alpha^t r(\omega_t,{\aloc}^*(\omega_t ),\omega_{t+1})+\alpha^2 v\left(\omega_2\right) ~\PP_1(\omega_1;\D\omega_2) ~\PP_0(x;\D\omega_1)\\
&=\inf_{\PP \in \mathfrak{P}_{x,{\ab}^*}} \E_{\PP}\left[\sum_{t=0}^{1} \alpha^t r(X_{t},{\aloc}^*(X_t),X_{t+1})+\alpha^2 v\left(X_2\right)\right],
\end{align*}
where we use \eqref{eq_T_a_equality} and the structure of the measures $\mathfrak{P}_{x,\ab^*}$. The general case for arbitrary $n$ follows with analogue arguments. Indeed, let the claim in \eqref{eq_induction_claim_1} be true for $n-1$, then it follows by the same argument as in \eqref{eq_TTv_1} and by the structure of every $\PP \in \mathfrak{P}_{\omega_1,\ab^*}$ that
\begin{align*}
&\lefteqn{\T ^n v(x)=\T \left(\T ^{n-1} v \right)(x)}\\
&\leq \inf_{\PP_0\ \in \mathbf{P}_{{\aloc^*}}} \int_{\Omloc}\bigg[ r\left(x,{\aloc^*}(x),\omega_1\right)\\
&\hspace{1.2cm}+\alpha \inf_{\PP  \in \mathfrak{P}_{\omega_1,\ab^*}} \int_{\Omega} \bigg\{ \sum_{t=1}^{n-1} \alpha^{t-1} r(\omega_t,{\aloc^*}(\omega_t),\omega_{t+1})+\alpha^{n-1} v\left(\omega_n\right) \bigg\} \PP(\D\omega) \bigg]\PP_0(x;\D\omega_1) \\
&=\inf_{{\PP_0\ \in \mathbf{P}_{{\aloc^*}}}} \int_{\Omloc}\bigg[ r\left(x,{\aloc^*}(x),\omega_1\right)\\
&\hspace{1.2cm}+\alpha {\inf_{{\PP_t  \in \mathbf{P}_{{\aloc}^*} ,\atop t=1,\dots,n-1}}} \int_{\Omloc}\cdots \int_{\Omloc} \bigg\{ \sum_{t=1}^{n-1} \alpha^{t-1} r(\omega_t,{\aloc^*}(\omega_t),\omega_{t+1})\\
&\hspace{6cm}+\alpha^{n-1} v\left(\omega_n\right) \bigg\} \PP_{n-1}(\omega_{n-1};\D\omega_n)\cdots \PP_{1}(\omega_1;\D\omega_2)\bigg]\PP_0(x;\D\omega_1)\\
&\leq {\inf_{{\PP_t  \in \mathbf{P}_{{\aloc^*}} ,\atop t=0,\dots,n-1}}} \int_{\Omloc}\int_{\Omloc}\cdots \int_{\Omloc} \bigg\{ \sum_{t=0}^{n-1} \alpha^t r(\omega_t,{\aloc^*}(\omega_t),\omega_{t+1})\\
&\hspace{6cm}+\alpha^n v\left(\omega_n\right) \bigg\} \PP_{n-1}(\omega_{n-1};\D\omega_{n})\cdots \PP_{1}(\omega_1;\D\omega_2) \PP_0(x;\D\omega_1)\\
&=\inf_{\PP \in \mathfrak{P}_{x,{\ab^*}}} \E_{\PP}\bigg[\sum_{t=0}^{n-1} \alpha^tr(X_{t},{\aloc^*}(X_t),X_{t+1})+\alpha ^nv(X_n)\bigg].
\end{align*}
According to \eqref{eq_induction_claim_1} we have for all $n\in \N$ that
\begin{align}
v(x)=\T v(x)&=\T ^nv(x)\leq \inf_{\PP \in \mathfrak{P}_{x,\ab^*}} \E_{\PP}\bigg[\sum_{t=0}^{n-1} \alpha^tr(X_{t},{\aloc}^*(X_t),X_{t+1})+\alpha ^nv(X_n)\bigg].\label{eq_sum_inf_1}
\end{align}
Further, we obtain for all ${ \PP = \delta_x \otimes \PP_0 \otimes \PP_1 \cdots} \in \mathfrak{P}_{x,\ab^*}$ and $n\in \N$ by { \eqref{eq_defn_CP_norm} and }\eqref{eq_growth_constraint_on_p} that
\begin{align*}
\E_{\PP}\bigg[\left|v(X_n)\right|\bigg] &\leq { \E_{\PP}\bigg[\|v\|_{C_p} (1+\|X_n\|^p)\bigg]} \\
&{ =\|v\|_{C_p}  \int_{\Omloc} \cdots \int_{\Omloc} (1+\|\omega_n\|^p)    \PP_{n-1}(\omega_{n-1};\D\omega_{n})\cdots \PP_{1}(\omega_1;\D\omega_2) \PP_0(x;\D\omega_1)}\\
&{ \leq \|v\|_{C_p}  \int_{\Omloc} \cdots \int_{\Omloc} C_P(1+\|\omega_{n-1}\|^p)    \PP_{n-2}(\omega_{n-2};\D\omega_{n-1})\cdots \PP_{1}(\omega_1;\D\omega_2) \PP_0(x;\D\omega_1)}\\
&{ =}  \|v\|_{C_p}  \E_{\PP}\bigg[C_P(1+\|X_{n-1}\|^p)\bigg].
\end{align*}
{ Therefore, by repeating this argument, we obtain that
\[
\E_{\PP}\bigg[\left|v(X_n)\right|\bigg] \leq  \|v\|_{C_p}  \E_{\PP} \bigg[C_P(1+\|X_{n-1}\|^p)\bigg]  \leq  \cdots \leq  \|v\|_{C_p} C_P^n(1+\|x\|^p).
\]
}
Note that Assumption~\ref{asu_2} implies {  $0 < C_P \cdot \alpha  <1$}. When letting $n\rightarrow \infty$, we thus have for all $\PP \in \mathfrak{P}_{x,\ab}$ that 
\begin{equation}\label{eq_limsup_ineq_1}
0 \leq \limsup_{n \rightarrow \infty} \E_{\PP}\bigg[\alpha^n \left| v(X_n)\right|\bigg] \leq \|v\|_{C_p}(1+\|x\|^p) \cdot  \limsup_{n \rightarrow \infty} \left(C_P\cdot \alpha \right)^n =0.
\end{equation}
Moreover, note that by the growth condition on $r$ in Assumption~\ref{asu_2}~(iii) we have for each $n \in \N$ that
\begin{equation}\label{eq_dominating_function}
\sum_{t=0}^{n-1} \alpha^tr(X_{t},{\aloc}^*(X_t),X_{t+1}) \leq \sum_{t=0}^{\infty} \alpha^t C_r(1+\|X_t\|^p+\|X_{t+1}\|^p).
\end{equation}
Furthermore, for all ${ \PP \equiv \delta_x \otimes \PP_0 \otimes \PP_1 \cdots}\in \mathfrak{P}_{x,\ab^*}$, by using  \eqref{eq_growth_constraint_on_p} and Beppo Levi's theorem, we have that
\begin{equation}\label{eq_domin_convergence}
\begin{aligned}
\E_{\PP}&\bigg[\sum_{t=0}^{\infty} \alpha^tC_r (1+\|X_t\|^p+\|X_{t+1}\|^p)\bigg] \\&={  \sum_{t=0}^{\infty}  \E_{\PP}\bigg[ \alpha^tC_r (\|X_t\|^p+1+\|X_{t+1}\|^p)\bigg] }\\
&{ =  \sum_{t=0}^{\infty} \int_{\Omloc} \cdots \int_{\Omloc}  \alpha^tC_r (\|\omega_t\|^p+1+\|\omega_{t+1}\|^p)  \PP_{t}(\omega_{t};\D\omega_{t+1})\cdots \PP_{1}(\omega_1;\D\omega_2) \PP_0(x;\D\omega_1)}\\
&{\leq   \sum_{t=0}^{\infty} \int_{\Omloc} \cdots \int_{\Omloc}   \alpha^tC_r \bigg(\|\omega_t\|^p+C_P(1+\|\omega_{t}\|^p)\bigg)  \PP_{t-1}(\omega_{t-1};\D\omega_{t})\cdots \PP_{1}(\omega_1;\D\omega_2) \PP_0(x;\D\omega_1)}\\
&={  \sum_{t=0}^{\infty}  \E_{\PP}\bigg[ \alpha^tC_r \bigg(\|X_t\|^p+C_P(1+\|X_{t}\|^p)\bigg)\bigg] }.
\end{aligned}
\end{equation}
{ Recall that $\alpha \cdot C_P <1$ according to Assumption~\ref{asu_2}~(iv).}
Therefore, by repeating the same arguments { using \eqref{eq_growth_constraint_on_p}},  we obtain that
\begin{equation}\label{eq_domin_convergence_2}
\begin{aligned}
\E_{\PP}&\bigg[\sum_{t=0}^{\infty} \alpha^tC_r (1+\|X_t\|^p+\|X_{t+1}\|^p)\bigg] \\
&\leq {  \sum_{t=0}^{\infty} \alpha^t C_r(1+C_P) \E_{\PP}\bigg[(1+\|X_t\|^p)\bigg] }\\
&{= \sum_{t=0}^{\infty} \alpha^t C_r(1+C_P)\int_{\Omloc} \cdots \int_{\Omloc} (1+\|\omega_t\|^p)\PP_{t-1}(\omega_{t-1};\D\omega_{t})\cdots \PP_{1}(\omega_1;\D\omega_2) \PP_0(x;\D\omega_1)}\\
&{\leq \sum_{t=0}^{\infty} \alpha^t C_r(1+C_P)\int_{\Omloc} \cdots \int_{\Omloc} C_P(1+\|\omega_{t-1}\|^p)\PP_{t-2}(\omega_{t-2};\D\omega_{t-1})\cdots \PP_{1}(\omega_1;\D\omega_2) \PP_0(x;\D\omega_1)}\\
&{= \sum_{t=0}^{\infty} \alpha^t C_r(1+C_P) C_P\E_{\PP} \left[(1+\|X_{t-1}\|^p)\right]}\\
&\leq \sum_{t=0}^{\infty} \alpha^t C_r(1+C_P){ C_P^t}(1+\|x\|^p)\\
&= C_r(1+C_P)\bigg[\sum_{t=0}^{\infty}  { ( \alpha \cdot C_P)^t}\bigg]\cdot (1+\|x\|^p)\\
&=\frac{{ C_r(1+C_P)}(1+\|x\|^p)}{{ 1-\alpha C_P}}<\infty.
\end{aligned}
\end{equation}
Hence the dominating function in \eqref{eq_dominating_function} is integrable and we obtain, by using the dominated convergence theorem and \eqref{eq_limsup_ineq_1}, that
\begin{equation}\label{eq_v_leq_V}
\begin{aligned}
v(x)&\leq \limsup_{n \rightarrow \infty}\inf_{\PP \in \mathfrak{P}_{x,\ab^*}} \E_{\PP}\bigg[\sum_{t=0}^{n-1} \alpha^tr(X_{t},{\aloc}^*(X_t),X_{t+1})+\alpha ^nv(X_n)\bigg]\\
&\leq  \inf_{\PP \in \mathfrak{P}_{x,\ab^*}} \limsup_{n \rightarrow \infty} \E_{\PP}\bigg[\sum_{t=0}^{n-1} \alpha^tr(X_{t},{\aloc}^*(X_t),X_{t+1})\bigg]+\limsup_{n \rightarrow \infty} \E_{\PP}\bigg[\alpha ^n\left|v(X_n)\right|\bigg]\\
&=\inf_{\PP \in \mathfrak{P}_{x,\ab^*}} \E_{\PP}\bigg[\sum_{t=0}^{\infty} \alpha^tr(X_{t},{\aloc}^*(X_t),X_{t+1})\bigg]
\leq  V(x).
\end{aligned}
\end{equation}
Next, we show the inequality $v(x) \geq V(x)$.
To this end, let $\PP_{0}^*:\Omloc \times A \rightarrow \mathcal{M}_1(\Omloc)$ be defined as in part (i) with respect to the unique { fixed} point $v\in C_p(\Omloc,\R)$ of $\T$. Moreover, for every $\ab=(a_t)_{t\in \N_0} \in \A
$ let $\PP_{x,\ab}^*:=\delta_x\otimes \PP_{a_0}^* \otimes \PP_{a_1}^*  \otimes \cdots$, where for $t\in \N$ we define $\PP_{a_t}^*:\Omloc \ni \omega_t \mapsto \PP_{0}^*(\omega_t ,a_t(\omega_t ))\in \mathcal{P}(\omega_t,a_t(\omega_t))$. 
%
 Thus, we have, by using the dominated convergence theorem with the same dominating function as in \eqref{eq_dominating_function}, that
\begin{equation}\label{eq_V_smaller_v}
\begin{aligned}
V(x)&= \sup_{\ab\in \mathcal{A}}\inf_{\PP \in \mathfrak{P}_{x,\ab}} \E_{\PP}\bigg[\sum_{t=0}^\infty \alpha^tr(X_t,a_t(X_t),X_{t+1})\bigg] \\
&\leq  \sup_{\ab\in \mathcal{A}}\E_{\PP_{x,\ab}^*}\bigg[\sum_{t=0}^\infty \alpha^tr(X_t,a_t(X_t),X_{t+1})\bigg]\\
&=\sup_{\ab\in \mathcal{A}}\sum_{t=0}^\infty \E_{\PP_{x,\ab}^*}\bigg[\alpha^t r(X_t,a_t(X_t),X_{t+1})\bigg]\\
&=\sup_{\ab\in \mathcal{A}}\sum_{t=0}^\infty \bigg(\alpha^t \E_{\PP_{x,\ab}^*}\bigg[r(X_t,a_t(X_t),X_{t+1})+\alpha v(X_{t+1})\bigg]-\E_{\PP_{x,\ab}^*}\bigg[\alpha^{t+1}v(X_{t+1})\bigg]\bigg)\\
&= \sup_{\ab\in \mathcal{A}}\sum_{t=0}^\infty \bigg(\alpha^t \int_{\Omloc} \cdots \int_{\Omloc} r(\omega_t,a_t(\omega_t),\omega_{t+1})+\alpha v(\omega_{t+1})\PP_{0}^*(\omega_t,a_t(\omega_t);\D \omega_{t+1})\cdots \PP_{0}^*(x,a_0(x);\D \omega_1) \\
&\hspace{8.6cm}-\E_{\PP_{x,\ab}^*}\bigg[\alpha^{t+1}v(X_{t+1})\bigg] \bigg).
\end{aligned}
\end{equation}
Moreover, by using the results from part~(i) we have for all $\omega_t \in \Omloc$
\begin{align}
&\int_{\Omloc} r(\omega_t,a_t(\omega_t),\omega_{t+1})+\alpha v(\omega_{t+1})\PP_{0}^*(\omega_t,a_t(\omega_t);\D \omega_{t+1}) \notag \\
&= \inf_{\PP_0 \in \mathcal{P}(\omega_t, a_t(\omega_t))} \int_{\Omloc} r(\omega_t,a_t(\omega_t),\omega_{t+1})+\alpha v(\omega_{t+1})\PP_{0}(\omega_t,a_t(\omega_t);\D \omega_{t+1})  \notag  \\
&\leq \sup_{\aloc \in \Aloc}\inf_{\PP_0 \in \mathcal{P}(\omega_t, \aloc(\omega_t))} \int_{\Omloc} r(\omega_t,\aloc(\omega_t),\omega_{t+1})+\alpha v(\omega_{t+1})\PP_{0}(\omega_t,\aloc(\omega_t);\D \omega_{t+1}) \notag \\
&=\int_{\Omloc} r(\omega_t,\aloc^*(\omega_t),\omega_{t+1})+\alpha v(\omega_{t+1})\PP_{0}^*(\omega_t,\aloc^*(\omega_t);\D \omega_{t+1}) =\T v(\omega_t)=v(\omega_t). \label{eq_lastline_integrals}
\end{align} 
Hence, we obtain with \eqref{eq_V_smaller_v} and \eqref{eq_lastline_integrals} that
\begin{equation*}
\begin{aligned}
V(x) &\leq \sup_{\ab\in \mathcal{A}}\sum_{t=0}^\infty \bigg(\alpha^t \int_{\Omloc} \cdots \int_{\Omloc} v(\omega_t)\PP_{0}^*(\omega_{t-1},a_{t-1}(\omega_{t-1});\D \omega_{t})\cdots \PP_{0}^*(x,a_0(x);\D \omega_1)\\
&\hspace{8cm}-\E_{\PP_{x,\ab}^*}\bigg[\alpha^{t+1}v(X_{t+1})\bigg] \bigg)\\
&= \sup_{\ab\in \mathcal{A}} \sum_{t=0}^\infty \left(\alpha^{t}\E_{\PP_{x,\ab}^*}\bigg[v(X_{t})\bigg]-\alpha^{t+1}\E_{\PP_{x,\ab}^*}\bigg[v(X_{t+1})\bigg]\right) \\
&=\sup_{\ab\in \mathcal{A}} v(x)=v(x).
\end{aligned}
\end{equation*}
This shows $V(x)=v(x)$. Eventually, to see that the first line of \eqref{eq_thm_assertion_3} holds, we compute by using \eqref{eq_lastline_integrals}, the definition $\PP_x^*:=\delta_x \otimes \PP_{\operatorname{loc}}^* \otimes \PP_{\operatorname{loc}}^* \otimes \cdots$ as well as the dominated convergence theorem that
\begin{align}
v(x) &= \sum_{t=0}^\infty \left(\alpha^{t}\E_{\PP_{x}^*}\bigg[v(X_{t})\bigg]-\alpha^{t+1}\E_{\PP_{x}^*}\bigg[v(X_{t+1})\bigg]\right) \notag \\ 
&= \sum_{t=0}^\infty \left(\alpha^{t}\E_{\PP_{x}^*}\bigg[r(X_t,\aloc^*(X_t),X_{t+1})+\alpha v(X_{t+1})\bigg]-\alpha^{t+1}\E_{\PP_{x}^*}\bigg[v(X_{t+1})\bigg]\right) \notag \\
&= \sum_{t=0}^\infty \E_{\PP_{x}^*}\bigg[\alpha^{t}r(X_t,\aloc^*(X_t),X_{t+1})\bigg]= \E_{\PP_{x}^*}\bigg[\sum_{t=0}^\infty \alpha^{t}r(X_t,\aloc^*(X_t),X_{t+1})\bigg].  \notag 
\end{align}
Moreover, by \eqref{eq_v_leq_V}, as we have shown that $V=v$, we obtain that 
\begin{equation}
V(x)=\inf_{\PP \in \mathfrak{P}_{x,\ab^*}}\E_{\PP}\bigg[\sum_{t=0}^\infty \alpha^tr(X_{t},\aloc^*(X_t),X_{t+1})\bigg].
\end{equation}
\end{proof}

\section{Proof of Results in Section~\ref{sec_distributional_uncertainty}}
\label{sec_proofs_sec3}
\subsection{Proof of Results in Section~\ref{sec_wasserstein}}

\begin{proof}[Proof of Proposition~\ref{prop_wasserstein}]
Let $x\in \Omloc$, $a \in A$.\\
We see that $\mathcal{P}(x,a)$ is nonempty since the measure $\widehat{\PP}(x,a)$ is contained in $\mathcal{P}(x,a)$ by definition of the $q$-Wasserstein-ball. 

The compactness of $\mathcal{B}^{(q)}_\varepsilon\left(\widehat{\PP}(x,a)\right)$ with respect to $\tau_0$, which is the topology induced by the weak convergence of measures, follows from, e.g., \cite[Theorem 1]{yue2020linear}, where we use the assumption that $\widehat{\PP}(x,a)$ has finite $q$-th moments. 

To show the upper hemicontinuity of $\mathcal{P}$, we apply Lemma~\ref{lem_upper_hemi}. Let $(x^{(n)},a^{(n)})_{n\in \N} \subseteq \Omloc  \times A$ such that $(x^{(n)},a^{(n)}) \rightarrow (x,a)\in \Omloc \times A$ for $n \rightarrow \infty$. Further, consider a sequence $(\PP^{(n)})_{n \in \N}$ such that $\PP^{(n)} \in  \mathcal{B}^{(q)}_\varepsilon\left(\widehat{\PP}(x^{(n)},a^{(n)})\right)$ for all $n \in \N$, i.e., we have $\left(\left(x^{(n)},a^{(n)}\right),\PP^{(n)}\right)_{n \in \N} \subseteq \operatorname{Gr} \mathcal{P}$. 

Let $(\delta_n)_{n \in \N} \subseteq (0,1)$ with $\lim_{m \rightarrow \infty} \delta_n = 0$. Note that, since $\Omloc \times A \ni (x,a) \mapsto \widehat{\PP}(x,a)$ is, by assumption, continuous in $\tau_q$ we have $\lim_{n \rightarrow \infty} W_q\left(\widehat{\PP}(x,a), \widehat{\PP}(x^{(n)},a^{(n)})\right) =0$. Hence, there exists a subsequence $(\widehat{\PP}(x^{n_k)},a^{(n_k)})_{k \in \N}$ such that 
\begin{equation}\label{eq_proof_uhc_1}
 W_q\left(\widehat{\PP}(x,a), \widehat{\PP}(x^{(n_k)},a^{(n_k)})\right) < \delta_k \cdot \varepsilon \text{ for all } k \in \N.
\end{equation}
This implies for each $\PP^{(n_k)}$, $k \in \N$, that 
\begin{align*}
W_q\left(\widehat{\PP}(x,a), \PP^{(n_k)}\right)\leq W_q\left(\widehat{\PP}(x,a), \widehat{\PP}(x^{(n_k)},a^{(n_k)}) \right)+W_q\left( \widehat{\PP}(x^{(n_k)},a^{(n_k)}), \PP^{(n_k)}\right) \leq \delta_k \cdot \varepsilon+\varepsilon \leq 2 \varepsilon.
\end{align*}
Hence, $\PP^{(n_k)} \in \mathcal{B}_{2\varepsilon}^{(q)}(\widehat{\PP}(x,a))$ for all $k \in \N$. By the compactness of $\mathcal{B}_{2\varepsilon}^{(q)}(\widehat{\PP}(x,a))$ in $\tau_0$, there exists a subsequence $(\PP^{(n_{k_\ell})})_{\ell \in \N}$ such that $\PP^{(n_{k_\ell})} \xrightarrow{\tau_0} \PP$ as $\ell \rightarrow \infty$ for some $\PP \in  \mathcal{B}_{2\varepsilon}^{(q)}(\widehat{\PP}(x,a))$ . In particular, since by assumption $\widehat{\PP}(x,a)$ possesses finite $q$-th moments, $\PP$ has also finite $q$-th moments, see \cite[Lemma 1]{yue2020linear}. It remains to prove that $\PP \in \mathcal{B}^{(q)}_{\varepsilon}(\widehat{\PP}(x,a))$. To that end, define for each $k \in \N$ 
\begin{equation}\label{eq_proof_uhc_2}
\widetilde{\PP}^{(n_k)}:= (1-\delta_k) \cdot \PP^{(n_k)} + \delta_k \cdot \widehat{\PP}(x^{(n_k)},a^{(n_k)}).
\end{equation}
Then, for each $k \in \N$ we have
\begin{equation}\label{eq_proof_uhc_3}
\begin{aligned}
W_q&\left(\widehat{\PP}(x^{(n_k)},a^{(n_k)}), \widetilde{\PP}^{(n_k)}\right) \\
&= W_q\left((1-\delta_k) \cdot \widehat{\PP}(x^{(n_k)},a^{(n_k)})+\delta_k \cdot \widehat{\PP}(x^{(n_k)},a^{(n_k)}),
(1-\delta_k) \cdot \PP^{(n_k)} + \delta_k \cdot \widehat{\PP}(x^{(n_k)},a^{(n_k)})\right)\\
&=(1-\delta_k) \cdot  W_q\left(\widehat{\PP}(x^{(n_k)},a^{(n_k)}),\PP^{(n_k)}\right) \leq (1-\delta_k) \cdot \varepsilon.
\end{aligned}
\end{equation}
Therefore, by \eqref{eq_proof_uhc_1} and \eqref{eq_proof_uhc_3} we have for each $\ell \in \N$ that
\begin{equation}\label{eq_proof_uhc_4}
\begin{aligned}
W_q\left(\widehat{\PP}(x,a), \widetilde{\PP}^{({n_k}_\ell)}\right)&\leq W_q\left(\widehat{\PP}(x,a), \widehat{\PP}(x^{({n_k}_\ell)},a^{({n_k}_\ell)})\right)+W_q\left(\widehat{\PP}(x^{({n_k}_\ell)},a^{({n_k}_\ell)}), \widetilde{\PP}^{({n_k}_\ell)}\right)\\
&\leq \delta_{k_{\ell}}\cdot \varepsilon+(1-\delta_{k_\ell})\cdot \varepsilon=\varepsilon.
\end{aligned}
\end{equation}
Furthermore, we have by \eqref{eq_proof_uhc_2} that 
\begin{equation}\label{eq_proof_uhc_5}
\lim_{\ell \rightarrow \infty}\widetilde{\PP}^{({n_k}_\ell)} = \lim_{\ell \rightarrow \infty} {\PP}^{({n_k}_\ell)} = \PP \text{ in } \tau_0.
\end{equation}
Since $\mu \mapsto W_q(\widehat{\PP}(x,a),\mu)$ is lower semicontinuous in $\tau_0$, see \cite[Corollary 5.3]{clement2008wasserstein}, we obtain from \eqref{eq_proof_uhc_4} and \eqref{eq_proof_uhc_5} that 
\[
W_q\left(\widehat{\PP}(x,a),\PP\right) \leq \liminf_{\ell \rightarrow \infty} W_q\left(\widehat{\PP}(x,a), \widetilde{\PP}^{({n_k}_\ell)}\right) \leq \varepsilon,
\]
and hence $\PP \in \mathcal{B}_{\varepsilon}^{(q)}(\widehat{\PP}(x,a))$.
The assertion that $\mathcal{P}$ is upper hemicontinuous follows now with the characterization of upper hemicontinuity provided in Lemma~\ref{lem_upper_hemi}.

To show the lower hemicontinuity of $\mathcal{P}$ we first define the set-valued map 
\[
\accentset{\circ}{\mathcal{P}}:\Omloc \times A \ni (x,a) \twoheadrightarrow  \accentset{\circ}{\mathcal{B}}_{\varepsilon}^{(q)}(\widehat{\PP}(x,a)):=\left\{\PP\in \mathcal{M}_1(\Omloc)~\middle|~W_q(\PP,\widehat{\PP}(x,a))< \varepsilon \right\}
\]
and conclude the lower hemicontinuity of  $\accentset{\circ}{\mathcal{P}}$ with Lemma~\ref{lem_lower_hemi}.
 To this end, we consider a sequence $(x^{(n)},a^{(n)})_{n\in \N} \subset \Omloc  \times A$ such that $(x^{(n)},a^{(n)}) \rightarrow (x,a)\in \Omloc \times A$ for $n \rightarrow \infty$, and we consider some $\PP \in \accentset{\circ}{\mathcal{P}}\left((x,a)\right)=\accentset{\circ}{\mathcal{B}}_{\varepsilon}^{(q)}(\widehat{\PP}(x,a))$. Note that since $\accentset{\circ}{\mathcal{B}}_{\varepsilon}^{(q)}(\widehat{\PP}(x,a))$ is defined as an open ball with respect to $\tau_q$, there exists some $0< \delta< \varepsilon$ such that $\PP \in \accentset{\circ}{\mathcal{B}}_{\varepsilon-\delta}^{(q)}(\widehat{\PP}(x,a))$. We define for $n \in \N$ the measure 
\[
\PP^{(n)}:=
\begin{cases}
\widehat{\PP}(x^{(n)},a^{(n)}),&\text{if } W_q\left(\widehat{\PP}(x^{(n)},a^{(n)}),\widehat{\PP}(x,a)\right) \geq \delta\\
\PP, &\text{else}.
\end{cases}
\]
Then, we claim that $\PP^{(n)} \in \accentset{\circ}{\mathcal{P}}\left((x^{(n)},a^{(n)})\right)$ for all $n \in \N$. Indeed, if $W_q\left(\widehat{\PP}(x^{(n)},a^{(n)}),\widehat{\PP}(x,a)\right) \geq \delta$ this follows by definition of $\PP^{(n)}$, whereas if $W_q\left(\widehat{\PP}(x^{(n)},a^{(n)}),\widehat{\PP}(x,a)\right) < \delta$, then  $\PP^{(n)}=\PP$, and hence by the triangle inequality
\[
W_q\left(\PP,\widehat{\PP}(x^{(n)},a^{(n)})\right)\leq W_q\left(\PP,\widehat{\PP}(x,a)\right)+W_q\left(\widehat{\PP}(x,a),\widehat{\PP}(x^{(n)},a^{(n)})\right) < (\varepsilon-\delta)+\delta=\varepsilon.
\]
By the continuity of $(x,a) \mapsto \widehat{\PP}(x,a)$ in $\tau_q$ we have that $\widehat{\PP}(x^{(n)},a^{(n)}) \xrightarrow{\tau_q} \widehat{\PP}(x,a)$ as $n \rightarrow \infty$. Thus, there exists some $N \in \N$ such that we have $\PP^{(n)}=\PP$ for all $n \geq N$ and thus, in particular $\PP^{(n)} \rightarrow \PP$ weakly for $n \rightarrow \infty$, which concludes the lower hemicontinuity of $\accentset{\circ}{\mathcal{P}}$ with Lemma~\ref{lem_lower_hemi}. Next, we claim that the $\tau_0$-closure of $\accentset{\circ}{\mathcal{B}}_{\varepsilon}^{(q)}(\widehat{\PP}(x,a))$, denoted by $\operatorname{cl}_{\tau_0}\left(\accentset{\circ}{\mathcal{B}}_{\varepsilon}^{(q)}(\widehat{\PP}(x,a))\right)$, coincides with ${\mathcal{B}_{\varepsilon}}^{(q)}(\widehat{\PP}(x,a))$. Indeed, the inclusion ${\mathcal{B}_{\varepsilon}}^{(q)}(\widehat{\PP}(x,a)) \subseteq \operatorname{cl}_{\tau_0}\left(\accentset{\circ}{\mathcal{B}}_{\varepsilon}^{(q)}(\widehat{\PP}(x,a))\right)$ follows, since $\operatorname{cl}_{\tau_0}\left(\accentset{\circ}{\mathcal{B}}_{\varepsilon}^{(q)}(\widehat{\PP}(x,a))\right)$ is closed in $\tau_0$ and hence also in $\tau_q$. To show the reverse inclusion $\operatorname{cl}_{\tau_0}\left(\accentset{\circ}{\mathcal{B}}_{\varepsilon}^{(q)}(\widehat{\PP}(x,a))\right) \subseteq {\mathcal{B}_{\varepsilon}}^{(q)}(\widehat{\PP}(x,a))$ let $\PP \in \operatorname{cl}_{\tau_0}\left(\accentset{\circ}{\mathcal{B}}_{\varepsilon}^{(q)}(\widehat{\PP}(x,a))\right)$. Then, there exists a sequence $(\PP^{(n)})_{n \in \N} \subseteq \accentset{\circ}{\mathcal{B}}_{\varepsilon}^{(q)}(\widehat{\PP}(x,a))$ with $\PP^{(n)}\xrightarrow{\tau_0} \PP$ as $n \rightarrow \infty$. Hence by using the lower semicontinuity of $\mu \mapsto W_q(\mu,\widehat{\PP}(x,a))$ with respect to $\tau_0$ we obtain
\[
W_q\left(\PP, \widehat{\PP}(x,a)\right) \leq \liminf_{n \rightarrow \infty} W_q\left(\PP^{(n)},\widehat{\PP}(x,a)\right)\leq \varepsilon.
\]
Hence,  $\operatorname{cl}_{\tau_0}\left(\accentset{\circ}{\mathcal{B}}_{\varepsilon}^{(q)}(\widehat{\PP}(x,a))\right)={\mathcal{B}_{\varepsilon}}^{(q)}(\widehat{\PP}(x,a))$ and \cite[Lemma 17.22]{Aliprantis} implies that the set-valued map $\mathcal{P}:\Omloc \times A \ni (x,a) \twoheadrightarrow \operatorname{cl}_{\tau_0}\left(\accentset{\circ}{\mathcal{B}}_{\varepsilon}^{(q)}(\widehat{\PP}(x,a))\right)$ is lower hemicontinuous.

Eventually, since $p=0$, the growth constraint \eqref{eq_growth_constraint_on_p} is automatically fulfilled.
\end{proof}
\subsection{Proof of Results in Section~\ref{sec_parametric}}

\begin{proof}[Proof of Propositon~\ref{prop_knightian_1}]
Let $(x,a) \in \Omloc \times A$.\\
The nonemptiness of $\mathcal{P}(x,a)$ follows directly since $\Theta$ is nonempty.

To show the compactness of $\mathcal{P}(x,a)$ let $\left(\PP^{(n)}\right)_{n \in \N} \subseteq \mathcal{P}(x,a)$, i.e., for all $n \in \N$ we have $\PP^{(n)} = \widehat{\PP}(x,a, \theta^{(n)})$ for some $\theta^{(n)} \in \Theta(x,a)$. The compactness of $\Theta(x,a)$ implies the existence of a subsequence $(\theta^{(n_k)})_{k \in \N} \subseteq \Theta(x,a)$ such that $\theta^{(n_k)} \rightarrow \theta \in \Theta(x,a)$ for $k \rightarrow \infty$. Hence, since $\widehat{\PP}$ is continuous, it follows $\widehat{\PP}(x,a,\theta^{(n_k)}) \rightarrow \widehat{\PP}(x,a,\theta) \in \mathcal{P}(x,a)$ in $\tau_p$ for $k \rightarrow \infty$ .

We apply Lemma~\ref{lem_upper_hemi} to show the upper hemicontinuity of $\mathcal{P}$. To this end, consider a sequence $(x^{(n)},a^{(n)})_{n \in \N}\subseteq \Omloc \times A$ with $\lim_{n \rightarrow \infty}(x^{(n)},a^{(n)}) = (x,a)$ and a sequence $(\PP^{(n)})_{n \in \N}$ with $\PP^{(n)} \in \mathcal{P}(x^{(n)},a^{(n)})$ for all $n \in \N$. We have a representation $\PP^{(n)}=\widehat{\PP}(x^{(n)},a^{(n)},\theta^{(n)})$ for some $\theta^{(n)}\in \Theta(x^{(n)},a^{(n)})$ for all $n \in \N$. Then, since $\Theta$ is upper hemicontinuous, there exists a subsequence $(\theta^{(n_k)})_{k \in \N}$ with $\theta^{(n_k)} \in \Theta(x^{(n_k)},a^{(n_k)})$ for all $k \in \N$ such that $\theta^{(n_k)} \rightarrow \theta$ for $k \rightarrow \infty$ for some $\theta \in \Theta(x,a)$. Hence with the continuity of $\widehat{\PP}$ it follows $\PP^{(n_k)} \rightarrow \PP:=\widehat{\PP}(x,a,\theta)\in \mathcal{P}(x,a)$ in $\tau_p$ for $k \rightarrow \infty$.

To show the lower hemicontinuity we let $(x^{(n)},a^{(n)})_{n \in \N}\subseteq \Omloc \times A$ with $\lim_{n \rightarrow \infty}(x^{(n)},a^{(n)}) = (x,a)$ and $\PP:=\widehat{\PP}(x,a,\theta) \in \mathcal{P}(x,a)$ for some $\theta \in \Theta(x,a)$. Then, the lower hemicontinuity of $\Theta$ implies the existence of a subsequence $(x^{(n_k)},a^{(n_k)})_{k \in \N}$ and of a sequence $(\theta^{(k)})_{k \in \N}$ with $\theta^{(k)} \in \Theta(x^{(n_k)},a^{(n_k)})$ for all $k \in \N$ such that $\theta^{(k)} \rightarrow \theta$ for $k \rightarrow \infty$. Hence, it follows with the continuity of $\widehat{\PP}$ that $\mathcal{P}\left(x^{(n_k)},a^{(n_k)}\right)\ni \PP^{(k)}:=\widehat{\PP}(x^{(n_k)},a^{(n_k)},\theta^{(k)}) \rightarrow \PP$ for $k \rightarrow \infty$, implying with Lemma~\ref{lem_lower_hemi} the lower hemicontinuity of $\mathcal{P}$.

\end{proof}

\subsection{Proof of Results in Section~\ref{sec_uncertainty_autocorr}}
Before reporting the proof of Proposition~\ref{prop_auto_correlation}, we establish the following lemma.
\begin{lem}\label{lem_continuity_dirac_plus_measure}
{ Let $D \in \N$ and let $Z \subseteq \R^D$ be closed. Moreover, }let $\mathcal{D}:= \left\{\delta_x~\middle|~x\in { Z}^{m-1}\right\} \subseteq \left(\mathcal{M}_1({ Z}^{m-1}), \tau_p \right)$ be the closed subset consisting of all Dirac measures on ${ Z}^{m-1}$. Then, for any $p\in \{0,1\}$, the map
\begin{align*}
\varphi:\left(\mathcal{D},\tau_p \right) \times \left(\mathcal{M}_1({ Z}),\tau_p \right) &\rightarrow \left(\mathcal{M}_1({ Z}^m),\tau_p \right)\\
\left(\delta_x,\PP \right) &\mapsto \delta_x \otimes \PP 
\end{align*}
is continuous.
\end{lem}
\begin{proof}
{ We show the sequential continuity of the map $\varphi$.} First, we consider the case $p=0$. Let { $\delta_x \in \mathcal{D}$ for some $x\in Z^{m-1}$ and $\PP \in \mathcal{M}_1(Z)$, and let }$(\delta_{x^{(n)}})_{n \in \N} \subseteq \mathcal{D}$ and $\left(\PP^{(n)}\right)_{n \in \N} \subseteq \mathcal{M}_1({ Z})$ with 
 $\delta_{x^{(n)}} \xrightarrow{\tau_0} { \delta_x \in \mathcal{D}}$ and $\PP^{(n)} \xrightarrow{\tau_0} \PP\in \mathcal{M}_1({ Z})$ for $n \rightarrow \infty$. Now, let $f:{ Z}^m \rightarrow \R$ be Lipschitz continuous with Lipschitz constant $L >0$. Then we have
\begin{align}
&\lim_{n \rightarrow \infty} \left| \int_{{ Z}^m} f(y,z) \delta_{x^{(n)}}(\D y) \otimes \PP^{(n)}(\D z) -\int_{{ Z}^m} f(y,z) \delta_{x}(\D y) \otimes \PP(\D z)\right| \notag \\
=&\lim_{n \rightarrow \infty} \left| \int_{{ Z}} f(x^{(n)},z) \PP^{(n)}(\D z) -\int_{{ Z}} f(x,z) \PP(\D z)\right|  \notag \\
\leq & \lim_{n \rightarrow \infty} \left( \int_{{ Z}} \left|f(x^{(n)},z) -f(x,z) \right|\PP^{(n)}(\D z) +\left|\int_{{ Z}} f(x,z)\PP^{(n)}(\D z) - \int_{{ Z}} f(x,z) \PP(\D z)\right| \right)  \notag \\
\leq & \lim_{n \rightarrow \infty}  L \cdot \|x^{(n)}-x\| +\lim_{n \rightarrow \infty} \left|\int_{{ Z}} f(x,z)\PP^{(n)}(\D z) - \int_{{ Z}} f(x,z) \PP(\D z)\right|  = 0, \label{eq_proof_lemma_dirac_eq_1}
\end{align}
where the second summand in \eqref{eq_proof_lemma_dirac_eq_1} vanishes due to $\PP^{(n)} \xrightarrow{\tau_0} \PP$.
By \cite[Theorem 18.7]{jacod_protter} we conclude  that $\varphi$ is { (sequential)} continuous.

Now, we consider the case $p=1$. 
Let { $\delta_x \in \mathcal{D}$ for some $x\in Z^{m-1}$ and $\PP \in \mathcal{M}_1(Z)$, and let }$(\delta_{x^{(n)}})_{n \in \N} \subseteq \mathcal{D}$ and $\left(\PP^{(n)}\right)_{n \in \N} \subseteq \mathcal{M}_1({ Z})$ with 
 $\delta_{x^{(n)}} \xrightarrow{\tau_1} \delta_x \in \mathcal{D}$ and $\PP^{(n)} \xrightarrow{\tau_1} \PP \in \mathcal{M}_1({ Z})$ for $n \rightarrow \infty$. Since convergence in $\tau_1$ implies convergence in $\tau_0$,  we obtain, by the already considered case $p=0$, that $\delta_{x^{(n)}}\otimes \PP^{(n)} \xrightarrow{\tau_0} \delta_x \otimes \PP$ for $n \rightarrow \infty$. It remains to show that the convergence also follows with respect to $\tau_1$. To conclude the convergence in $\tau_1$ it suffices, by \cite[Theorem 6.9]{villani2009optimal}, to show that 
\[
\lim_{n \rightarrow \infty} \int_{{ Z}^m} \|(y,z)\|\delta_{x^{(n)}}(\D y)\otimes \PP^{(n)}(\D z) = \int_{{ Z}^m} \|(y,z)\|\delta_{x}(\D y)\otimes \PP(\D z).
\]
To see this, note that
\begin{align*}
&\lim_{n \rightarrow \infty} \left|\int_{{ Z}^m} \|(y,z)\|\delta_{x^{(n)}}(\D y)\otimes \PP^{(n)}(\D z)-\int_{{ Z}^m} \|(y,z)\|\delta_{x}(\D y)\otimes \PP(\D z)\right| \\
=&\lim_{n \rightarrow \infty} \left|\int_{{ Z}} \|(x^{(n)},z)\|\PP^{(n)}(\D z)-\int_{{ Z}} \|(x,z)\|\PP(\D z)\right|\\
\leq &\lim_{n \rightarrow \infty} \left(\int_{{ Z}} \left|\|(x^{(n)},z)\|-\|(x,z)\|\right|\PP^{(n)}(\D z)+\left|\int_{{ Z}} \|(x,z)\|\PP^{(n)}(\D z)-\int_{{ Z}} \|(x,z)\|\PP(\D z)\right|\right)\\
\leq &\lim_{n \rightarrow \infty} \left(\|x^{(n)}-x\|+\left|\int_{{ Z}} \|(x,z)\|\PP^{(n)}(\D z)-\int_{{ Z}} \|(x,z)\|\PP(\D z)\right|\right)=0,
\end{align*}
where we use that ${ Z} \ni y \mapsto \|(x,y)\| \in C_1({ Z},\R)$ and $\PP^{(n)} \xrightarrow{\tau_1} \PP$ for $n \rightarrow \infty$.
\end{proof}

\begin{proof}[Proof of Proposition~\ref{prop_auto_correlation}]
Let $(x,a) \in \Omloc \times A$.\\
It is immediate that $\mathcal{P}(x,a) \neq \emptyset$, since $\widetilde{\mathcal{P}}(x,a)\neq \emptyset$ by assumption.

To show the compactness of $\mathcal{P}(x,a)$ we consider a sequence $\left(\PP^{(n)}\right)_{n \in \N} \subseteq \mathcal{P}(x,a)$, where for all $n \in \N$ we have $\PP^{(n)}=\delta_{\pi(x)} \otimes \widetilde{\PP}^{(n)}$ for some $\widetilde{\PP}^{(n)} \in \widetilde{\mathcal{P}}(x,a)$. Then, by the compactness of $\widetilde{\mathcal{P}}(x,a)$ there exists a subsequence $\left(\widetilde{\PP}^{(n_k)}\right)_{k \in \N}$ such that $\widetilde{\PP}^{(n_k)} \rightarrow \widetilde{\PP}\in \widetilde{\mathcal{P}}(x,a)$ in $\tau_p$ as $k \rightarrow \infty$ . Now, let $g \in C_p(\Omloc, \R)$. Then the map ${ Z} \ni y \mapsto g(\pi(x),y)$ is contained in $C_p({ Z},\R)$, and hence
\begin{align*}
\lim_{k \rightarrow \infty}\int_{\Omloc} g(z) \PP^{(n_k)}(\D z)&= \lim_{k \rightarrow \infty}\int_{{ Z}} g(\pi(x),y)\widetilde{\PP}^{(n_k)}(\D y)\\\
&= \int_{{ Z}} g(\pi(x),y)\widetilde{\PP}(\D y) =\int_{\Omloc} g(z) \PP(\D z)
\end{align*}
for $\PP:= \delta_{\pi(x)} \otimes \widetilde{\PP} \in {\mathcal{P}}(x,a)$, which proves the compactness of  $\mathcal{P}(x,a)$.

To show the upper hemicontinuity of $\mathcal{P}$, let $(x^{(n)},a^{(n)}) \subseteq \Omloc \times A$ with $(x^{(n)},a^{(n)}) \rightarrow (x,a)$ for $n \rightarrow \infty$, and let $\PP^{(n)} \in \mathcal{P}(x^{(n)},a^{(n)})$ for all $n \in \N$. Then, we have a representation $\PP^{(n)}=\delta_{\pi(x^{(n)})}\otimes \widetilde{\PP}^{(n)}$ with $\widetilde{\PP}^{(n)} \in \widetilde{\mathcal{P}}(x^{(n)},a^{(n)})$ for all $n \in \N$. By the upper hemicontinuity of $\widetilde{\mathcal{P}}$, there exists, according to Lemma~\ref{lem_upper_hemi}, a subsequence $\left(\widetilde{\PP}^{(n_k)}\right)_{k \in \N}$ with $\widetilde{\PP}^{(n_k)} \rightarrow \widetilde{\PP} \in \widetilde{\mathcal{P}}(x,a)$ in $\tau_p$ as $k \rightarrow \infty$. Moreover $ \delta_{\pi\left(x^{(n)}\right)} \rightarrow \delta_{\pi\left(x\right)}$ in $\tau_1$ as $n \rightarrow \infty$.

We apply Lemma~\ref{lem_continuity_dirac_plus_measure} and obtain that
 $\delta_{\pi\left(x^{(n_k)}\right)}\otimes \widetilde{\PP}^{(n_k)} \rightarrow \delta_{\pi(x)} \otimes \widetilde{\PP} \in \mathcal{P}(x,a)$, and hence the upper hemicontinuity follows with Lemma~\ref{lem_upper_hemi}.

To prove the lower hemicontinuity of $\mathcal{P}$ we consider again a sequence $(x^{(n)},a^{(n)}) \subseteq \Omloc \times A$ with $(x^{(n)},a^{(n)}) \rightarrow (x,a)$ for $n \rightarrow \infty$, and some $\PP \in \mathcal{P}(x,a)$ with a representation $\PP= \delta_{\pi(x)} \otimes \widetilde{\PP}$ for $\widetilde{\PP} \in \widetilde{\mathcal{P}}(x,a)$. By the lower hemicontinuity of $\widetilde{\mathcal{P}}$ there exists, according to Lemma~\ref{lem_lower_hemi}, a subsequence $(x^{(n_k)},a^{(n_k)})_{k \in \N}$ and $\widetilde{\PP}^{(n_k)} \in \widetilde{\mathcal{P}}(x^{(n_k)},a^{(n_k)})$ for all $k \in \N$ such that $\widetilde{\PP}^{(n_k)} \rightarrow \widetilde{\PP}$ in $\tau_p$. Then, we set $\PP^{(n_k)}:= \delta_{\pi(x^{(n_k)})} \otimes \widetilde{\PP}^{(n_k)}  \in \mathcal{P}(x^{(n_k)},a^{(n_k)})$ for all $k \in \N$, and we conclude $\PP^{(n_k)} \rightarrow \PP$ in $\tau_p$ for $k \rightarrow \infty$ with Lemma~\ref{lem_continuity_dirac_plus_measure}. Hence, the lower hemicontinuity follows with Lemma~\ref{lem_lower_hemi}.
\end{proof}

\section{Proof of Results in Section~\ref{sec_portfolio_optimization}}
\label{sec_proofs_sec4}
\subsection{Proof of Results in Section~\ref{sec_portfolio_wasserstein}}

\begin{proof}[Proof of Proposition~\ref{prop_portfolio_wasserstein}]
Since Assumption~\ref{asu_p}~(ii) is automatically fulfilled for $p=0$, the fulfilment of Assumption~\ref{asu_p} follows from Proposition~\ref{prop_wasserstein} and Proposition~\ref{prop_auto_correlation}, once we have shown that $\Omloc \ni x \mapsto \widehat{\PP}(x) \in \mathcal{M}_1({ Z} )$ is continuous in $\tau_q$ and possesses finite $q$-th moments.

To show the { (sequential)} continuity of $\widehat{\PP}$, { let $X_t \in \Omloc$ and let } $(X_t^{(n)})_{n\in \N}\subseteq \Omloc$ { be a sequence} with $X_t^{(n)} \rightarrow X_t \in \Omloc$ for $n \rightarrow \infty$.
By construction $\Omloc \ni x \mapsto\pi_s(x)\in [0,1]$ is continuous for all $s=m,\dots,N-1$, which implies for all $g \in C_q({ Z} ,\R)$ that
 \begin{align*}
 \lim_{n \rightarrow \infty} \int_{{ Z} } g(y) \widehat{\PP}\left(X_t^{(n)}; \D y \right)&= \lim_{n \rightarrow \infty}  \sum_{s=m}^{N-1}\pi_s(X_t^{(n)})g(\mathscr{R}_{s+1})\\
 & = \sum_{s=m}^{N-1}\pi_s(X_t)g(\mathscr{R}_{s+1}) =\int_{{ Z} } g(y) \widehat{\PP}\left(X_t;\D y\right).
 \end{align*}
Moreover, the existence of the $q$-th moment follows by 
  \begin{align*}
 \int_{{ Z} } \|y\|^q\widehat{\PP}\left(X_t; \D y \right)&=   \sum_{s=m}^{N-1}\pi_s(X_t)\cdot \left\|\mathscr{R}_{s+1}\right\|^q< \infty.
 \end{align*}

Now, to verify Assumption~\ref{asu_2} note that $r$ is continuous and that the compactness of ${ Z} $ and of $A$ imply that $r$ is bounded, and thus Assumption~\ref{asu_2}~(i) and (iii) are fulfilled. Next, let $X_t,X_t'\in \Omloc$, $X_{t+1}=\left(\mathcal{R}_{t-m+2},\cdots,\mathcal{R}_{t+1}\right)\in \Omloc$, and let $a_t,a_t' \in A$.
Then, by the Cauchy--Schwarz inequality we see that 
\begin{align*}
\left|r(X_t,a_t,X_{t+1})-r(X_t',a_t',X_{t+1})\right| &= \left| \sum_{i=1}^D (a_t^i-{a'_t}^i) \mathcal{R}_{t+1}^i { - \lambda \cdot  \left(a_t^T \cdot \Sigma_{\mathcal{R}} \cdot a_t -a_t'^T \cdot \Sigma_{\mathcal{R}} \cdot a_t'\right)}\right| \\
&\leq \|\mathcal{R}_{t+1}\|\cdot  \|a_t-a_t'\| { + \lambda \cdot \left( \left|a_t^T \cdot \Sigma_{\mathcal{R}} \cdot (a_t-a_t') \right| + \left|(a_t-a_t')^T \cdot \Sigma_{\mathcal{R}} \cdot a_t'\right| \right)}\\
&\leq  \max_{z\in Z}\|z\| \cdot \|a_t-a_t'\| { + \lambda \cdot 2 \cdot \max_{b \in A} \|b\| \cdot  \|\Sigma_{\mathcal{R}}\|_F \cdot \|a_t-a_t'\|}\\   
& { = \left(\max_{z\in Z}\|z\|+ \lambda \cdot 2 \cdot \max_{b \in A} \|b\| \cdot  \|\Sigma_{\mathcal{R}}\|_F \right) \cdot \|a_t-a_t'\| ,}
\end{align*}
{ where 
$\|\cdot \|_F$ denotes the Frobenius-norm.} Hence, Assumption~\ref{asu_2}~(ii) is fulfilled.
\end{proof}

\subsection{Proof of Results in Section~\ref{sec_portfolio_knightian}}
\begin{proof}[Proof of Proposition~\ref{prop_parametric}]
To show Assumption~\ref{asu_p}~(ii), let 
\begin{equation}
C_P:=1+\sqrt{\varepsilon^2+\frac{1}{m}+4 \cdot \frac{\varepsilon^2+1}{m-1} } 
\end{equation}
Now we consider some $(x,a) \in \Omloc \times A$ and some $\PP \in \mathcal{P}(x,a)$. Then, we have a representation of the form $\PP=\delta_{\pi(x)} \otimes \widetilde{\PP}$ for some $\widetilde{\PP} \in \widetilde{\mathcal{P}}(x)$, where $\widetilde{\PP} \sim \mathcal{N}_D(\mu, \Sigma)$ with $(\mu, \Sigma) \in \R^D \times \R^{D \times D}$ fulfilling $\|\mu - \mathfrak{m}(x) || \leq \varepsilon$ and $\Sigma = \mathfrak{c}(y)$ for some $y \in \Omloc$ with $\|y- x\| \leq \varepsilon$.
Therefore, we have with Jensen's inequality that 
\begin{equation}\label{eq_proof_p_asu_portfolio_1}
\begin{aligned}
\int_{\Omloc} 1+ \|y\| \PP(\D y)&= 1+\int_{\R^D} \left\|\left(\pi(x),z\right) \right\| \widetilde{\PP}(\D z)\\
 &\leq 1+\int_{\R^D} \left\|\pi(x)\right\| + \left\|z \right\| \widetilde{\PP}(\D z)\\
  &\leq 1+\|x\| + \sqrt{\int_{\R^D}\left\|z \right\|^2 \widetilde{\PP}(\D z)}.
\end{aligned}
\end{equation}
Moreover, for $Z = (Z_1,\dots,Z_D) \sim \mathcal{N}_D(\mu, \Sigma)$ we have
\begin{equation}\label{eq_proof_p_asu_portfolio_2}
\begin{aligned}
\int_{\R^D}\left\|z \right\|^2 \widetilde{\PP}(\D z) =\E\left[\left\|Z\right\|^2\right]&=\E\left[\sum_{i=1}^D Z_i^2\right]\\
&=\sum_{i=1}^D \E\left[ Z_i\right]^2+\sum_{i=1}^D \Var \left(Z_i\right)
\\
&= \left\|\mu\right\|^2 +\operatorname{trace} \left(\Sigma\right).
\end{aligned}
\end{equation}
In the next step, we write $x=(x_i^{(j)})_{i=1,\dots,m}^{j=1,\dots,D}$, use the Cauchy--Schwarz inequality, and compute 
\begin{equation}\label{eq_proof_p_asu_portfolio_3}
\begin{aligned}
\left\|\mu\right\|^2 &\leq \left(\|\mu- \mathfrak{m}(x)\|+\|\mathfrak{m}(x)\|\right)^2\\
&\leq \left(\varepsilon+\sqrt{\sum_{j=1}^D \left(\frac{1}{m}\sum_{i=1}^m x_i^{(j)} \right)^2}\right)^2\\
&\leq \left(\varepsilon+\sqrt{\sum_{j=1}^D \left(\frac{1}{m} \sum_{i=1}^m \left(x_i^{(j)}\right)^2 \right)}\right)^2\\
&= \left(\varepsilon+\frac{1}{\sqrt{m}}\|x\|\right)^2 \leq \left(\varepsilon^2+\frac{1}{m}\right)\left(1+\|x\|^2\right).
\end{aligned}
\end{equation}
Further, we write $y=(y_i^{(j)})_{i=1,\dots,m}^{j=1,\dots,D}$, and obtain with the Cauchy--Schwarz inequality that
\begin{equation}\label{eq_proof_p_asu_portfolio_4}
\begin{aligned}
\|\mathfrak{m}(y)-\mathfrak{m}(x)\|^2 &= \frac{1}{m^2}\sum_{j=1}^D \left(\sum_{i=1}^m \left(y_i^{(j)}-x_i^{(j)}\right)\right)^2\\
&\leq \frac{1}{m} \sum_{j=1}^D \sum_{i=1}^m \left(y_i^{(j)}-x_i^{(j)}\right)^2=\frac{\|y-x\|^2}{m} \leq \frac{\varepsilon^2}{m}.
\end{aligned}
\end{equation}
The above inequality \eqref{eq_proof_p_asu_portfolio_4}, and $\|\mathfrak{m}(x)\| \leq \tfrac{1}{\sqrt{m}}\|x\|$ (see also \eqref{eq_proof_p_asu_portfolio_3}) imply together with the Cauchy--Schwarz inequality that
\begin{equation}\label{eq_proof_p_asu_portfolio_5}
\begin{aligned}
\operatorname{trace}(\Sigma)&= \frac{1}{m-1}\sum_{i=1}^m \operatorname{trace} \left((y_i-\mathfrak{m}(y))(y_i-\mathfrak{m}(y))^{T}\right)\\
&= \frac{1}{m-1}\sum_{i=1}^m \sum_{j=1}^D \left(y_i^{(j)}-\mathfrak{m}(y)^{(j)}\right)^2\\
&\leq  \frac{2}{m-1}\sum_{i=1}^m \sum_{j=1}^D \left[\left(y_i^{(j)}\right)^2+\left(\mathfrak{m}(y)^{(j)}\right)^2 \right]\\
&=  \frac{2}{m-1}\|y\|^2+\frac{2m}{m-1}\|\mathfrak{m}(y)\|^2\\
&\leq  \frac{2}{m-1}\left(\|y-x\|+\|x\|\right)^2+\frac{2m}{m-1}\left(\|\mathfrak{m}(y)-\mathfrak{m}(x)\|+\|\mathfrak{m}(x)\|\right)^2\\
&\leq  \frac{2}{m-1}\left(\varepsilon+\|x\|\right)^2+\frac{2m}{m-1}\left(\frac{\varepsilon}{\sqrt{m}}+\frac{1}{\sqrt{m}}\|x\|\right)^2\\
&\leq  \frac{2}{m-1}\left(\varepsilon^2+1\right)\left(1+\|x\|^2\right)+\frac{2m}{m-1}\cdot \frac{\varepsilon^2+1}{{m}} \cdot  (1+\|x\|^2) \\
&= 4\cdot \frac{\varepsilon^2+1}{m-1} \cdot \left(1+\|x\|^2\right).
\end{aligned}
\end{equation}
Hence, by combining \eqref{eq_proof_p_asu_portfolio_1}, \eqref{eq_proof_p_asu_portfolio_2}, \eqref{eq_proof_p_asu_portfolio_3}, and \eqref{eq_proof_p_asu_portfolio_5} we have
\begin{equation*}
\begin{aligned}
\int_{\Omloc} 1+ \|y\| \PP(\D y) &\leq 1+\|x\|+\sqrt{\left(\varepsilon^2+\frac{1}{m}\right)\left(1+\|x\|^2\right)+4\cdot \frac{\varepsilon^2+1}{m-1} \cdot \left(1+\|x\|^2\right)}\\
&\leq \left(1+\sqrt{\varepsilon^2+\frac{1}{m}+4 \cdot \frac{\varepsilon^2+1}{m-1} } \right)\cdot \left(1+\|x\|\right)\\
&=C_P\cdot \left(1+\|x\|\right),
\end{aligned}
\end{equation*}
as required in Assumption~\ref{asu_p}~(ii).

Since Assumption~\ref{asu_p}~(ii) is fulfilled, the fulfilment of Assumption~\ref{asu_p} follows now with an application of Proposition~\ref{prop_auto_correlation}. 
Thus, to verify the assumptions of Proposition~\ref{prop_auto_correlation}, we need to show that $\Omloc \ni x \mapsto \widetilde{\mathcal{P}}(x) \twoheadrightarrow \left(\mathcal{M}_1(\R^D),\tau_1 \right)$ is nonempty, compact-valued, and continuous. This, in turn follows from Proposition~\ref{prop_knightian_1} once we have shown that $\Omloc \ni x \twoheadrightarrow \Theta(x) \subseteq \R^D \times \R^{D \times D}$ is nonempty, compact-valued, and continuous and that 
\begin{equation}\label{eq_map_proof_param_example_1}
\begin{aligned}
\{(x,{\mu},{\Sigma}) ~|~x \in \Omloc, ~ ({\mu},{\Sigma}) \in \Theta(x) \} &\rightarrow (\mathcal{M}_1(\R^D), \tau_1)\\
(x,{\mu},{\Sigma}) &\mapsto \widehat{\PP}(x,{\mu},{\Sigma}):= \mathcal{N}_D( {\mu},{\Sigma}).
\end{aligned}
\end{equation}
is continuous.

To that end, let $x \in \Omloc$. \\
The non-emptiness of $\Theta(x)$ follows by definition. 

To show the compactness of $\Theta(x)$, let $(\mu^{(n)},\Sigma^{(n)})_{n \in \N} \subseteq \Theta(x)$. Then, we have $\|\mu^{(n)}- \mathfrak{m}(x)\| \leq \varepsilon$ for all $n \in \N$ as well as $\Sigma^{(n)}=\mathfrak{c}(y^{(n)})$ for some $y^{(n)} \in \Omloc$ with $\|y^{(n)} - x \| \leq \varepsilon$. Then, according to the Bolzano--Weierstrass theorem there exists a subsequence $(\mu^{(n_k)},y^{(n_k)})_{k \in \N} \subseteq \R^D \times \Omloc$ such that $y^{(n_k)} \rightarrow y \in \Omloc$ with $\|y-x\| \leq \varepsilon$, and $\mu^{(n_k)} \rightarrow \mu \in \R^D$ with $\|\mu- \mathfrak{m}(x)\| \leq \varepsilon$ for $k \rightarrow \infty$. Since $\mathfrak{c}$ is continuous we obtain that $(\mu^{(n_k)}, \Sigma^{(n_k)}) \rightarrow (\mu, \Sigma):= (\mu, \mathfrak{c}(y)) \in \Theta(x)$ for $k \rightarrow \infty$.

To show the upper hemicontinuity of $\Theta$, let $(x^{(n)})_{n \in \N} \subseteq \Omloc $ with $(x^{(n)}) \rightarrow x \in \Omloc $ for $n \rightarrow \infty$ as well as $(\mu^{(n)}, \Sigma^{(n)})_{n \in \N}$ with $\left(\mu^{(n)}, \Sigma^{(n)}\right) \in \Theta(x^{(n)})$ for all $n \in \N$. We have for all $n \in \N$ that $\|\mu^{(n)}- \mathfrak{m}(x^{(n)})\| \leq \varepsilon$ and that $\Sigma^{(n)}=\mathfrak{c}(y^{(n)})$ with $||y^{(n)}-x^{(n)}\| \leq \varepsilon$ for some $y^{(n)} \in \Omloc$. Therefore, since $ \left\|\mu^{(n)}- \mathfrak{m}(x) \right\| \leq \left\| \mu^{(n)}- \mathfrak{m}(x^{(n)})\right\|+\left\|\mathfrak{m}(x^{(n)}) -\mathfrak{m}(x)\right\|$, the continuity of $\mathfrak{m}$ ensures for every  $n$ large enough that $\|\mu^{(n)} - \mathfrak{m}(x)\| \leq 2\varepsilon$. Hence, there exists according to the Bolzano--Weierstrass theorem  a subsequence $(\mu^{(n_k)})_{k \in \N}$ with $\mu^{(n_k)} \rightarrow \mu$ for $k \rightarrow \infty$  for some $\mu \in \R^D$. Therefore, since $\mathfrak{m}$ is continuous, we obtain 
\begin{equation}\label{eq_mu-m(x)_1}
\|\mu- \mathfrak{m}(x)\| = \lim_{k \rightarrow \infty} \left\|\mu^{(n_k)}-\mathfrak{m}\left(x^{(n_k)}\right)\right\| \leq \varepsilon.
\end{equation}
Analogously, we have that $\|y^{(n)}-x\|\leq ||y^{(n)}-x^{(n)}\|+\|x^{(n)}-x\|< 2 \varepsilon$ for every $n$ large enough. This implies the existence of a subsequence $(y^{(n_k)})_{k \in \N}$ converging against some $y \in \Omloc$ with $\|y-x\| =\lim_{k \rightarrow \infty} ||y^{(n_k)}-x^{(n_k)}\| \leq \varepsilon$. Then, for $\Sigma:=\mathfrak{c}(y)$ we have $(\mu, \Sigma) \in \Theta(x)$ and $(\mu^{(n_k)},\Sigma^{(n_k)}) \rightarrow (\mu, \Sigma)$ for $k \rightarrow \infty$. Thus, the upper hemicontinuity follows with Lemma~\ref{lem_upper_hemi}.

To show the lower hemicontinuity of $\Theta$ we consider a sequence  $(x^{(n)})_{n \in \N} \subseteq \Omloc$ with $x^{(n)} \rightarrow x \in \Omloc$ for $n \rightarrow \infty$ and some $(\mu, \Sigma) \in \Theta(x)$. We have by definition $\|\mu-\mathfrak{m}(x)|| \leq \varepsilon$ as well as $\Sigma = \mathfrak{c}(y)$ for some $y \in \R^D$ with $\|y-x\|\leq \varepsilon$. We define for every $n \in \N$
\[
\mu^{(n)}:= \left(1-\frac{1}{n}\right)\mu+\frac{1}{n}\mathfrak{m}\big(x^{(n)}\big).
\]
Then, due to the convergence $\mathfrak{m}(x^{(n)}) \rightarrow \mathfrak{m}(x)$, there exists  a subsequence $(x^{(n_k)})_{k \in \N}$ such that for every $k \in \N$ we have $\left\|\mathfrak{m}(x^{(n_k)}) - \mathfrak{m}(x) \right\| < \varepsilon/(n_k-1)$. This implies for all $k \in \N$ that
\begin{equation}\label{eq_convergence_mu_1}
\begin{aligned}
\|\mu^{(n_k)}-\mathfrak{m}(x^{(n_k)})\|&=\left(1-\tfrac{1}{n_k}\right) \left\|\mu - \mathfrak{m}(x^{(n_k)}) \right\|\\
&\leq \left(1-\tfrac{1}{n_k}\right) \bigg(\left\|\mu - \mathfrak{m}(x) \right\| +\left\|\mathfrak{m}(x) - \mathfrak{m}(x^{(n_k)}) \right\|  \bigg)\leq \left(1-\tfrac{1}{n_k}\right)\left(\varepsilon +\frac{\varepsilon}{n_k-1}\right) =  \varepsilon.
\end{aligned}
\end{equation}
Next, we define for all $n \in \N$
\begin{equation}\label{eq_proof_parameter_defn_Sigma_n}
y^{(n)}:= \left(1-\frac{1}{n}\right)y+\frac{1}{n}x^{(n)},\qquad \Sigma^{(n)}:= \mathfrak{c}\left(y^{(n)}\right).
\end{equation}
We obtain by the convergence $x^{(n)}\rightarrow x$ for $n \rightarrow \infty$ the existence of a subsequence $x^{(n_{k_l})}$ such that $\|x^{(n_{k_l})}-x\|< \varepsilon/(n_{k_l}-1)$ for all $l \in \N$.
This implies
\begin{equation}\label{eq_convergence_y_1}
\begin{aligned}
\left\|y^{({n_k}_l)}-x^{({n_k}_l)}\right\|&=\left(1-\frac{1}{{n_k}_l}\right)\cdot \left\|y-x^{({n_k}_l)}\right\| \\
&\leq \left(1-\frac{1}{{n_k}_l}\right)\cdot \left(\left\|y-x\right\|+\left\|x-x^{({n_k}_l)}\right\|\right)\\
&\leq \left(1-\tfrac{1}{n_{k_l}}\right)\left(\varepsilon +\frac{\varepsilon}{n_{k_l}-1}\right) =\varepsilon.
\end{aligned}
\end{equation}
Hence, with \eqref{eq_convergence_mu_1}, \eqref{eq_proof_parameter_defn_Sigma_n}, and \eqref{eq_convergence_y_1}, we have shown the existence of a subsequence $\left(\mu^{({n_k}_l)},\Sigma^{({n_k}_l)}\right)_{l \in \N}$ with $\left(\mu^{({n_k}_l)},\Sigma^{({n_k}_l)}\right) \in \Theta\left(x^{({n_k}_l)}\right)$ for all $l \in \N$ and such that, by the continuity of $r$, $\left(\mu^{({n_k}_l)},\Sigma^{({n_k}_l)}\right) \rightarrow (\mu, \Sigma)$ for $l \rightarrow \infty$. This implies the lower hemicontinuity of $\Theta$ by Lemma~\ref{lem_lower_hemi}.

It remains to show that the map defined in \eqref{eq_map_proof_param_example_1} is continuous with respect to $\tau_1$. 

To that end, consider a sequence $(x^{(n)})_{n \in \N} \subseteq \Omloc $ as well as a sequence $(\mu^{(n)}, \Sigma^{(n)})_{n \in \N}$ with $(\mu^{(n)}, \Sigma^{(n)}) \in \Theta(x^{(n)})$ for all $n \in \N$ and such that 
$\left(x^{(n)},\mu^{(n)},\Sigma^{(n)}\right) \rightarrow (x, \mu, \Sigma ) \in \Omloc \times \Theta(x)$ for $n \rightarrow \infty$. Then we write $\PP^{(n)}:= \mathcal{N}_D(\mu^{(n)}, \Sigma^{(n)}) \in \mathcal{M}_1(\R^D)$ for $n \in \N$ as well as $\PP:= \mathcal{N}_D(\mu, \Sigma) \in \mathcal{M}_1(\R^D)$. The characteristic function of $\PP^{(n)}$, denoted by 
\[
\R^D \ni u \mapsto \varphi_{\PP^{(n)}}(u):= \exp\left(i u^T \mu^{(n)}- \tfrac{1}{2} u^T \Sigma^{(n)} u\right)
\]
converges for $n \rightarrow \infty $ pointwise against 
\[
\R^D \ni u \mapsto \varphi_{\PP}(u):= \exp\left(i u^T \mu- \tfrac{1}{2} u^T \Sigma u\right),
\]
which is the characteristic function of $\PP$, and hence by Lévy's continuity theorem (see, e.g., \cite[Theorem 19.1]{jacod_protter}) we have $\PP^{(n)} \rightarrow \PP$ weakly, i.e., in $\tau_0$ for $n \rightarrow \infty$.

The convergence of $\PP^{(n)} \rightarrow \PP$ with respect to $\tau_1$ now follows with, e.g.,  \cite[Example 3.8.15]{bogachev1998gaussian}, since $(\PP^{(n)})_{n \in \N}$, and $\PP$ are Gaussian.

To verify Assumption~\ref{asu_2} first note that $r$ is continuous, and hence  Assumption~\ref{asu_2}~(i) is fulfilled.
Let $X_t,X_t'\in \Omloc$, $X_{t+1}=\left(\mathcal{R}_{t-m+2},\cdots,\mathcal{R}_{t+1}\right)\in \Omloc$, and let $a_t,a_t' \in A$.
Then, the Cauchy--Schwarz inequality implies
\begin{align*}
\left|r(X_t,a_t,X_{t+1})-r(X_t',a_t',X_{t+1})\right| &= \left| \sum_{i=1}^D (a_t^i-{a'_t}^i) \mathcal{R}_{t+1}^i { - \lambda \cdot  \left(a_t^T \cdot \Sigma_{\mathcal{R}} \cdot a_t -a_t'^T \cdot \Sigma_{\mathcal{R}} \cdot a_t'\right)}\right| \\
&\leq \|\mathcal{R}_{t+1}\|\cdot  \|a_t-a_t'\| { + \lambda \cdot \left( \left|a_t^T \cdot \Sigma_{\mathcal{R}} \cdot (a_t-a_t') \right| + \left|(a_t-a_t')^T \cdot \Sigma_{\mathcal{R}} \cdot a_t'\right| \right)}\\
&\leq  \|X_{t+1}\| \cdot \|a_t-a_t'\| { + \lambda \cdot 2 \cdot \max_{b \in A} \|b\| \cdot  \|\Sigma_{\mathcal{R}}\|_F \cdot \|a_t-a_t'\|}\\   &= { \left(\|X_{t+1}\|+ \lambda \cdot 2 \cdot \max_{b \in A} \|b\| \cdot  \|\Sigma_{\mathcal{R}}\|_F \right) \cdot \|a_t-a_t'\|,}
\end{align*}
{ where 
$\|\cdot \|_F$ denotes the Frobenius-norm. This}
implies Assumption~\ref{asu_2}~(ii).
Moreover, we have by using the Cauchy--Schwarz inequality that
\begin{align*}
\left|r(X_t,a_t,X_{t+1})\right| &= \left| \sum_{i=1}^D a_t^i \mathcal{R}_{t+1}^i { - \lambda \cdot  \left(a_t^T \cdot \Sigma_{\mathcal{R}} \cdot a_t \right)} \right| \\
&\leq \|a_t\|   \cdot\|\mathcal{R}_{t+1}\| { + \lambda \|a_t\|^2 \|\Sigma_{\mathcal{R}}\|_F} \\
&{ \leq \max_{b \in A} \|b\|   \cdot\|X_{t+1}\| { + \lambda \max_{b \in A} \|b\|^2 \|\Sigma_{\mathcal{R}}\|_F}}\\
&{ \leq \left(\max_{b \in A} \|b\|    + \lambda \max_{b \in A} \|b\|^2 \|\Sigma_{\mathcal{R}}\|_F \right)\cdot \left(1+ \|X_{t+1}\|\right),}
\end{align*}
as required in Assumption ~\ref{asu_2}~(iii).
\end{proof}

\section*{Acknowledgments}
\noindent
Financial support by the MOE AcRF Tier 1 Grant \emph{RG74/21} and by the  Nanyang Assistant Professorship Grant (NAP Grant) \emph{Machine Learning based Algorithms in Finance and Insurance} is gratefully acknowledged. 

\appendix

\section{Supplementary Results}\label{appendix_supplementary}
The first auxiliary result is Banach's { fixed} point theorem, compare, e.g., \cite[Theorem A 3.5.]{bauerle2011markov}, or any standard monograph on analysis or functional analysis.
\begin{thm}[Banach's { Fixed} Point Theorem]\label{thm_banach}
Let $M$ be a complete metric space with metric $d(x,y)$ and let $\T :M\rightarrow M$ be an operator such that
there exists a number $\beta\in (0,1)$ such that $d(\T v, \T w) \leq \beta d(v,w)$ for all $v,w \in M$. Then, we have that
\begin{itemize}
\item[(i)] $\T $ has a unique { fixed} point $v^*$ in $M$, i.e., $\T v^*=v^*$.
\item[(ii)] $\lim_{n \rightarrow \infty} \T ^nv = v^*$ for all $v \in M$.
\item[(iii)] For $v \in M$ we obtain 
\[
d(v^*,\T ^nv)\leq \frac{\beta^n}{1-\beta}d(\T v,v).
\]
\end{itemize}
\end{thm}

The following result, Berge's Maximum Theorem, can for example be found in \cite[Theorem 17.31]{Aliprantis}.
\begin{thm}[Berge's Maximum Theorem]\label{thm_berge}
Let $\varphi:X \twoheadrightarrow Y$ be a upper and lower hemicontinuous correspondence between topological spaces with nonempty compact values, and suppose that $f: \left\{(x,y) \in X \times Y ~\middle|~y \in \varphi(x)\right\} \rightarrow \R$ is continuous. Then the following holds.
\begin{itemize}
\item[(i)]
The function 
\begin{align*}
m: X &\rightarrow \R\\
x&\mapsto\max_{y \in \varphi(x)}f(x,y)
\end{align*}
is continuous.
\item[(ii)]
The correspondence
\begin{align*}
c: X &\twoheadrightarrow Y \\
x &\mapsto \left\{y \in \varphi(x)~\middle|~f(x,y)=m(x)\right\}
\end{align*}
has nonempty, compact values.
\item[(iii)]
If $Y$ is Hausdorff, then $c$ is upper hemicontinuous.
\end{itemize}
\end{thm}
{ We also provide the assertion of the \emph{measurable maximum theorem}\footnote{{Note that every upper hemicontinuous correspondence is (weakly) measurable, see \cite[Lemma 17.4, Definition 18.1 and Lemma 18.2]{Aliprantis}. Moreover, if $S$ is a topological space and $\Sigma$ its Borel $\sigma$-field, then every continuous function $\Psi: S \times X \rightarrow Z$ is a Caratheodory function, see \cite[Definition 4.50]{Aliprantis}.}}, see, e.g. \cite[Theorem 18.19]{Aliprantis}.
\begin{thm}[Measurable Maximum Theorem]\label{thm_measurable_maximum}
Let $X$ be a separable metrizable space and $(S,\Sigma)$ be a measurable space. Let $\varphi:S \twoheadrightarrow X$ be a weakly measurable correspondence with nonempty compact values, and suppose $f:S\times X \rightarrow \R$ is a Caratheodory function. Define the value fucntion $m:S\rightarrow \R$ by
\begin{align*}
m:S &\rightarrow \R\\
s&\mapsto \max_{x\in \varphi(s)}f(s,x),
\end{align*}
and the correspondence of maximizers by
\begin{align*}
\mu: S &\twoheadrightarrow X\\
s&\mapsto \{x\in \varphi(s)~|~f(s,x)=m(s)\},
\end{align*}
Then the following holds.
\begin{itemize}
\item[(i)] The value function $m$ is measurable.
\item[(ii)] The $\operatorname{argmax}$ correspondence $\mu$ has nonempty and compact values.
\item[(iii)]
The $\operatorname{argmax}$ correspondence $\mu$ is measurable and admits a measurable selector.
\end{itemize}
\end{thm}
}
The following two lemmas provide characterizations of upper and lower hemicontinuity, respectively\footnote{{To illustrate the notions of lower and upper hemicontinuity we also refer to \cite[Example 17.3]{Aliprantis} where examples are provided for correspondences that are upper hemicontinuous but not lower hemicontinuous and vice versa.}}. The results can be found, e.g., in \cite[Theorem 17.20]{Aliprantis},  and \cite[Theorem 17.21]{Aliprantis}.
\begin{lem}[Upper Hemicontinuity]\label{lem_upper_hemi}
Assume that the topological space $X$ is first countable and that $Y$ is metrizable. Then, for a correspondence $\varphi:X \twoheadrightarrow  Y$  the following statements are equivalent.
\begin{itemize}
\item[(i)] The correspondence $\varphi$ is upper hemicontinuous and $\varphi(x)$ is compact for all $x\in X$.
\item[(ii)]
For any $x\in X$, if a sequence $\left((x^{(n)},y^{(n)})\right)_{n \in \N} \subseteq \operatorname{Gr}\varphi$ satisfies $x^{(n)} \rightarrow x$ for $n \rightarrow \infty$, then there exists a subsequence $\left(y^{(n_k)}\right)_{k \in \N}$ with $y^{(n_k)} \rightarrow y  \in \varphi(x)$ for $k \rightarrow \infty$.
\end{itemize}
\end{lem}

\begin{lem}[Lower Hemicontinuity]\label{lem_lower_hemi}
For a correspondence $\varphi:X \twoheadrightarrow  Y$ between first countable topological spaces the following statements are equivalent.
\begin{itemize}
\item[(i)]
The correspondence $\varphi$ is lower hemicontinuous.
\item[(ii)]
For any $x\in X$, if $x^{(n)} \rightarrow x$ for $n \rightarrow \infty$, then for each $y \in \varphi(x)$ there exists a subsequence $\left(x^{(n_k)}\right)_{k \in \N}$ and elements $y^{(k)} \in \varphi\left(x^{(n_k)}\right)$ for each $k\in \N$ such that $y^{(k)} \rightarrow y$ for $k \rightarrow \infty$.
\end{itemize}
\end{lem}

\bibliographystyle{plain} 
\bibliography{literature}

\begin{thebibliography}{10}

\bibitem{aguirregabiria2002swapping}
Victor Aguirregabiria and Pedro Mira.
\newblock Swapping the nested fixed point algorithm: A class of estimators for
  discrete {M}arkov decision models.
\newblock {\em Econometrica}, 70(4):1519--1543, 2002.

\bibitem{Aliprantis}
Charalambos~D. Aliprantis and Kim~C. Border.
\newblock {\em Infinite dimensional analysis}.
\newblock Springer, Berlin, third edition, 2006.
\newblock A hitchhiker's guide.

\bibitem{angiuli2022reinforcement}
Andrea Angiuli, Nils Detering, Jean-Pierre Fouque, and Jimin Lin.
\newblock Reinforcement learning algorithm for mixed mean field control games.
\newblock {\em arXiv preprint arXiv:2205.02330}, 2022.

\bibitem{angiuli2021reinforcement}
Andrea Angiuli, Jean-Pierre Fouque, and Mathieu Lauriere.
\newblock Reinforcement learning for mean field games, with applications to
  economics.
\newblock {\em arXiv preprint arXiv:2106.13755}, 2021.

\bibitem{bartl2019exponential}
Daniel Bartl.
\newblock Exponential utility maximization under model uncertainty for
  unbounded endowments.
\newblock {\em The Annals of Applied Probability}, 29(1):577--612, 2019.

\bibitem{bartl2019robust}
Daniel Bartl, Patrick Cheridito, and Michael Kupper.
\newblock Robust expected utility maximization with medial limits.
\newblock {\em Journal of Mathematical Analysis and Applications},
  471(1-2):752--775, 2019.

\bibitem{bartl2021duality}
Daniel Bartl, Michael Kupper, and Ariel Neufeld.
\newblock Duality theory for robust utility maximisation.
\newblock {\em Finance and Stochastics}, 25(3):469--503, 2021.

\bibitem{bauerle2021q}
Nicole B{\"a}uerle and Alexander Glauner.
\newblock Q-learning for distributionally robust {M}arkov decision processes.
\newblock In {\em Modern Trends in Controlled Stochastic Processes:}, pages
  108--128. Springer, 2021.

\bibitem{bauerle2009mdp}
Nicole B{\"a}uerle and Ulrich Rieder.
\newblock {MDP} algorithms for portfolio optimization problems in pure jump
  markets.
\newblock {\em Finance and Stochastics}, 13(4):591--611, 2009.

\bibitem{bauerle2011markov}
Nicole B{\"a}uerle and Ulrich Rieder.
\newblock {\em {M}arkov decision processes with applications to finance}.
\newblock Springer Science \& Business Media, 2011.

\bibitem{berge}
Claude Berge.
\newblock {\em Espaces topologiques: {F}onctions multivoques}.
\newblock Collection Universitaire de Math\'{e}matiques, Vol. III. Dunod,
  Paris, 1959.

\bibitem{bertoluzzo2012reinforcement}
Francesco Bertoluzzo and Marco Corazza.
\newblock Reinforcement learning for automatic financial trading: Introduction
  and some applications.
\newblock {\em University Ca'Foscari of Venice, Dept. of Economics Research
  Paper Series No}, 33, 2012.

\bibitem{biagini2017robust}
Sara Biagini and Mustafa~{\c{C}} P{\i}nar.
\newblock The robust {M}erton problem of an ambiguity averse investor.
\newblock {\em Mathematics and Financial Economics}, 11(1):1--24, 2017.

\bibitem{blanchard2018multiple}
Romain Blanchard and Laurence Carassus.
\newblock Multiple-priors optimal investment in discrete time for unbounded
  utility function.
\newblock {\em The Annals of Applied Probability}, 28(3):1856--1892, 2018.

\bibitem{blanchet2021distributionally}
Jose Blanchet, Lin Chen, and Xun~Yu Zhou.
\newblock Distributionally robust mean-variance portfolio selection with
  {W}asserstein distances.
\newblock {\em Management Science}, 2021.

\bibitem{blanchet2019robust}
Jose Blanchet, Yang Kang, and Karthyek Murthy.
\newblock Robust {W}asserstein profile inference and applications to machine
  learning.
\newblock {\em Journal of Applied Probability}, 56(3):830--857, 2019.

\bibitem{bogachev1998gaussian}
Vladimir~Igorevich Bogachev.
\newblock {\em Gaussian measures}.
\newblock Number~62. American Mathematical Soc., 1998.

\bibitem{boyd2017multi}
Stephen Boyd, Enzo Busseti, Steve Diamond, Ronald~N Kahn, Kwangmoo Koh, Peter
  Nystrup, and Jan Speth.
\newblock Multi-period trading via convex optimization.
\newblock {\em Foundations and Trends{\textregistered} in Optimization},
  3(1):1--76, 2017.

\bibitem{cao2021deep}
Jay Cao, Jacky Chen, John Hull, and Zissis Poulos.
\newblock Deep hedging of derivatives using reinforcement learning.
\newblock {\em The Journal of Financial Data Science}, 3(1):10--27, 2021.

\bibitem{chang2017incorporating}
Ying-Hua Chang and Ming-Sheng Lee.
\newblock Incorporating {M}arkov decision process on genetic algorithms to
  formulate trading strategies for stock markets.
\newblock {\em Applied Soft Computing}, 52:1143--1153, 2017.

\bibitem{chau2019robust}
Huy~N Chau and Mikl{\'o}s R{\'a}sonyi.
\newblock Robust utility maximisation in markets with transaction costs.
\newblock {\em Finance and Stochastics}, 23(3):677--696, 2019.

\bibitem{chen2019distributionally}
Zhi Chen, Pengqian Yu, and William~B Haskell.
\newblock Distributionally robust optimization for sequential decision-making.
\newblock {\em Optimization}, 68(12):2397--2426, 2019.

\bibitem{denis2013optimal}
Laurent Denis and Magali Kervarec.
\newblock Optimal investment under model uncertainty in nondominated models.
\newblock {\em SIAM Journal on control and optimization}, 51(3):1803--1822,
  2013.

\bibitem{ding1993long}
Zhuanxin Ding, Clive~WJ Granger, and Robert~F Engle.
\newblock A long memory property of stock market returns and a new model.
\newblock {\em Journal of empirical finance}, 1(1):83--106, 1993.

\bibitem{dixon2020machine}
Matthew~F Dixon, Igor Halperin, and Paul Bilokon.
\newblock {\em Machine Learning in Finance}.
\newblock Springer, 2020.

\bibitem{du2020deep}
Jiayi Du, Muyang Jin, Petter~N Kolm, Gordon Ritter, Yixuan Wang, and Bofei
  Zhang.
\newblock Deep reinforcement learning for option replication and hedging.
\newblock {\em The Journal of Financial Data Science}, 2(4):44--57, 2020.

\bibitem{du2020new}
Ningning Du, Yankui Liu, and Ying Liu.
\newblock A new data-driven distributionally robust portfolio optimization
  method based on {W}asserstein ambiguity set.
\newblock {\em IEEE Access}, 9:3174--3194, 2020.

\bibitem{filos2019reinforcement}
Angelos Filos.
\newblock {\em Reinforcement learning for portfolio management}.
\newblock PhD thesis, Imperial College London, 2019.

\bibitem{fouque2016portfolio}
Jean-Pierre Fouque, Chi~Seng Pun, and Hoi~Ying Wong.
\newblock Portfolio optimization with ambiguous correlation and stochastic
  volatilities.
\newblock {\em SIAM Journal on Control and Optimization}, 54(5):2309--2338,
  2016.

\bibitem{gold2003fx}
Carl Gold.
\newblock {FX} trading via recurrent reinforcement learning.
\newblock In {\em 2003 IEEE International Conference on Computational
  Intelligence for Financial Engineering, 2003. Proceedings.}, pages 363--370.
  IEEE, 2003.

\bibitem{guo2022robust}
Ivan Guo, Nicolas Langren{\'e}, Gr{\'e}goire Loeper, and Wei Ning.
\newblock Robust utility maximization under model uncertainty via a
  penalization approach.
\newblock {\em Mathematics and Financial Economics}, 16(1):51--88, 2022.

\bibitem{gut2009multivariate}
Allan Gut.
\newblock The multivariate normal distribution.
\newblock In {\em An Intermediate Course in Probability}, pages 117--145.
  Springer, 2009.

\bibitem{halperin2020qlbs}
Igor Halperin.
\newblock {QLBS}: {Q}-learner in the {B}lack-{S}choles (-{M}erton) worlds.
\newblock {\em The Journal of Derivatives}, 28(1):99--122, 2020.

\bibitem{hambly2021recent}
Ben Hambly, Renyuan Xu, and Huining Yang.
\newblock Recent advances in reinforcement learning in finance.
\newblock {\em arXiv preprint arXiv:2112.04553}, 2021.

\bibitem{hornik1991approximation}
Kurt Hornik.
\newblock Approximation capabilities of multilayer feedforward networks.
\newblock {\em Neural networks}, 4(2):251--257, 1991.

\bibitem{hu2019deep}
Yuh-Jong Hu and Shang-Jen Lin.
\newblock Deep reinforcement learning for optimizing finance portfolio
  management.
\newblock In {\em 2019 Amity International Conference on Artificial
  Intelligence (AICAI)}, pages 14--20. IEEE, 2019.

\bibitem{ionescu_tulceau}
C.~T. Ionescu~Tulcea.
\newblock Mesures dans les espaces produits.
\newblock {\em Atti Accad. Naz. Lincei Rend. Cl. Sci. Fis. Mat. Nat. (8)},
  7:208--211 (1950), 1949.

\bibitem{ismail2019robust}
Amine Ismail and Huy{\^e}n Pham.
\newblock Robust {M}arkowitz mean-variance portfolio selection under ambiguous
  covariance matrix.
\newblock {\em Mathematical Finance}, 29(1):174--207, 2019.

\bibitem{jacod_protter}
Jean Jacod and Philip Protter.
\newblock {\em Probability essentials}.
\newblock Universitext. Springer-Verlag, Berlin, second edition, 2003.

\bibitem{kingma2014adam}
Diederik~P Kingma and Jimmy Ba.
\newblock Adam: A method for stochastic optimization.
\newblock {\em arXiv preprint arXiv:1412.6980}, 2014.

\bibitem{klenke2013probability}
Achim Klenke.
\newblock {\em Probability theory: a comprehensive course}.
\newblock Springer Science \& Business Media, 2013.

\bibitem{knight1921risk}
Frank~Hyneman Knight.
\newblock {\em Risk, uncertainty and profit}, volume~31.
\newblock Houghton Mifflin, 1921.

\bibitem{kuratowski1930probleme}
Casimir Kuratowski.
\newblock Sur le problème des courbes gauches en topologie.
\newblock {\em Fundamenta mathematicae}, 15(1):271--283, 1930.

\bibitem{lecun2015deep}
Yann LeCun, Yoshua Bengio, and Geoffrey Hinton.
\newblock Deep learning.
\newblock {\em nature}, 521(7553):436--444, 2015.

\bibitem{li2009learning}
Yuxi Li, Csaba Szepesvari, and Dale Schuurmans.
\newblock Learning exercise policies for {A}merican options.
\newblock In {\em Artificial Intelligence and Statistics}, pages 352--359.
  PMLR, 2009.

\bibitem{liang2020robust}
Zongxia Liang and Ming Ma.
\newblock Robust consumption-investment problem under {CRRA} and {CARA}
  utilities with time-varying confidence sets.
\newblock {\em Mathematical Finance}, 30(3):1035--1072, 2020.

\bibitem{lin2021optimal}
Qian Lin and Frank Riedel.
\newblock Optimal consumption and portfolio choice with ambiguous interest
  rates and volatility.
\newblock {\em Economic Theory}, 71(3):1189--1202, 2021.

\bibitem{lin2020horizon}
Qian Lin, Xianming Sun, and Chao Zhou.
\newblock Horizon-unbiased investment with ambiguity.
\newblock {\em Journal of Economic Dynamics and Control}, 114:103896, 2020.

\bibitem{lutkebohmert2021robust}
Eva L{\"u}tkebohmert, Thorsten Schmidt, and Julian Sester.
\newblock Robust deep hedging.
\newblock {\em Quantitative Finance}, 2022.

\bibitem{matoussi2012robust}
Anis Matoussi, Dylan Possama\"{\i}, and Chao Zhou.
\newblock Robust utility maximization in nondominated models with 2{BSDE}: the
  uncertain volatility model.
\newblock {\em Math. Finance}, 25(2):258--287, 2015.

\bibitem{moody1998performance}
John Moody, Lizhong Wu, Yuansong Liao, and Matthew Saffell.
\newblock Performance functions and reinforcement learning for trading systems
  and portfolios.
\newblock {\em Journal of Forecasting}, 17(5-6):441--470, 1998.

\bibitem{neufeld2018robust}
Ariel Neufeld and Marcel Nutz.
\newblock Robust utility maximization with {L{\'e}vy} processes.
\newblock {\em Mathematical Finance}, 28(1):82--105, 2018.

\bibitem{neufeld2022detecting}
Ariel Neufeld, Julian Sester, and Daiying Yin.
\newblock Detecting data-driven robust statistical arbitrage strategies with
  deep neural networks.
\newblock {\em arXiv preprint arXiv:2203.03179}, 2022.

\bibitem{neufeld2019nonconcave}
Ariel Neufeld and Mario {\v{S}}iki{\'c}.
\newblock Nonconcave robust optimization with discrete strategies under
  {K}nightian uncertainty.
\newblock {\em Mathematical Methods of Operations Research}, 90(2):229--253,
  2019.

\bibitem{obloj2021distributionally}
Jan Ob{\l}{\'o}j and Johannes Wiesel.
\newblock Distributionally robust portfolio maximization and marginal utility
  pricing in one period financial markets.
\newblock {\em Mathematical Finance}, 31(4):1454--1493, 2021.

\bibitem{pham2022portfolio}
Huyen Pham, Xiaoli Wei, and Chao Zhou.
\newblock Portfolio diversification and model uncertainty: A robust dynamic
  mean-variance approach.
\newblock {\em Mathematical Finance}, 32(1):349--404, 2022.

\bibitem{pun2021g}
Chi~Seng Pun.
\newblock G-expected utility maximization with ambiguous equicorrelation.
\newblock {\em Quantitative Finance}, 21(3):403--419, 2021.

\bibitem{rust1994structural}
John Rust.
\newblock Structural estimation of {M}arkov decision processes.
\newblock {\em Handbook of econometrics}, 4:3081--3143, 1994.

\bibitem{schal2002markov}
Manfred Sch{\"a}l.
\newblock {M}arkov decision processes in finance and dynamic options.
\newblock In {\em Handbook of {M}arkov decision processes}, pages 461--487.
  Springer, 2002.

\bibitem{shuvo2020markov}
Salman~Sadiq Shuvo, Yasin Yilmaz, Alan Bush, and Mark Hafen.
\newblock A {M}arkov decision process model for socio-economic systems impacted
  by climate change.
\newblock In {\em International Conference on Machine Learning}, pages
  8872--8883. PMLR, 2020.

\bibitem{srisuma2012semiparametric}
Sorawoot Srisuma and Oliver Linton.
\newblock Semiparametric estimation of {M}arkov decision processes with
  continuous state space.
\newblock {\em Journal of Econometrics}, 166(2):320--341, 2012.

\bibitem{tevzadze2013robust}
Revaz Tevzadze, Teimuraz Toronjadze, and Tamaz Uzunashvili.
\newblock Robust utility maximization for a diffusion market model with
  misspecified coefficients.
\newblock {\em Finance and Stochastics}, 17(3):535--563, 2013.

\bibitem{uugurlu2017controlled}
Kerem U{\u{g}}urlu.
\newblock Controlled {M}arkov decision processes with {AVaR} criteria for
  unbounded costs.
\newblock {\em Journal of Computational and Applied Mathematics}, 319:24--37,
  2017.

\bibitem{uugurlu2018robust}
Kerem U{\u{g}}urlu.
\newblock Robust optimal control using conditional risk mappings in infinite
  horizon.
\newblock {\em Journal of Computational and Applied Mathematics}, 344:275--287,
  2018.

\bibitem{villani2009optimal}
C{\'e}dric Villani.
\newblock {\em Optimal transport: old and new}, volume 338.
\newblock Springer, 2009.

\bibitem{white1993survey}
Douglas~J White.
\newblock A survey of applications of {M}arkov decision processes.
\newblock {\em Journal of the operational research society}, 44(11):1073--1096,
  1993.

\bibitem{xiong2018practical}
Zhuoran Xiong, Xiao-Yang Liu, Shan Zhong, Hongyang Yang, and Anwar Walid.
\newblock Practical deep reinforcement learning approach for stock trading.
\newblock {\em arXiv preprint arXiv:1811.07522}, 2018.

\bibitem{xu2010distributionally}
Huan Xu and Shie Mannor.
\newblock Distributionally robust {M}arkov decision processes.
\newblock {\em Mathematics of Operations Research}, 37(2):288--300, 2012.

\bibitem{yang2019constrained}
Zhou Yang, Gechun Liang, and Chao Zhou.
\newblock Constrained portfolio-consumption strategies with uncertain
  parameters and borrowing costs.
\newblock {\em Mathematics and Financial Economics}, 13(3):393--427, 2019.

\bibitem{yu2019model}
Pengqian Yu, Joon~Sern Lee, Ilya Kulyatin, Zekun Shi, and Sakyasingha Dasgupta.
\newblock Model-based deep reinforcement learning for dynamic portfolio
  optimization.
\newblock {\em arXiv preprint arXiv:1901.08740}, 2019.

\bibitem{yue2020linear}
Man-Chung Yue, Daniel Kuhn, and Wolfram Wiesemann.
\newblock On linear optimization over {W}asserstein balls.
\newblock {\em arXiv preprint arXiv:2004.07162}, 2020.

\bibitem{zhang2020deep}
Zihao Zhang, Stefan Zohren, and Stephen Roberts.
\newblock Deep reinforcement learning for trading.
\newblock {\em The Journal of Financial Data Science}, 2(2):25--40, 2020.

\end{thebibliography}
\end{document}